\documentclass[moor]{informs4post}
\usepackage{eqndefns-left} 
\RequirePackage{tgtermes}
\RequirePackage{newtxtext}
\RequirePackage{newtxmath}

\RequirePackage{bm}
\RequirePackage{endnotes}

\usepackage{setspace}

\SingleSpacedXI

\usepackage[sort&compress]{natbib}
 \bibpunct[, ]{[}{]}{,}{n}{}{,}%
 \def\bibfont{\small}%
\EquationsNumberedBySection 

\TheoremsNumberedBySection  
\ECRepeatTheorems  %

\MANUSCRIPTNO{}

\usepackage[english]{babel}
\usepackage[utf8]{inputenc} 
\usepackage[T1]{fontenc}    
\usepackage{url}            
\usepackage{booktabs}       
\usepackage{amsfonts}       
\usepackage{nicefrac}       
\usepackage{microtype}      
\usepackage[margin=1in]{geometry}
\usepackage[ruled,linesnumbered]{algorithm2e}
\usepackage{amssymb}
\usepackage{amsfonts}
\usepackage{amsmath}

\usepackage{epsfig}
\usepackage{subcaption}
\DeclareCaptionFont{singlespacing}{}

\theoremstyle{EX}
\newtheorem{mainresult}{Main Result}
\theoremstyle{TH}

\usepackage{adjustbox}
\PassOptionsToPackage{hypertexnames=false}{hyperref}
\usepackage{pdfcomment}
\usepackage{todonotes}
\usepackage{array}
\usepackage{wrapfig}
\usepackage{enumitem}
\usepackage{multirow}
\usepackage{diagbox}

\newcommand{\dist}{\operatorname{Dist}}

\newcommand{\gap}{\operatorname{Gap}}

\newcommand{\ep}{\operatorname{\textsf{E}}}

\newcommand{\econs}{\operatorname{E_{cons}}}
\newcommand{\egap}{\operatorname{E_{gap}}}

\newcommand{\condG}{\Gamma}
\newcommand{\condC}{\mathcal{C}}

\newcommand{\width}{\operatorname{Width}}

\usepackage{xcolor}

\newcommand{\eps}{\varepsilon}

\newcommand{\Diam}{\mathrm{Diam}}

\newcommand{\calF}{{\cal F}}

\newcommand{\calS}{{\cal S}}

\newcommand{\calW}{{\cal W}}
\newcommand{\calX}{{\cal X}}
\newcommand{\calY}{{\cal Y}}
\newcommand{\calZ}{{\cal Z}}

\newcommand{\x}{{\boldsymbol x}}

\begin{document}


\RUNAUTHOR{Xiong and Freund}

\RUNTITLE{Level-Set Geometry and the Theoretical Performance of PDHG for Conic Linear Optimization}

\TITLE{Level-Set Geometry and the Theoretical Performance of PDHG for Conic Linear Optimization}

\ARTICLEAUTHORS{%
\AUTHOR{Zikai Xiong}
\AFF{Department of Industrial Engineering and Management Sciences, Northwestern University, 2145 Sheridan Road, Evanston, IL 60208, USA. \href{mailto:zikai.xiong@northwestern.edu}{zikai.xiong@northwestern.edu}.}
\AUTHOR{Robert M. Freund}
\AFF{MIT Sloan School of Management, 77 Massachusetts Avenue, Cambridge, MA 02139, USA. \href{mailto:rfreund@mit.edu}{rfreund@mit.edu}.}
} 

\ABSTRACT{We consider solving  (convex) conic linear optimization problems, at the scale where matrix-factorization-free methods are attractive or necessary. The restarted primal-dual hybrid gradient method (rPDHG) -- with heuristic enhancements and GPU implementation -- has been very successful in solving huge-scale linear optimization problems (LPs). However, its application to more general conic convex optimization problems is not so well-studied. We analyze the theoretical performance of rPDHG for general (convex) conic linear optimization, and LP as a special case thereof. We show a relationship between the geometry of the primal-dual $\delta$-(sub-)level sets $\mathcal{W}_\delta$ and the convergence rate of rPDHG. Specifically, we prove a bound on the convergence rate of rPDHG that improves when there is a primal-dual (sub-)level set $\mathcal{W}_\delta$ for which (i) $\mathcal{W}_\delta$ is close to the optimal solution set in Hausdorff distance, and (ii) the ratio of the diameter to the ``conic radius'' of $\mathcal{W}_\delta$ is small. And in the special case of LP, the performance of rPDHG is bounded only by this ratio applied to the (sub-)level set corresponding to the best non-optimal extreme point. Depending on the problem instance, this ratio can take on extreme values and can result in excellent or poor performance of rPDHG both in theory and in practice.
}%



\MSCCLASS{90C05, 90C06, 90C25, 90C47}

\maketitle

\section{Introduction, and preview of main results}\label{sec:intro}

In this paper, we focus on the following general conic linear optimization problem (CLP):
\Equationvalidatefalse
\begin{equation}\tag{P}\label{pro: general primal clp}
	\min_{x\in\mathbb{R}^n}  \ c^\top x \quad 	\text{s.t.} \ Ax = b, \  x \in K_p \ ,\end{equation}
\Equationvalidatetrue where $K_p \subseteq \mathbb{R}^n$ is a closed convex cone, $A\in \mathbb{R}^{m\times n}$ is the constraint matrix,  $b \in \mathbb{R}^m$ is the right-hand side vector, and $c \in \mathbb{R}^n$ is the objective vector.
The family of CLPs has emerged since the 1990s as a fundamental problem class in convex optimization.  CLP includes standard linear optimization problems (LPs), second-order cone optimization problems (SOCPs), semidefinite optimization problems (SDPs), and exponential cone optimization problems as important subclasses. LPs are those instances of CLP for which $K_p$ is the nonnegative orthant $\mathbb{R}^n_+$; LP arises in application domains as varied as manufacturing \cite{bowman1956production,hanssmann1960linear}, transportation \cite{charnes1954stepping}, economics \cite{greene2003econometric}, computer science \cite{cormen2022introduction}, and medicine \cite{wagner2004large} among many others \cite{dantzig2002linear}. SOCPs are another subclass of CLP where $K_p$ is a cross-product of second-order cones $\mathbb{K}^{d+1}_{\textsc{soc}}$. They have significant applications in finance \cite{levy1970international,markowitz1950theories}, statistics \cite{tibshirani1996regression}, and others \cite{lobo1998applications}. The broad class of semidefinite optimization problems (SDPs), where $K_p$ is a cross-product of semidefinite cones $\mathbb{S}^{d \times d}_+$, has also received significant attention, though more for its overarching breadth of potential applications than for practical industrial use \cite{blekherman2012semidefinite,wolkowicz2012handbook,alizadeh1995interior}. Other CLP instances include exponential cone optimization problems, which model formulations involving exponential, logarithmic, and entropy-type functions \cite{eisenberg1959consensus,rujeerapaiboon2016robust}.
 
Algorithms for small and medium-size CLP instances have been extensively researched both theoretically and computationally.  Classic algorithms such as simplex/pivoting methods and interior-point methods (IPMs) form the foundation of modern optimization solvers and have had a profound impact on optimization quite broadly.  However, their success is premised on being able to repeatedly solve linear equation systems to high accuracy at each iteration, whose operations grow superlinearly with respect to the size of the data (measured with the dimensions $m$ and/or $n$ of the operators or the number of nonzero entries $\textsf{nnz}$ in the data $A$, $b$, $c$). This renders the methods impractical for some huge-scale instances. Furthermore, the associated matrix factorizations are not well suited to parallel or distributed computation. In contrast, first-order methods (FOMs) are emerging as an alternative approach for solving large-scale CLPs, since they require no or only very few matrix factorizations. Instead, the primary computational cost of FOMs lies in computing matrix-vector products that are needed to compute (or estimate) gradients. As such, FOMs are inherently more suitable for exploiting data sparsity, and furthermore are well suited for parallel and/or distributed computer architecture and modern graphics processing units (GPUs).

Among FOMs, PDHG and closely related primal-dual methods have become especially prominent. For LP, the restarted primal-dual hybrid gradient method (rPDHG) \cite{applegate2023faster} underlies the solver PDLP \cite{applegate2021practical}, admits infeasibility detection \cite{applegate2024infeasibility}, has benefited from recent GPU implementations \cite{lu2023cupdlp,lu2023cupdlp-c}, and has been embedded in commercial solvers as a base algorithm for LP \cite{coptgithub,xpressnwes}. Similar matrix-factorization-free methods have also appeared beyond LP, including PDHG-type methods for convex quadratic programming \cite{lu2023practicalqp,huang2024restarted} and the recent rPDHG-based solver for CLP, including second-order cone and exponential cone optimization problems \cite{lin2025pdcs}. It is also known that PDHG is equivalent to linearized ADMM or preconditioned ADMM under suitable formulations \cite{liu2021acceleration}, and ADMM-based methods have been widely used in practice for solving general CLP instances \cite{o2016conic,o2021operator}. These developments motivate studying PDHG-type methods from a viewpoint that is not restricted to LP.

\begin{wrapfigure}{hr}{0.35\textwidth}
	\centering
	\includegraphics[width=1\linewidth]{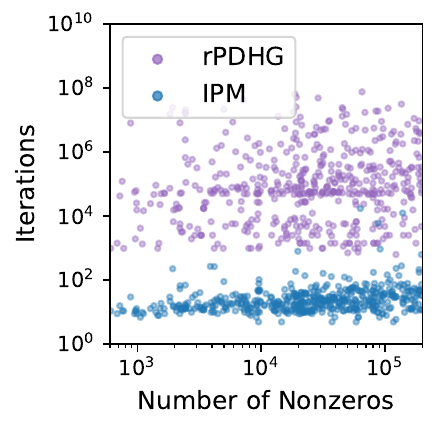}
	\caption{\small The number of iterations of rPDHG and IPM for solving LP instances with different numbers of nonzeros in the constraint matrix.}\label{fig:number_of_iteration_count}
\end{wrapfigure}

Despite the promising performance of rPDHG on many large-scale instances, the method can also perform poorly on certain instances -- even very small ones. Figure \ref{fig:number_of_iteration_count} illustrates this phenomenon on a set of LP examples: it shows the distribution of the number of iterations required by a standard rPDHG implementation and a standard IPM implementation for a set of LP instances taken from the MIPLIB 2017 dataset \cite{gleixner2021miplib}. While rPDHG, as a first-order method, enjoys a much lower per-iteration cost than IPMs, its iteration count can vary by orders of magnitude across instances, even when the instances have similar size as measured by the number of nonzeros $\textsf{nnz}$ and even when the instances are quite small. This begs the question of what instance-specific conditions cause some instances to be more difficult to solve?  The identification of these condition measures can result in new insights into the performance of rPDHG, and potentially lead to new algorithmic enhancements. Answers to this question lead to the study of traditional as well as novel condition measures for the general family of CLPs as well as more specifically to LP instances.

There has been some recent research focused on different condition measures to analyze the complexity of rPDHG on LP. \cite{applegate2023faster} shows that the linear convergence of rPDHG relies on the sharpness of a ``normalized duality gap'' and bounds this sharpness using a global Hoffman constant of the KKT system. However, the Hoffman constant is usually overly conservative, and is difficult to analyze, compute, or heuristically improve. \cite{xiong2023computational} further connects the sharpness constant to two natural and intuitive condition measures of LP, namely the ``limiting error ratio''  and LP sharpness.  However, both of these analyses ultimately rely on the positivity of the sharpness constant. This creates a limitation: for many CLP instances the relevant sharpness may be zero (which causes the theory to break down).  Furthermore, even in the case of LP instances (where the sharpness must be positive) the sharpness can be exponentially close to zero. In such cases, the practical performance of rPDHG can be much better than what sharpness-based theory would suggest. This discrepancy between theory and practical performance points to the need to find new condition measures (beyond sharpness) for which the theory can yield a better understanding of the performance of rPDHG. It also motivates the need for a unified geometric framework that can explain both small-sharpness LP instances and more general CLP instances where sharpness may fail entirely.

Based on the above discussion, this paper seeks to address the following questions: what are the condition measures  that impact the performance of rPDHG on general CLP instances (in theory and practice), and how do these condition measures enter the complexity
bounds for rPDHG?  

\subsection{Symmetric subspace form}\label{sec:intro_subspace}

Instead of directly studying the (primal) problem \eqref{pro: general primal clp}, we appeal to the primal--dual pair of problems written in the spaces of their cone variables. 
In concept, a primal--dual pair of conic problems can be written in the symmetric \textit{subspace form}
\Equationvalidatefalse
\begin{equation}\tag{P$_{\mathrm{sub}}$}\label{pro:preview_primal_subspace}
	\min_{x\in\mathbb{R}^n} \ \hat{s}^{\top}x
	\quad \text{s.t.} \quad x\in V_p := \hat{x}+\mathcal{L},\quad x\in K_p,
\end{equation}
\Equationvalidatetrue%
\Equationvalidatefalse
\begin{equation}\tag{D$_{\mathrm{sub}}$}\label{pro:preview_dual_subspace}
	\max_{s\in\mathbb{R}^n} \ -\hat{x}^{\top}s
	\quad \text{s.t.} \quad s\in V_d := \hat{s}+\mathcal{L}^{\bot},\quad s\in K_d,
\end{equation}
\Equationvalidatetrue%
where $\mathcal{L}$ is a linear subspace, $\mathcal{L}^{\bot}$ is its orthogonal complement, $K_d:=K_p^*$ is the dual cone of $K_p$, and $\hat x$, $\hat s$ are chosen so that $\hat{x}\in\mathcal{L}^{\bot}$ and $\hat{s}\in\mathcal{L}$.  We write $V_p:=\hat{x}+\mathcal{L}$ and $V_d:=\hat{s}+\mathcal{L}^{\bot}$ for the primal and dual affine spaces, whose associated linear subspaces are $\mathcal{L}$ and $\mathcal{L}^{\bot}$, respectively. Replacing $\hat{x}$ by another point in $V_p$ or replacing $\hat{s}$ by another point in $V_d$ preserves the corresponding affine spaces, and after accounting for the resulting constant shifts in the objective functions it preserves the feasible sets and optimal solution sets.  We opt for choosing $\hat x$, $\hat s$ so that $\hat{x}\in\mathcal{L}^{\bot}$ and $\hat{s}\in\mathcal{L}$ because this makes the form symmetric in the primal and dual format.

For the data representation \eqref{pro: general primal clp}, the subspace in \eqref{pro:preview_primal_subspace} is $\mathcal{L}=\operatorname{Null}(A)$, $\hat{x}$ can be chosen as $\hat x = A^\top(AA^\top)^\dag b$ (where $\cdot ^\dag$ denotes the Moore-Penrose inverse), and $\hat{s}$ can be chosen as $\hat s = P_{\mathcal{L}}(c)$ (where $P_{\mathcal{L}}(\cdot)$ denotes projection onto $\mathcal L$). Equivalently, $V_p=A^\top(AA^\top)^\dag b+\mathcal{L}$ and $V_d=P_{\mathcal{L}}(c)+\mathcal{L}^{\bot}$.  In Section~\ref{sec:get_general_clp_dual} we present the full derivation from the standard form primal and dual problems.  This subspace viewpoint is classical in conic optimization; see, for example, \cite{nesterov1994interior,renegar2001mathematical,ben2001lectures,nemirovski2005cone}. 

Let $w:=(x,s)$ denote the primal--dual cone variables.  In the subspace form, define
\begin{equation}\label{def_p_d_feasibleset}
	\calF:=V\cap K \ , \ \mathrm{where} \ \ 
	V:=V_p\times V_d \ \ \mathrm{and} \ \ 
	K:=K_p\times K_d \ .
\end{equation}
Thus $V$ is the primal--dual affine subspace, $K$ is the primal--dual cone,  and $\calF$ is the set of primal--dual feasible cone variables.  Figure~\ref{fig:sublevel_geometry} conceptually illustrates these three objects for an LP, where the cone is the nonnegative orthant.

\begin{figure}[t]
	\begin{subfigure}{0.32\textwidth}
		\includegraphics[width=\linewidth]{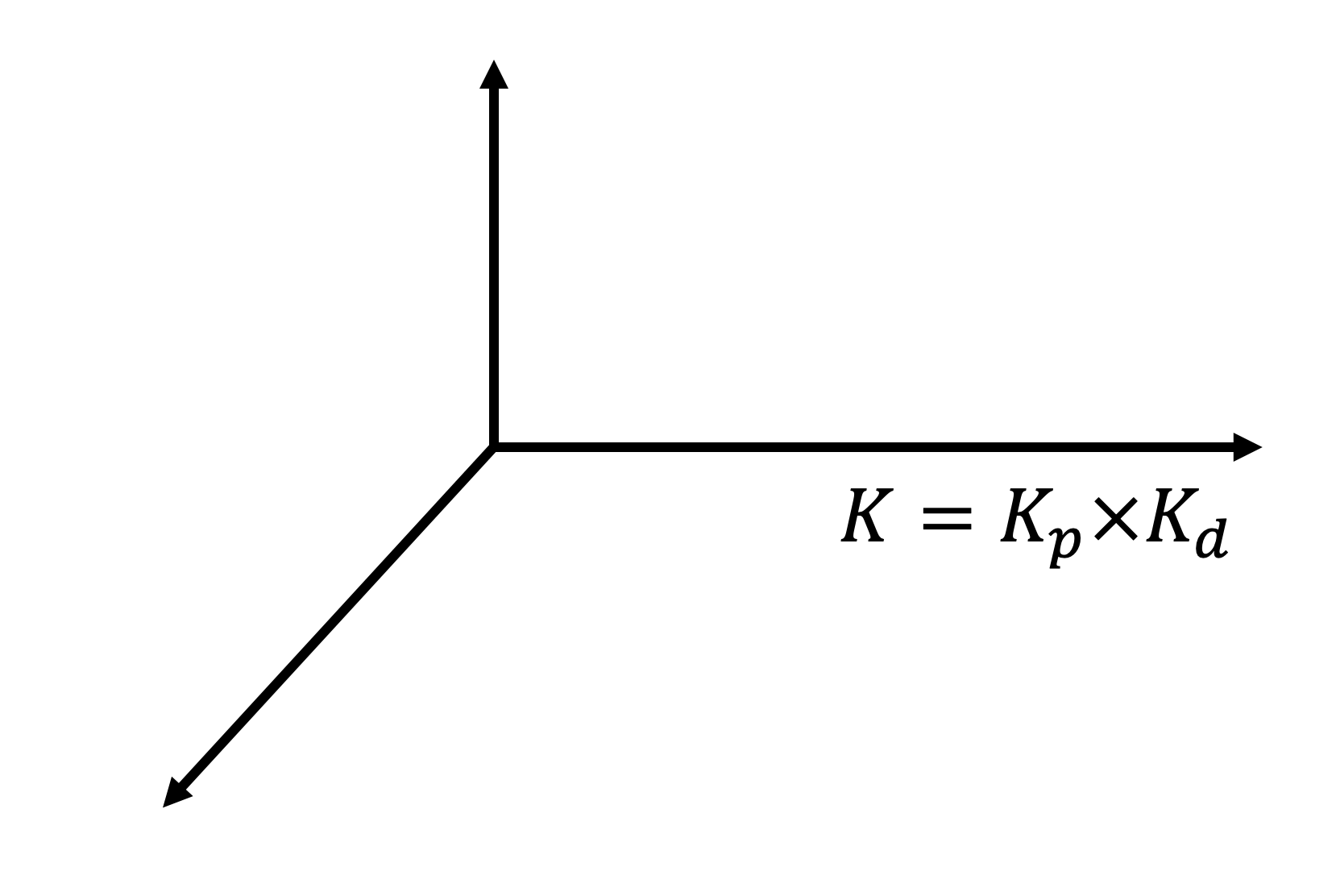}
		\caption{}\label{fig:pdcone}
	\end{subfigure}
	\hfill
	\begin{subfigure}{0.32\textwidth}
		\includegraphics[width=\linewidth]{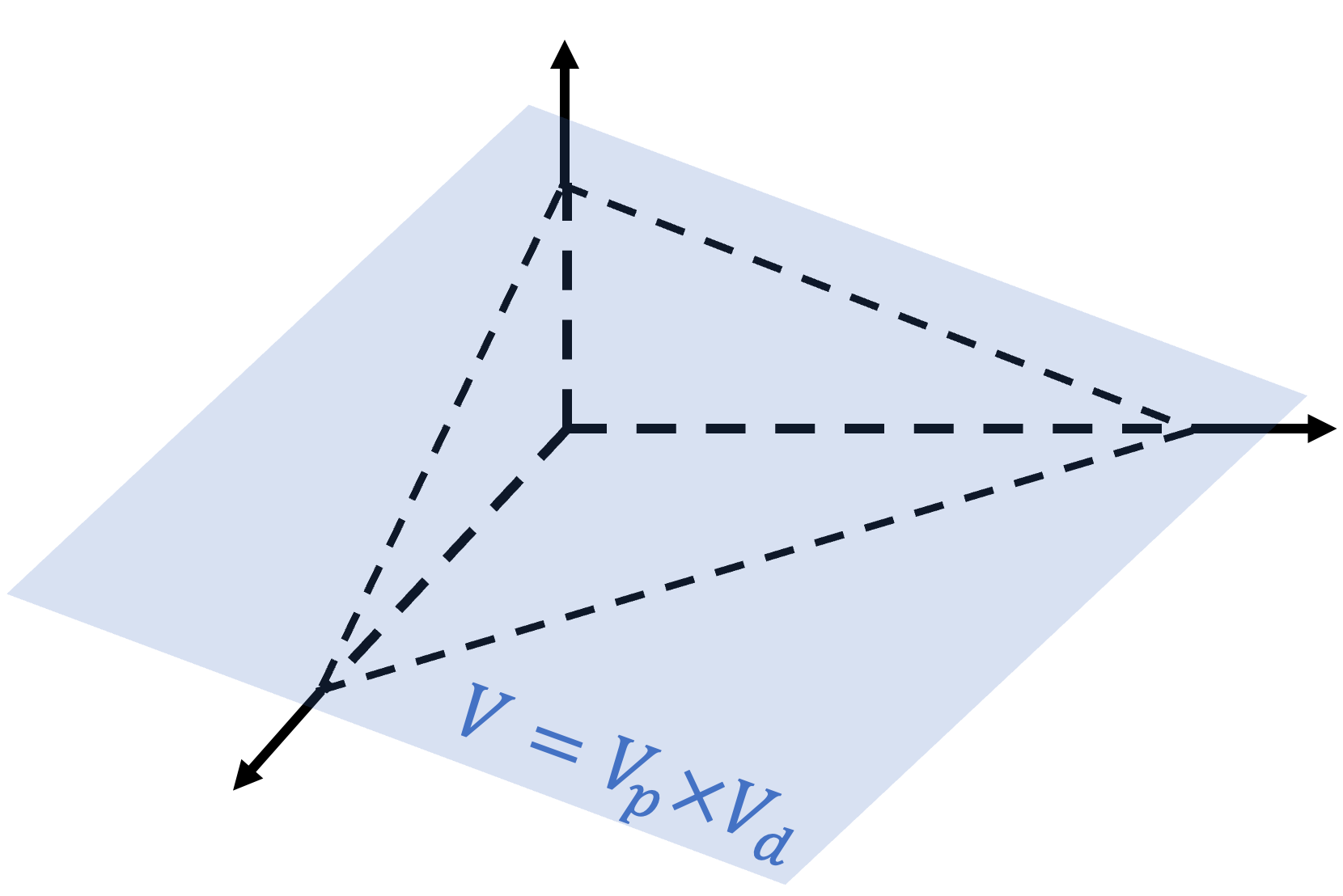}
		\caption{}\label{fig:pdaffine}
	\end{subfigure}
	\hfill
	\begin{subfigure}{0.32\textwidth}
		\includegraphics[width=\linewidth]{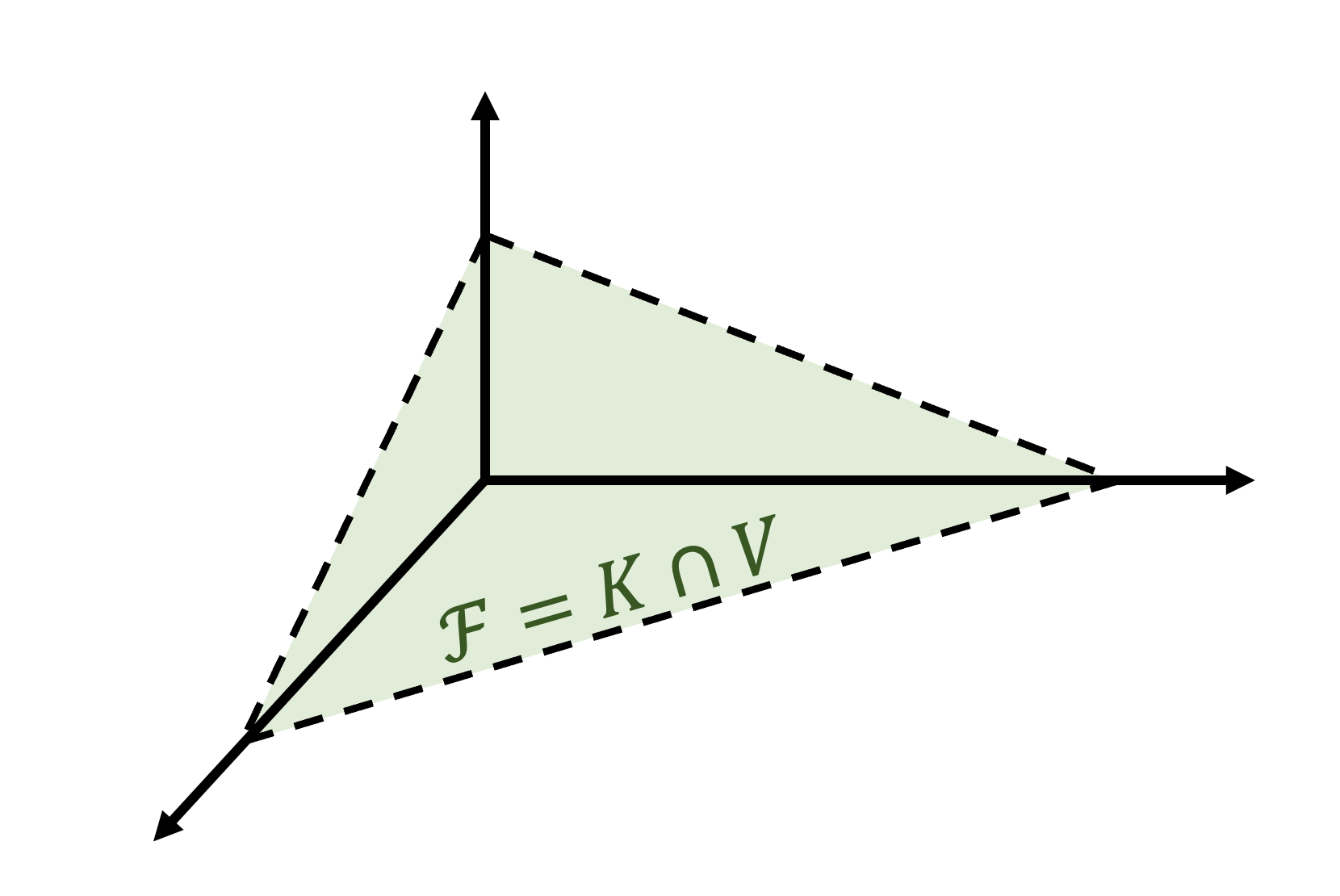}
		\caption{}\label{fig:pdfeasible}
	\end{subfigure} 
	\caption{Illustration of (a) the cone $K$, (b) the affine subspace $V$ (blue region) and (c) the feasible set $\mathcal{F}$ (green region) for an LP. For ease of visualization, these figures are shown in three dimensions even though the dimension of the cone variable pair $w=(x,s)$ is $2n$.}\label{fig:sublevel_geometry}
\end{figure} 

The duality gap of $w=(x,s)$ in the symmetric subspace form is denoted by $\gap(w)$ and defined as
\begin{equation}\label{def_duality_gap}
	\gap(w) :=\hat{s}^{\top}x+\hat{x}^{\top}s .
\end{equation}
We also use  $\gap(x,s)$ to denote $\gap(w)$. 
For feasible $w=(x,s)\in\calF$, this gap is also equal to $x^\top s$, since $x\in\hat{x}+\mathcal{L}$, $s\in\hat{s}+\mathcal{L}^{\bot}$, $\hat{x}\in\mathcal{L}^{\bot}$, and $\hat{s}\in\mathcal{L}$.  Hence $\gap(w)\ge0$ for feasible $w$ because $x\in K_p$ and $s\in K_d=K_p^*$.  Optimality is characterized by feasibility and zero duality gap.  We therefore define the primal--dual optimal solution set as
\begin{equation}\label{def_optimal_W}
	\calW^\star := \calF \cap \{w\in \mathbb{R}^{2n}: \gap(w)=0\} \ .
\end{equation}


\subsection{Condition measures of the primal-dual level sets}\label{sec:intro_sublevel-set-condition-number}

Our analysis in this paper heavily involves three condition measures related to the primal-dual level sets, and all three condition measures have a geometric flavor.  (Level sets are also sometimes called sublevel sets in the literature. In this paper we always call them level sets to be consistent.)
\begin{definition}{\bf ($\delta$-level set)}\label{def level set}
	For any $\delta \ge  0$, the $\delta$-level set is defined as
	\begin{equation}\label{eq delta level set}
		\calW_\delta :=   \calF \cap \{w\in\mathbb{R}^{2n}: \gap(w) \le \delta\} \ ,
	\end{equation}
	which is the set of feasible primal-dual solution $w:=(x,s)$ whose duality gap $\gap(w)$ is at most $\delta$.
\end{definition}
\noindent
Observe that when $\delta = 0$, then $\calW_0 = \calW^\star$, and for all $\delta \ge 0$ we have $\calW^\star \subseteq \calW_\delta$. See Figure \ref{fig:PDsublevel} for an illustration of the $\delta$-level set $\calW_\delta$ (the yellow region) for an LP instance with a unique optimum, where the cone is the nonnegative orthant and $\calW^\star = \{w^\star\}$.  (The figure builds on the instance shown in Figure \ref{fig:sublevel_geometry}.)
The points with duality gaps equal to $0$ and $\delta$ are denoted by the two red dashed lines in the figure.

\begin{figure}[t]
	\centering  
	\begin{subfigure}{0.38\textwidth}
		\includegraphics[width=\linewidth]{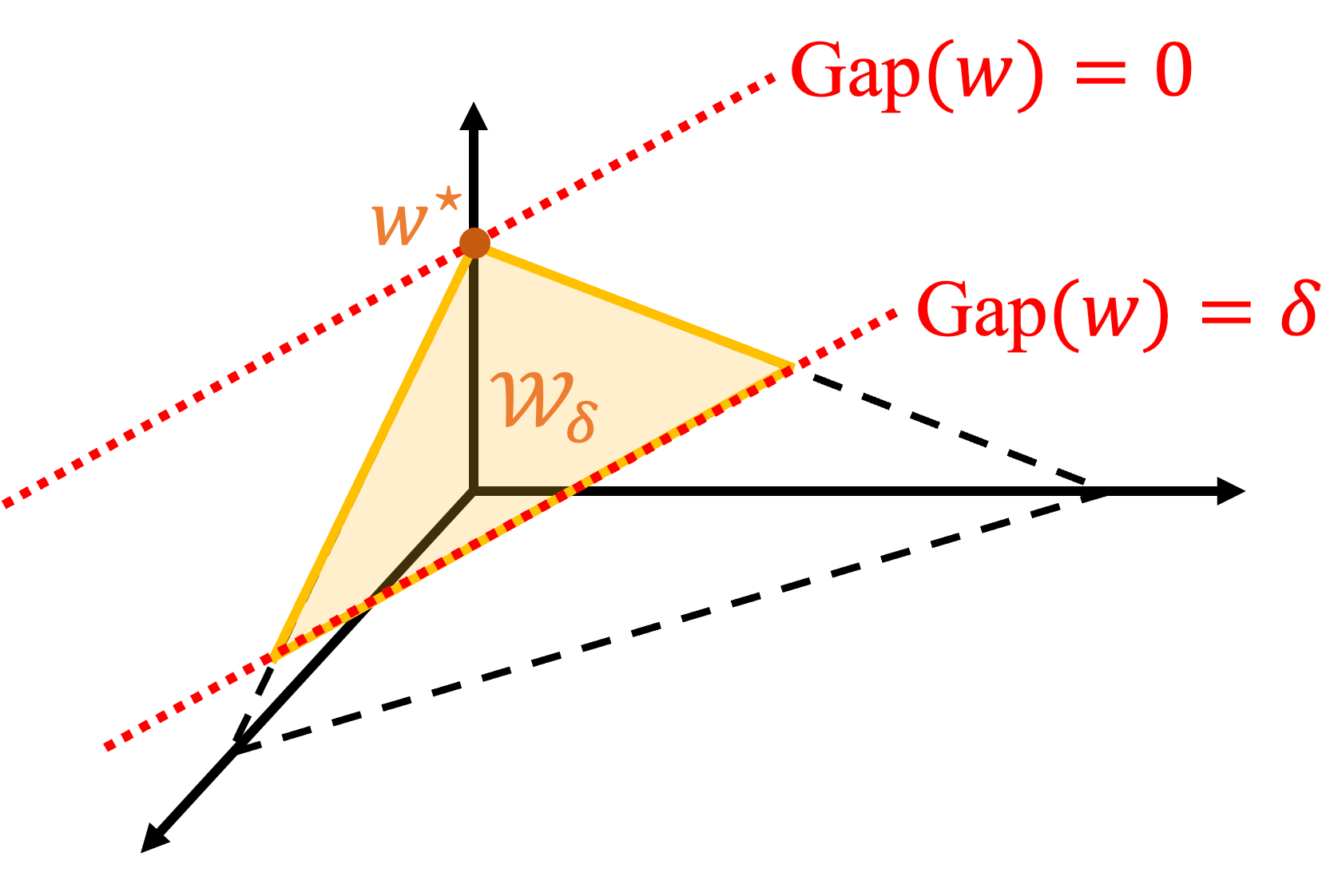}
		\caption{Level set $\calW_\delta$}\label{fig:PDsublevel}
	\end{subfigure} 
	\begin{subfigure}{0.38\textwidth}
		\includegraphics[width=\linewidth]{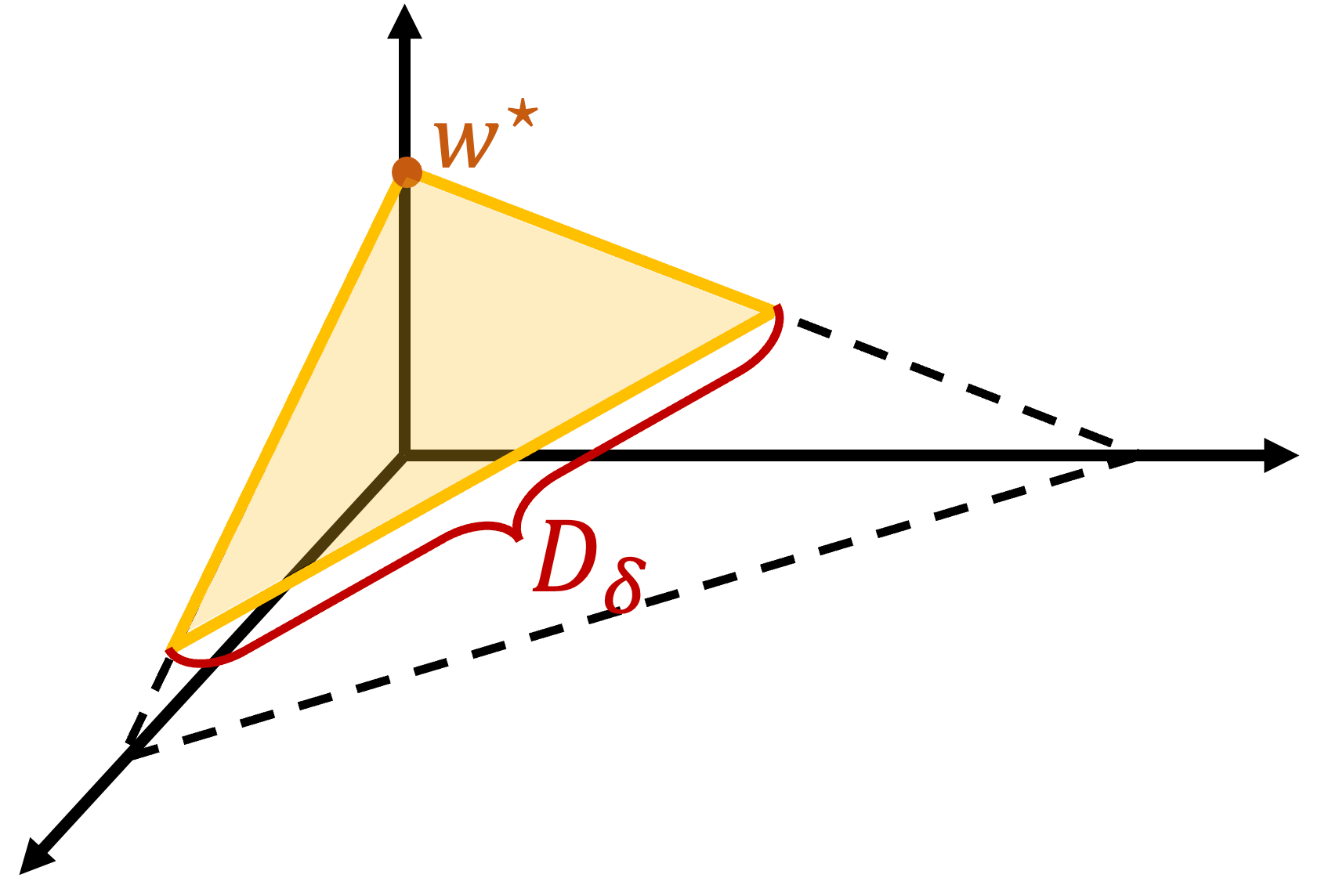}
		\caption{Diameter $D_\delta$}\label{fig:Ddelta}
	\end{subfigure} \hfill\\ 
	\begin{subfigure}{0.38\textwidth}
		\includegraphics[width=\linewidth]{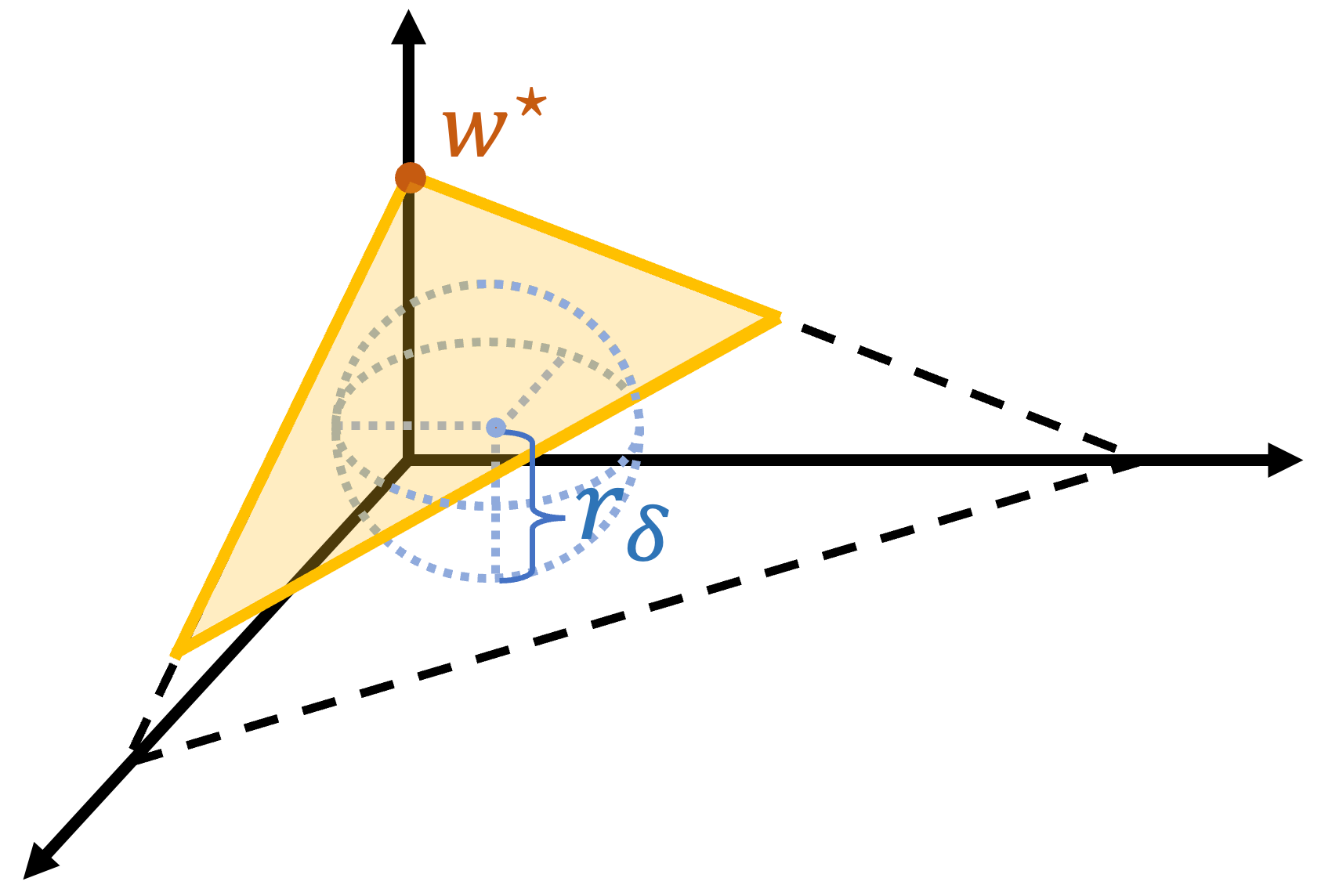}
		\caption{Conic radius $r_\delta$}\label{fig:rdelta}
	\end{subfigure} 
	\begin{subfigure}{0.38\textwidth}
		\includegraphics[width=\linewidth]{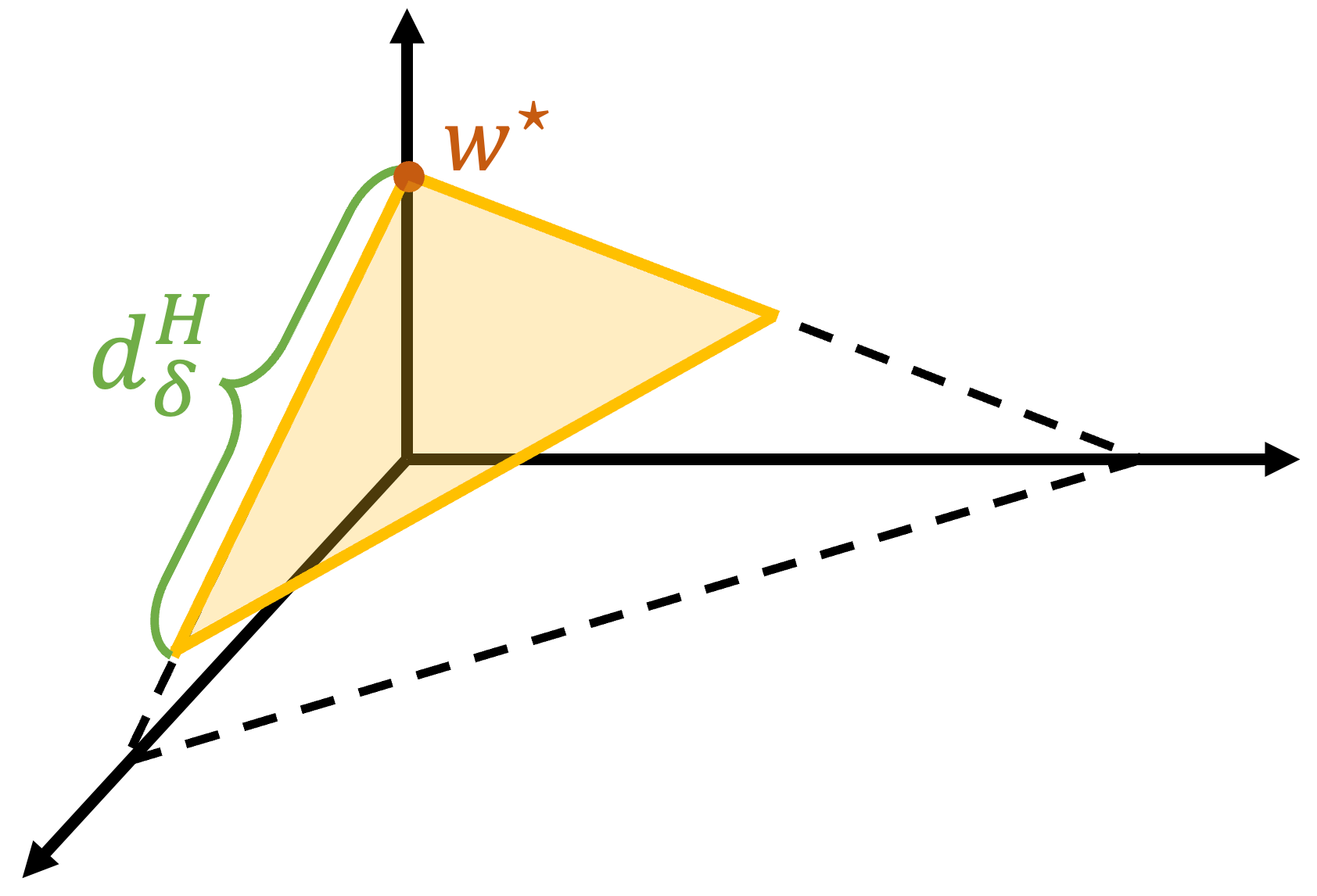}
		\caption{Hausdorff distance $d^H_\delta$}\label{fig:dHdelta}
	\end{subfigure}\hfill
	\caption{Illustration of the level set $\calW_\delta$, and the three geometric condition measures: the diameter $D_\delta$, the conic radius $r_\delta$, and the Hausdorff distance $d^H_\delta$. For visualization, the figures are shown in three dimensions even though the dimension of the cone variable pair $w=(x,s)$ is $2n$.}\label{fig:condition_measures}
\end{figure}

We now present the three condition measures we will use as the basis of our computational guarantees for rPDHG.\begin{definition}{\bf (Diameter of $\calW_\delta$)}\label{def diameter}
	For any $\delta$, the diameter of $\calW_\delta$ is:
	\begin{equation}\label{eq diameter}
		D_\delta := \max_{u,v \in \calW_\delta} \|u - v\| \ .
	\end{equation}
\end{definition}
\begin{definition}{\bf (Conic radius and conic center of $\calW_\delta$)}\label{def radius}
	For any $\delta > 0$, the conic radius $r_\delta$ of $\calW_\delta$ is the optimal value of the following problem:
	\begin{equation}\label{eq radius}
		\begin{aligned}
			(w_\delta,r_\delta) \in & \arg\max_{w\in\mathbb{R}^{2n},r\ge0}  \quad r                                                   \\
			                       & \ \ \ \ \ \ \quad  \operatorname{s.t.} \quad  \ \  w \in \calW_\delta, \ B(w,r)\subseteq K  \ ,
		\end{aligned}
	\end{equation}
\end{definition} \noindent and any optimal $w_\delta$ of this problem is called a conic center of $\calW_\delta$.  Thus  $w_\delta$ is a point of maximum distance from the boundary of $K$ among all points in $\calW_\delta$, and $r_\delta$ is the distance from $w_\delta$ to the boundary of $K$.

\begin{definition}{\bf (Hausdorff distance between $\calW_\delta$ and $\calW^\star$)}\label{def distance to optima}
	For any $\delta$, let $d^H_\delta$ denote the Hausdorff distance between $\calW_\delta$ and $\calW^\star$, namely
	\begin{equation}\label{eq distance}
		d^H_\delta := D^H(\calW_\delta, \calW^\star) = \max_{w\in\calW_\delta} \dist(w,\calW^\star) \ ,
	\end{equation}
	where $D^H(\cdot,\cdot)$ denotes the Hausdorff distance. (Note that the second equality above holds because $\calW^\star \subseteq \calW_\delta$.)
\end{definition}

Figure~\ref{fig:condition_measures} illustrates these quantities for an LP instance
with a unique optimal solution $w^\star$. The yellow region is the level set
$\calW_\delta$. The diameter $D_\delta$ is the largest distance between two points
in $\calW_\delta$. The conic radius $r_\delta$ is the radius of the largest Euclidean
ball centered at a point of $\calW_\delta$ and contained in the cone $K$. For the
nonnegative orthant, this is equivalently the largest possible value of the smallest
component of $w$ among points $w \in \calW_\delta$. The quantity $d_\delta^H$ is the maximum
distance from a point in $\calW_\delta$ to the optimal solution set $\calW^\star$.

\subsection{Preview of our main computational guarantee results}\label{sec:intro_preview}

We now preview our two main computational guarantee results.  The formal statements and proofs are presented later in Section~\ref{sec:complexity_clp} for general CLP instances, and are specialized to LP instances in Section~\ref{sec:complexity_lp}.  Our computational bounds are for the algorithm rPDHG (PDHG with restarts) which performs iterations of PDHG with occasional restarts, and is presented in Algorithm \ref{alg: PDHG with restarts} in Section \ref{sec:complexity_clp}.  Our computational bounds count the number of iterations of PDHG that are invoked in rPDHG.  (A detailed description of the basic PDHG iteration itself is presented in Section \ref{sec:pdhg_clp}.)  The computational cost of an iteration of PDHG is dominated by two matrix-vector products with $A$ and $A^\top$, respectively, plus a projection onto $K_p$. See Section~\ref{sec:pdhg_clp} for details. 

We define a primal-dual solution pair $w=(x,s)$ to be $\eps$-optimal if its distance to the constraints and its duality gap are both at most $\eps$, namely if $\max\{\dist(w,V),\dist(w,K)\}\le\eps$ and $\gap(w)\le\eps$.  In this preview we use $\tilde{O}(\cdot)$ to suppress absolute constants and logarithmic factors except for $\ln(1/\eps)$.  To keep the displayed bounds simple, we also suppose in this preview that $\dist(0,\calW^\star)\ge1$; otherwise one should replace $\dist(0,\calW^\star)$ below by $1+\dist(0,\calW^\star)$.

\begin{mainresult}[Less formal restatement of Theorem~\ref{thm overall complexity clp}]
\label{mainres:clp}
\emph{Under a standard primal-dual Slater condition and using standard step-size rules, rPDHG requires at most
\begin{equation}\label{eq:mainresult1}
	\widetilde{O}\left(
	\inf_{\delta>0}
	\left\{
	\frac{\kappa D_\delta}{r_\delta}\ln\left(\frac{1}{\eps}\right)
	+
	\frac{\kappa d^H_\delta \dist(0,\calW^\star)}{\eps}
	\right\}
	\right)
\end{equation}
iterations of PDHG to compute an $\eps$-optimal primal-dual solution pair for \eqref{pro:preview_primal_subspace} and \eqref{pro:preview_dual_subspace}. }
\end{mainresult}

Main Result~\ref{mainres:clp} is the central result of this paper, which is a computational guarantee for rPDHG for general conic linear optimization problems in terms of the three geometric condition measures $D_\delta$, $r_\delta$, and $d^H_\delta$ of the primal-dual level sets $\calW_\delta$, and the optimality tolerance $\eps$.  The bound \eqref{eq:mainresult1} also involves two other quantities, namely $\kappa$ which is the standard condition number of the matrix $A$, and $\dist(0,\calW^\star)$, which is the norm of the least-norm primal-dual optimal solution.  We note that $\dist(0,\calW^\star)$ arises naturally in the stability analysis of conic optimization under perturbation of the ``rim'' data of the primal/dual paired problems, see for example \cite{renegar1994some}.  

We now discuss and try to give some understanding of the bound \eqref{eq:mainresult1}.  The bound in the curly brackets $\{\cdot\}$ of \eqref{eq:mainresult1} holds for any chosen value of $\delta>0$, so let us choose such a $\delta$ and keep it fixed for now.  The first term of the bound depends logarithmically on $1/\eps$ and is controlled by $\kappa D_\delta/r_\delta$, while the second term is proportional to $1/\eps$ and is controlled by $\kappa d^H_\delta\dist(0,\calW^\star)$.  Examining the first term, the ratio $D_\delta/r_\delta$ is a measure of the shape of $\calW_\delta$: a smaller ratio means that the level set $\calW_\delta$ has a relatively large conic radius compared with its diameter, and hence is geometrically more favorable.  In the second term, the quantity $d^H_\delta$ measures how far the level set $\calW_\delta$ lies from $\calW^\star$.  Thus the bound says that rPDHG has a better guarantee if the level set $\calW_\delta$ is geometrically well-conditioned, in the sense of having a small ratio $D_\delta/r_\delta$, and being close to the optimal solution set, in the sense of having a small value of $d^H_\delta$.

Of course the bound \eqref{eq:mainresult1} is the \underline{infimum} of the quantity in the curly brackets $\{ \cdot \}$ of \eqref{eq:mainresult1} over all choices of $\delta >0$, i.e., over all level sets $\calW_\delta$, and this offers the opportunity to choose the value of $\delta$ that yields the best value of the curly-bracketed term.  Thus the computational guarantee is better if there exists any such level set $\calW_\delta$ with favorable geometry.  

It is also instructive in considering the bound \eqref{eq:mainresult1} to look at the limiting values of the geometric measures as $\delta$ decreases to $0$.  It is simple to see that $d^H_\delta \searrow 0$ as $\delta \searrow 0$ so the effect of the second term in \eqref{eq:mainresult1} disappears as $\delta \searrow 0$.  But also $r_\delta \searrow 0$ as $\delta \searrow 0$, so the first term in \eqref{eq:mainresult1} may blow up as $\delta \searrow 0$ depending on the behavior of $D_\delta$.  When there are multiple optima (either in the primal or the dual) then $D_\delta$ is bounded away from zero for all $\delta \ge 0$ since $D_\delta \ge D_0 = \max_{u,v \in \calW^\star} \|u - v\| > 0$.  In this case $D_\delta/r_\delta \rightarrow +\infty$ as $\delta \searrow 0$, and the sublinear convergence rate in the second term is likely to be the dominant term in the computational bound.  Indeed, even when the optimum is unique, it is still possible that $D_\delta/r_\delta \rightarrow +\infty$ as $\delta \searrow 0$, in which case the second term in the bound can be the dominant term.  Even in these cases where the first term blows up as $\delta$ goes to $0$, there may well be choices $\delta$ for which $\calW_\delta$ has favorable geometry and which yield reasonably good values of the overall bound in \eqref{eq:mainresult1}.

As alluded to above, when the level sets $\calW_\delta$ remain geometrically well-conditioned in the sense that $D_\delta/r_\delta$ is bounded as $\delta \searrow 0$, then Main Result~\ref{mainres:clp} immediately yields a linear convergence guarantee.

\begin{corollary}[Preview of Corollary~\ref{cor: linear convergence complexity clp}]\label{cor:intro_linear_convergence}
Under the setting of Main Result~\ref{mainres:clp},  rPDHG requires at most $
	\widetilde{O}\left(\kappa \left(\liminf_{\delta\searrow0}\frac{D_\delta}{r_\delta}\right)\ln\left(\frac{1}{\eps}\right)\right)$
iterations of PDHG to compute an $\eps$-optimal primal-dual solution pair for \eqref{pro:preview_primal_subspace} and \eqref{pro:preview_dual_subspace}.
\end{corollary}
 
\noindent Xiong \cite{xiong2024accessible} has reported experimental evidence that this bound is consistent with computational practice.

The level-set geometry also recovers the known linear convergence behavior of rPDHG for linear optimization. For LP instances, the polyhedral structure allows us to refine the linear convergence guarantee.  Let $\bar{\delta}$ denote the best suboptimal extreme point gap, namely the smallest positive value of $\gap(w)$ over extreme points of $\calF$ that are not optimal.

\begin{mainresult}[Less formal restatement of Theorem~\ref{thm overall complexity lp}]
\label{mainres:lp}
For LP instances (namely, $K_p=\mathbb{R}^n_+$), rPDHG requires at most
\begin{equation}\label{eq:mainresult2}
	\widetilde{O}\left(
	\kappa
	\left(\inf_{0<\delta\le\bar{\delta}}\frac{D_\delta}{r_\delta}\right)
	\ln\left(\frac{1}{\eps}\right)
	\right)
\end{equation}
iterations of PDHG to compute an $\eps$-optimal primal-dual solution pair for \eqref{pro:preview_primal_subspace} and \eqref{pro:preview_dual_subspace}.
\end{mainresult}

Main Result~\ref{mainres:lp} recovers the global linear convergence of rPDHG for LP; it follows as an application of the key machinery in the analysis and proof of Main Result~\ref{mainres:clp}.  Compared with Main Result~\ref{mainres:clp}, the LP guarantee eliminates the sublinear term involving $1/\eps$ from \eqref{eq:mainresult1}.  However, it does not mean that Main Result~\ref{mainres:lp} is always the better quantitative characterization of rPDHG, even for LP instances. The reason for this is that when $\bar{\delta}$ is very close to $0$, the corresponding values of $D_\delta/r_\delta$ might be very large (due to highly unfavorable geometry very near the optimal solution) and using such local geometry to characterize the global convergence behavior can be overly pessimistic.  By contrast, Main Result~\ref{mainres:clp} allows a more flexible tradeoff: the best value of $\delta$ in the infimum in \eqref{eq:mainresult1} may be quite a bit larger than $\bar \delta$. We will return to this point later in Section \ref{subsec:illustrative_lp_instance} where we will present a family of examples where the LP linear-convergence bound in Main Result~\ref{mainres:lp} is worse than the more general level-set tradeoff in Main Result~\ref{mainres:clp}.



\subsection{Other related work and follow-up work}
In addition to the research papers discussed earlier, several other works have also analyzed the performance of rPDHG and its variants for solving LP. \cite{hinder2023worst} presents a worst-case complexity of rPDHG on totally-unimodular LP instances, that does not rely on any condition measures.  \cite{lu2023geometry} develops a two-phase theory of the behavior of PDHG without restarts, where the initial sublinear convergence phase is followed by a linear convergence phase that is characterized by the Hoffman constant of a reduced system, and \cite{lu2022infimal} shows that the last iterate of PDHG without restarts also has a linear convergence rate but it is slower than that of rPDHG.

Several other first-order methods have been studied for LP and general CLP.  \cite{lin2021admm} proposes an ADMM-based interior-point method that leverages the framework of the homogeneous self-dual interior-point method and employs ADMM to solve the inner log-barrier problems. Further enhancements and extensions to CLP were subsequently developed by \cite{deng2024enhanced}. \cite{o2016conic,o2021operator} use ADMM to directly solve the homogeneous self-dual formulation for the general CLP, and \cite{lu2023practicalqp,huang2024restarted} study accelerated variants of rPDHG for solving a convex quadratic programming problem, which itself is a special case of SOCP and hence also of CLP.

\subsection{Notation}
Throughout this paper, we use the following notation for the most common cone examples: $\mathbb{R}^n_+$ denotes the nonnegative orthant,  $\mathbb{S}^{d\times d}_+$ denotes the semidefinite cone, which is the set of positive semidefinite symmetric matrices in $\mathbb{R}^{d\times d}$, $\mathbb{K}^{d+1}_{\mathsf{soc}}$ denotes the second-order cone $\{(x,t):x\in \mathbb{R}^d, t\in \mathbb{R}, \|x\|_2 \le t\}$, and $\mathbb{K}_{\mathsf{exp}}:=\operatorname{cl}\{(r,s,t)\in\mathbb{R}^3:s>0,\ s\exp(r/s)\le t\}$ denotes the exponential cone.

For a matrix $A\in\mathbb{R}^{m\times n}$, $\operatorname{Null}(A):=\{x\in\mathbb{R}^n:Ax = 0\}$ denotes the null space of $A$ and $\operatorname{Im}(A) :=\{Ax:x\in\mathbb{R}^n\}$ denotes the image of $A$. For any set $\calX\subset \mathbb{R}^n$,  $P_\calX: \mathbb{R}^n \to \mathbb{R}^n$ denotes the Euclidean projection onto $\calX$, namely, $P_\calX(x) := \arg\min_{\hat{x}\in \calX} \|x - \hat{x}\|$.
If not specified via definition, $\|\cdot\|$ in this paper denotes the Euclidean norm. The Moore-Penrose inverse of a square matrix $M$ is denoted by $M^\dag$. Let $B(x,r)$ denote the ball centered at $x$ with radius $r$. For any $M \in \mathbb{S}_{+}^{n\times n}$, $\|\cdot\|_M$ denotes the inner product ``norm'' induced by $M$, namely, $\|z\|_M :=\sqrt{z^\top Mz}$. (Here we allow $M$ to not be strictly positive definite, in which case $\|z\|_M$ is a semi-norm but not necessarily a norm.) For any $x \in \mathbb{R}^n$ and set $\calX\subset \mathbb{R}^n$, the Euclidean distance between $x$ and $\calX$ is denoted by $\dist(x,\calX):= \min_{\hat{x} \in \calX} \|x-\hat{x}\|$ and the $M$-norm distance between $x$ and $\calX$ is denoted by $\dist_M(x,\calX):= \min_{\hat{x} \in \calX} \|x-\hat{x}\|_M$.
For $A \in \mathbb{R}^{m\times n}$, we use $\sigma_{\max}^+(A)$ and $\sigma_{\min}^+(A)$ to denote the largest and smallest positive singular values of $A$, respectively.
For $x\in\mathbb{R}^n$, we use $x^+$ to denote the positive part of $x$. For any set $\calX \subset \mathbb{R}^n$, $\textsf{int}\calX$ denotes the interior of $\calX$. For the primal and dual affine spaces $V_p$ and $V_d$ introduced in Section~\ref{sec:intro_subspace}, $\mathcal{L}$ and $\mathcal{L}^{\bot}$ denote their associated linear subspaces, respectively. For any linear subspace $\mathcal{S}$ in $\mathbb{R}^n$, we use $\mathcal{S}^{\bot}$ to denote the corresponding orthogonal complement of $\mathcal{S}$. For any cone $K$, we use $K^*$ to denote the corresponding dual cone of $K$. For scalars $a$ and $b$ we denote $a \vee b := \max\{a,b\}$.

\section{Preliminaries: conic linear optimization problem, and PDHG}\label{sec:CLP}

\subsection{Symmetric subspace format for general conic linear optimization problem}\label{sec:get_general_clp_dual}

We now review how the data representation \eqref{pro: general primal clp} is converted (conceptually and/or computationally) to the symmetric subspace form introduced in Section~\ref{sec:intro_subspace}, which is standard in conic optimization, see for example \cite{ben2001lectures,nemirovski2005cone}.
We do not assume that $A$ has linearly independent rows, since eliminating linear dependence can be expensive for very large-scale instances.  
Let $y\in\mathbb{R}^m$ be the multiplier on the equation $Ax=b$. The saddlepoint formulation of \eqref{pro: general primal clp} is
\Equationvalidatefalse
\begin{equation}\tag{PD}\label{pro: general saddlepoint clp}
	\min_{x \in K_p\subseteq\mathbb{R}^n} \ \max_{y\in\mathbb{R}^m} \ L(x,y):= c^\top x  + b^\top y - x^\top A^\top y \ .
\end{equation}
\Equationvalidatetrue%
We obtain the dual problem by switching the order of the minimum and maximum, and introducing the slack variable $s=c-A^\top y$ yields the following form of the dual problem:
\Equationvalidatefalse
\begin{equation}\tag{D$_{y,s}$}\label{pro: general dual clp}
	\max_{y\in \mathbb{R}^m,\,s\in\mathbb{R}^n}  \ b^\top y \quad
	\text {s.t.} \quad c-A^\top y=s,\quad s\in K_d \ ,
\end{equation}
\Equationvalidatetrue%
where $K_d:=K_p^*:=\{s\in\mathbb{R}^n:x^\top s\ge 0 \text{ for all }x\in K_p\}$ is the dual cone of $K_p$.

The variable $s$ is the dual cone variable, and it is useful to express the dual objective directly in terms of $s$. Define $q:=A^\top(AA^\top)^\dag b$ and $q_0:=q^\top c=b^\top(AA^\top)^\dag A c$. Since primal feasibility implies $b\in\operatorname{Im}(A)$, we have $Aq=b$. Therefore, for any feasible pair $(y,s)$ of \eqref{pro: general dual clp} it holds that
\begin{equation}\label{eq by=-qs}
	b^\top y=q^\top A^\top y=q^\top(c-s)=q_0-q^\top s \ .
\end{equation}
This shows two things at once: all dual multipliers with the same slack variable $s$ have the same objective value, and the dual problem can be written solely in terms of the slack variable as
\Equationvalidatefalse
\begin{equation}\tag{D$_s$}\label{pro: general dual clp on s}
	\max_{s\in\mathbb{R}^n}  \ -q^\top s+q_0
	\quad \text{s.t.} \quad
	s\in c+\operatorname{Im}(A^\top),\quad s\in K_d \ .
\end{equation}
\Equationvalidatetrue%
Conversely, if $\hat{s}$ is feasible for \eqref{pro: general dual clp on s}, then $c-\hat{s}\in\operatorname{Im}(A^\top)$, and $\hat{y}:=(AA^\top)^\dag A(c-\hat{s})$ satisfies $c-A^\top\hat{y}=\hat{s}$. Thus \eqref{pro: general dual clp} and \eqref{pro: general dual clp on s} are equivalent for the purposes of feasibility and objective value. The quantities $q$ and $(AA^\top)^\dag$ are introduced for analysis and notation only; they do not need to be computed in order to run PDHG.

Recall from Section~\ref{sec:intro_subspace} that the objective vector of the primal problem in the symmetric subspace form is defined modulo the orthogonal linear subspace.  In the data representation above, we have $\mathcal L=\operatorname{Null}(A)$ and $\mathcal L^\bot=\operatorname{Im}(A^\top)$, and hence we may replace $c$ by its orthogonal projection onto $\mathcal L$.  Specifically, let $u:=(AA^\top)^\dag A c$ and $\bar c:=c-A^\top u=P_{\mathcal L}(c)$.  This replacement is without loss of generality: for every primal feasible $x$ we have $\bar c^\top x=c^\top x-u^\top b$, so the primal objective is shifted by the constant $u^\top b$, and $A^\top y+s=c$ is equivalent to $A^\top(y-u)+s=\bar c$, so the dual slack feasible set is unchanged.  The corresponding feasible and optimal multiplier sets are translated by $-u$.  Therefore $\calF$, $\calW^\star$, and the level-set condition measures introduced in Section \ref{sec:intro} are unchanged.  With this choice we have $\bar c\in\mathcal L$, $q_0=q^\top\bar c=0$, and the gap becomes $\gap(x,s)=\bar c^\top x+q^\top s$.  For simplicity, throughout the analysis we simply replace $\bar c$ by $c$.


With this convention in place, the affine spaces in the symmetric subspace form are
\begin{equation}\label{eq:data_subspace_representation}
	V_p=\{x\in\mathbb{R}^n:Ax=b\}=q+\operatorname{Null}(A),
	\quad
	V_d=\{s\in\mathbb{R}^n:c-A^\top y=s \text{ for some }y\}=c+\operatorname{Im}(A^\top).
\end{equation}
Thus $\mathcal{L}=\operatorname{Null}(A)$ is the linear subspace associated with $V_p$, while $\mathcal{L}^{\bot}=\operatorname{Im}(A^\top)$ is the linear subspace associated with $V_d$. Consequently, the data representation \eqref{pro: general primal clp} realizes the abstract pair \eqref{pro:preview_primal_subspace}--\eqref{pro:preview_dual_subspace} with $\hat{x}=q\in\mathcal{L}^{\bot}$ and $\hat{s}=c\in\mathcal{L}$. In particular, $V=V_p\times V_d$, $K=K_p\times K_d$, and $\calF=V\cap K$, as introduced in Section~\ref{sec:intro_subspace}. The primal objective value is $c^\top x$, the slack-variable dual objective value is $-q^\top s$, and the duality gap is $\gap(x,s)=c^\top x+q^\top s$, which matches the symmetric gap formula in \eqref{def_duality_gap}.

In most of our results we will make the following regularity assumption.

\begin{assumption}\label{assump:striclyfeasible}
	There exists a primal feasible solution in the interior of $K_p$ and a dual feasible solution in the interior of $K_d$.
\end{assumption}
\noindent
Assumption~\ref{assump:striclyfeasible} is the primal-dual
Slater condition for conic linear optimization, see for example
\cite{nesterov1994interior,renegar2001mathematical,YuTodd97,nesterov1998primal}. 
This assumption is not required for strong duality in LP, nor is it required for running first-order methods for LP \cite{applegate2023faster}; however, strong duality for CLP can fail in the absence of Assumption \ref{assump:striclyfeasible}. Also, Assumption \ref{assump:striclyfeasible} is not a prerequisite for running rPDHG; however, it guarantees finite values of the geometric measures $D_\delta$, $1/r_\delta$, and $d^H_\delta$ of the primal-dual level sets $\calW_\delta$, see \cite{freund2003primal}.

Finally, we use $\calX^\star$ and $\calS^\star$ to denote the optimal primal and dual slack-variable solution sets, respectively, so that $\calW^\star=\calX^\star\times\calS^\star$. We define the corresponding set of optimal dual multipliers and the saddlepoint solution set by
	$\calY^\star:=\{y\in\mathbb{R}^m:c-A^\top y\in\calS^\star\}$ and $\calZ^\star:=\calX^\star\times\calY^\star$.
It holds that $\calS^\star=c-A^\top\calY^\star$.

\subsection{PDHG for CLP}\label{sec:pdhg_clp}
The primal-dual hybrid gradient method (PDHG) was introduced in \cite{esser2010general,pock2009algorithm} in the context of solving general convex-concave saddlepoint problems, of which the saddlepoint problem \eqref{pro: general saddlepoint clp} is a class of instances.
Algorithm \ref{alg: one PDHG} describes a single iteration of PDHG for  \eqref{pro: general saddlepoint clp}, which we denote as \textsc{OnePDHG}$(x,y)$, where $\tau >0 $ and $\sigma >0$ are the primal and dual step-sizes, respectively.
\begin{algorithm}[htbp]
	\SetAlgoLined
	\SetKwProg{Fn}{Function}{}{}
	\Fn{\textsc{OnePDHG}$(x,y)$}{
		$x^{+} \leftarrow P_{K_p}\left(x-\tau\left(c-A^{\top} y\right)\right) $ \;\label{line:update_x}
		$y^{+} \leftarrow y+\sigma\left(b-A\left(2 x^{+}-x\right)\right)$ \;\label{line:update_y}
		return $z^+=(x^+,y^+)$ \;}
	\caption{One iteration of PDHG on $(x,y)$ for problem \eqref{pro: general saddlepoint clp}}\label{alg: one PDHG}
\end{algorithm}
Let $z:=(x,y)\in\mathbb{R}^{m+n}$ denote the combined primal/dual variables, and then PDHG generates iterates as follows:
$$
	z^{k+1} \leftarrow \textsc{OnePDHG}(z^k) \ \text{ for }k=0,1,2,\ldots .
$$
It should be noted that the convergence guarantees for PDHG rely on the step-sizes $\tau$ and $\sigma$ being sufficiently small. In particular, if the following condition is satisfied:
\begin{equation}\label{robsummer}
	M:=	\begin{pmatrix}
		\frac{1}{\tau}I_n & A^\top             \\
		A                & \frac{1}{\sigma}I_m
	\end{pmatrix}  \in  \mathbb{S}^{m+n}_+ \ , 
\end{equation}
then PDHG's average iterates will converge to a saddlepoint of the convex-concave problem \cite{chambolle2011first}. The above requirement is equivalently written as:
\begin{equation}\label{eq:general_stepsize}
	\tau > 0, \ \sigma >0, \ \text{ and } \ \tau\sigma \le  \left( \frac{1}{ \sigma_{\max}^+(A) } \right)^2 \ .
\end{equation}
Furthermore, the matrix $M$ defined in \eqref{robsummer} turns out to be particularly useful in analyzing the convergence of PDHG through its induced inner product norm defined by $\| z \|_M := \sqrt{z^\top M z}$, which will be used extensively in the rest of this paper. 

The main computational effort in executing \textsc{OnePDHG} is in computing the two matrix-vector products and computing the projection onto $K_p$. In practice, most CLP instances of interest are those where $K_p$ is a cross-product of standard cones, namely $\mathbb{R}^n_+$, $\mathbb{K}^{d+1}_{\textsc{soc}}$, $\mathbb{K}_{\mathsf{exp}}$, and $\mathbb{S}^{d \times d}_+$ \cite{nesterov1994interior}. These cones all have well-known projection operators \cite{parikh2014proximal,friberg2023projection}. Projection onto $\mathbb{R}^n_+$ is given by $
	P_{\mathbb{R}^n_+}(v):=v^+ $, and projection onto $\mathbb{K}^{d+1}_{\textsc{soc}}$ is given by:
$$
	P_{\mathbb{K}^{d+1}_{\mathsf{soc}}}(v, t) := \left\{\begin{array}{ll}0 & \text{ if }\|v\| \leq-t \\ (v, t) & \text{ if } \|v\| \leq t \\ \tfrac{1+t /\|v\|}{2}\cdot \left(v,\|v\|\right) & \text{ if } \|v\| \geq|t|\end{array} \right.\ .
$$
Projection onto an exponential cone can be computed by solving a univariate root-finding problem \cite{friberg2023projection,lin2025pdcs}. 
Among these cones, only projection onto the semidefinite cone may be computationally challenging because the classical exact projection usually involves a full matrix eigendecomposition.
Recent factorization-free approaches instead approximate the PSD-cone projection using polynomial filters evaluated through matrix multiplications, and these approximations can be implemented efficiently on GPUs \cite{kang2025factorizationfree}.

Furthermore, if $K_p$ is the cross-product of several cones, then each of these projections can be carried out independently (and in parallel) for each cone.

\subsection{Normalized duality gap for the saddlepoint problem \eqref{pro: general saddlepoint clp}}

To evaluate the quality of a candidate solution $z=(x,y)$, \cite{applegate2023faster} defined the ``normalized duality gap'' in the context of the saddlepoint formulation of LP. Here we simply extend this definition to CLP and we show that the normalized duality gap provides upper bounds on the residuals of the optimality conditions of CLP.
\begin{definition}[Normalized duality gap]
	Suppose that the step-sizes $\tau$ and $\sigma$ satisfy \eqref{eq:general_stepsize}, so that $M\succeq 0$ in \eqref{robsummer}. For any $z = (x,y)\in K_p \times \mathbb{R}^m$ and $r > 0$, define
	$$
		B(r;z) := \left\{\hat{z} := (\hat{x},\hat{y}):  \hat{x}\in K_p \text{ and } \|\hat{z} -z\|_M \le r  \right\} \ .
	$$
	The normalized duality gap of the saddlepoint problem \eqref{pro: general saddlepoint clp} is then defined as
	\begin{equation}\label{sunny}
		\rho(r;z) := \frac{1}{r}\sup_{  \hat{z} \in B(r;z) }  \big[ L(x,\hat{y}) - L(\hat{x},y) \big] \ .
	\end{equation}
\end{definition} 
\noindent The normalized duality gap defined in \cite{applegate2023faster} is just a special case of the above definition when $K_p = \mathbb{R}^n_+$. Lemma \ref{lm: convergence of PHDG without restart} below shows that the normalized duality gap yields upper bounds on the distances to the affine set $V$ and to the cone $K$ of the primal-dual cone variable pair $w=(x,s)$, and also bounds the duality gap $\gap(x,s)$. (Lemma 4 of \cite{applegate2023faster} showed that the normalized duality gap
bounds the KKT residuals for LP. The proof of Lemma
\ref{lm: convergence of PHDG without restart} follows the same basic argument
and extends the corresponding residual bounds to CLP.) Before stating the lemma we introduce the following definitions regarding the singular values and condition number of $A$:
\begin{equation}\label{eq  def lamdab min max}
	\lambda_{\max}:= \sigma_{\max}^+\left(A \right)\text{, }	\lambda_{\min}:= \sigma_{\min}^+\left(A \right), \text{ and }\kappa := \frac{\lambda_{\max}}{\lambda_{\min}}  \ .
\end{equation}

\begin{lemma}\label{lm: convergence of PHDG without restart}
	Suppose that $\tau$ and $\sigma$ satisfy \eqref{eq:general_stepsize}. For any $r > 0$ and $\bar{z} :=(\bar{x},\bar{y})$ such that $\bar{x} \in K_p$, and $\bar s :=c - A^\top \bar{y}$, the normalized duality gap $\rho(r;\bar{z})$ provides the following bounds for $\bar w :=(\bar x, \bar s)$:
	\begin{enumerate}
		\item Distance to the affine subspace: $ \dist(\bar{w},V)  \le \frac{1}{\sqrt{\sigma} \lambda_{\min}}\cdot \rho(r;\bar{z})$, \label{item_gap1}
		\item Distance to the cone: $ \dist(\bar{w},K)  \le \frac{1}{\sqrt{\tau}} \cdot \rho(r;\bar{z})$, and \label{item_gap2}
		\item Duality gap: $\gap(\bar{w})  \le   \max\{ r, \|\bar{z}\|_M\} \rho(r;\bar{z})$.\label{item_gap3}
	\end{enumerate}
\end{lemma}
\noindent The proof of Lemma \ref{lm: convergence of PHDG without restart} is given in Appendix~\ref{app:proof_lm_convergence_pdhg_without_restart}. It follows from Lemma \ref{lm: convergence of PHDG without restart} and the definition of the optimal solution set $\calW^\star$ in \eqref{def_optimal_W} that when the normalized duality gap is $0$, then $\bar{w} \in \calW^\star$. Moreover, if $\max\{ r, \|\bar{z}\|_M\}$ is not too large, the magnitude of $\rho(r;\bar{z})$ also measures how close to optimality the primal-dual solution $\bar{w}$ is. We will show later that under some mild initial point conditions, the magnitude of $\max\{ r, \|\bar{z}\|_M\}$ in PDHG is well-controlled by the distance to the optimal solution set.

For LP instances it is shown in \cite{applegate2023faster} that the normalized duality gap $\rho(r;z)$ can be easily computed or approximated.  In Appendix \ref{appendix:compute_rho} we extend this and show how to compute $\rho(r;z)$ for more general CLP instances.

\subsection{Sublinear convergence of PDHG for \eqref{pro: general saddlepoint clp}}

Let the $k$-th iterate of PDHG be denoted as $z^k$, and the average of the first $k$ iterates be denoted as $\bar{z}^k := \frac{1}{k}\sum_{i=1}^k z^i$. The iterates generated by PDHG satisfy the following desirable distance properties to the set of saddlepoints $\calZ^\star$, as stated in the following lemma.

\begin{lemma}{\bf (Nonexpansive property, essentially Proposition 2 of \cite{applegate2023faster})}\label{lm: nonexpansive property} Suppose that  $\sigma, \tau$ satisfy \eqref{eq:general_stepsize}. For any saddlepoint $z^\star$ of \eqref{pro: general saddlepoint clp}, and for all $k\ge0$, it holds that
	\begin{equation}\label{eq nonexpansive property}
		\| z^{k+1} -z^\star \|_M \le  \|z^k  - z^\star \|_M \ .
	\end{equation}
	Therefore under either the assignment $z:= z^k$ or $z := \bar z^k$ it holds that $
		\left\| z -z^\star\right\|_M \le \left\|z^0  - z^\star\right\|_M$.
\end{lemma}
\noindent Lemma \ref{lm: nonexpansive property} is essentially a restatement of Proposition 2 in \cite{applegate2023faster}. The inequality \eqref{eq nonexpansive property}, also known as the nonexpansive property, appears in many other operator splitting methods \cite{liang2016convergence,ryu2022large,applegate2023faster}. We also will make use of the following lemma, which is a restatement of Lemma 2.5 of \cite{xiong2023computational}.

\begin{lemma}{\bf (from Lemma 2.5 of \cite{xiong2023computational})}\label{lm: R in the opt gap convnergence} Suppose that $\tau$ and $\sigma$ satisfy \eqref{eq:general_stepsize}, and suppose $z^a$, $z^b$, and $z^c$ satisfy the nonexpansive properties:  $\|z^b - z^\star\|_M \le \| z^{a} - z^\star\|_M$ and $\|z^c - z^\star\|_M \le \| z^{a} - z^\star\|_M$ for every $z^\star \in \calZ^\star$. Then
	\begin{equation}\label{eq of lm: R in the opt gap convnergence}
		\begin{aligned}
			\max\{ \|z^b - z^c\|_M, \|z^b\|_M\}  \le & \ 2 \dist_M(z^a,\calZ^\star)  + \|z^{a} \|_M	\ .
		\end{aligned}
	\end{equation}
\end{lemma}

\noindent Using Lemma \ref{lm: nonexpansive property}, the sublinear convergence of the normalized duality gap has been demonstrated in \cite{applegate2023faster,xiong2023computational}, among others. Here we directly present a restatement of Lemma 2.2 of \cite{xiong2023computational}, which was initially developed for LP but in fact holds more broadly for conic optimization problems.
\begin{lemma}{\bf (Sublinear convergence of PDHG, from Lemma 2.2 of \cite{xiong2023computational})}\label{lm: original sublinear PDHG}
	Suppose that  $\sigma, \tau$ satisfy \eqref{eq:general_stepsize}. Then for any $z^0 := (x^0,y^0)$ with $x^0 \in K_p$, it holds for all $k \ge 1$ that
	\begin{equation}
		\rho(\|\bar{z}^k-z^0\|_M;\bar{z}^k) \le \frac{4\dist_M(z^0,\calZ^\star)}{k} \ .
	\end{equation}
\end{lemma}

\noindent Combining the results of Lemma \ref{lm: convergence of PHDG without restart} and Lemma \ref{lm: original sublinear PDHG}, we obtain the following corollary regarding sublinear convergence of PDHG for \eqref{pro: general saddlepoint clp}.
\begin{corollary}\label{thm covergence result}
	Suppose that  $\sigma, \tau$ satisfy \eqref{eq:general_stepsize}, and PDHG is initiated with $z^0 = (x^0,y^0)$, where $x^0\in K_p$. For all $k \ge 1$, let $\bar{s}^k := c - A^\top \bar{y}^k$. Then the following hold for $\bar{w}^k:= (\bar{x}^k,\bar{s}^k)$ for all $k \ge 1$:
	\begin{enumerate}
		\item Distance to the affine subspace: $ \dist(\bar{w}^k,V)  \le \frac{4}{\sqrt{\sigma} \lambda_{\min}}\cdot  \frac{\dist_M(z^0,\calZ^\star)}{k}$, \label{item_conv1}
		\item Distance to the cone: $ \dist(\bar{w}^k,K)  \le \frac{4}{\sqrt{\tau}} \cdot  \frac{\dist_M(z^0,\calZ^\star)}{k}$, and \label{item_conv2}
		\item Duality gap: $\gap(\bar{w}^k)  \le   \left(8\dist_M(z^0,\calZ^\star) + 4\|z^0\|_M\right)\cdot  \frac{\dist_M(z^0,\calZ^\star)}{k}$.\label{item_conv3}
	\end{enumerate}
\end{corollary}

\proof{Proof.}
	The upper bounds for the distances to the affine subspace $V$ and the cone $K$ follow directly from Lemma \ref{lm: convergence of PHDG without restart} and Lemma \ref{lm: original sublinear PDHG}.  To prove item (\textit{\ref{item_conv3}}.), we apply Lemma \ref{lm: R in the opt gap convnergence} with $z^a := z^0$, $z^b:=\bar{z}^k$ and $z^c:= z^0$, which then satisfy the nonexpansive properties of Lemma \ref{lm: R in the opt gap convnergence}, whereby it holds that
	$$
		\max\{ \|\bar{z}^k - z^0\|_M, \|\bar{z}^k\|_M\}  \le    2 \dist_M(z^0,\calZ^\star)  + \|z^0 \|_M \ .
	$$
	Then item (\textit{\ref{item_conv3}}.) of the corollary follows by applying the above inequality to item (\textit{\ref{item_gap3}}.) of Lemma \ref{lm: convergence of PHDG without restart} with $r = \|\bar{z}^k - z^0\|_M$. \Halmos\endproof

\section{Complexity of restarted-PDHG for CLP}\label{sec:complexity_clp}
In addition to the convergence analysis of PDHG, \cite{applegate2023faster,xiong2023computational} show that fixed-period and/or adaptive restarts lead to faster convergence of PDHG for LP, in both theory and practice. Algorithm \ref{alg: PDHG with restarts} describes our general restart scheme for PDHG for CLP, which is a generalization of Algorithm 1 in \cite{xiong2023computational} for LP. We refer to this algorithm as ``rPDHG'' for ``restarted-PDHG.'' 
\begin{algorithm}[htbp]
	\SetAlgoLined
	{\bf Input:} Initial iterate $z^{0,0}:=(x^{0,0}, y^{0,0})$, $n \gets 0$, and step-sizes $\tau,\sigma$ satisfying \eqref{eq:general_stepsize} \;
	\Repeat{\text{Either $z^{n,0}$ is a saddlepoint or $z^{n,0}$ satisfies some other convergence condition }}{
		\textbf{initialize the inner loop:} inner loop counter $k\gets 0$ \;
		\Repeat{ $\bar{z}^{n,k}$ satisfies some (verifiable) restart condition \ }{
			\textbf{conduct one step of PDHG: }$z^{n,k+1} \gets \textsc{OnePDHG}(z^{n,k})$ \;
			\textbf{compute the average iterate in the inner loop: }$\bar{z}^{n,k+1}\gets\frac{1}{k+1} \sum_{i=1}^{k+1} z^{n,i}$
			\label{line:average} \;  \label{line:output-is-average-of-iterates}
			$k\gets k+1$ \;
		}\label{line:restart_condition}
		\textbf{restart the outer loop:} $z^{n+1,0}\gets \bar{z}^{n,k}$, $n\gets n+1$ \;
	}
	{\bf Output:} $z^{n,0}$ ($ \ = (x^{n,0},  y^{n,0})$)
	\caption{rPDHG: restarted-PDHG}\label{alg: PDHG with restarts}
\end{algorithm}
Here $z^{k+1} \gets \textsc{OnePDHG}(z^k)$ is an iteration of PDHG as described in Algorithm \ref{alg: one PDHG}. For each iterate $z^{n,k} = (x^{n,k},y^{n,k})$, we define $s^{n,k}:= c - A^\top y^{n,k}$ and $\bar{s}^{n,k}:= c - A^\top \bar{y}^{n,k}$, and $\bar{s}^{n,k}$ denotes the average of the dual cone variable iterate values. The double superscript on the variables $z^{n,k}$, $s^{n,k}$, and $\bar{s}^{n,k}$ indexes the outer iteration counter followed by the inner iteration counter, so that $z^{n,k}$ is the $k$-th inner iteration of the $n$-th outer loop.

In order to implement Algorithm \ref{alg: PDHG with restarts} (rPDHG) it is necessary to specify a (verifiable) restart condition on the average iterate $\bar z^{n,k}$ in Line \ref{line:restart_condition} that is used to determine when to re-start PDHG.
We will primarily consider Algorithm \ref{alg: PDHG with restarts} (rPDHG) using the following restart condition in Line \ref{line:restart_condition}:
\begin{equation}\label{catsdogs}\rho(\|\bar{z}^{n,k} - z^{n,0}\|_M; \bar{z}^{n,k}) \le \beta \cdot \rho(\|z^{n,0} - z^{n-1,0}\|_M; z^{n,0}) \ , \end{equation}
for a specific value of $\beta \in (0,1)$ (in fact we will use $\beta = 1/e$ where $e$ is the base of the natural logarithm).  In this way \eqref{catsdogs} is nearly identical to the condition used in \cite{applegate2023faster}. We state this restart condition formally as:
\begin{definition}[$\beta$-restart condition]\label{def beta restart}
	For a given $\beta \in(0,1)$, the iteration $(n,k)$ satisfies the \textit{$\beta$-restart condition} if $n \ge 1$ and condition \eqref{catsdogs} is satisfied, or $n=0$ and $k=1$.
\end{definition}
\noindent Condition \eqref{catsdogs} states that the normalized duality gap shrinks by the factor $\beta$ between restart values $\bar{z}^{n,k}$ and ${z}^{n,0}$. In Appendix \ref{appendix:compute_rho} we show that computing the normalized duality gap can be done efficiently. Also, in practice the restart condition \eqref{catsdogs} does not need to be checked very frequently, so the overall cost of evaluating the restart condition is quite minor.

\begin{remark}[Step-size convention and equivalent reweighted instances]\label{rmk:step-size-vs-reweighting}
Recall the definitions of $\lambda_{\max}$, $\lambda_{\min}$, and $\kappa$ in \eqref{eq  def lamdab min max}. We define the standard step-sizes as follows:
\begin{equation}\label{eq:standard_step_size}
        \tau_0 := \frac{1}{\kappa},
        \qquad
        \sigma_0 := \frac{1}{\lambda_{\max}\lambda_{\min}} \ .
\end{equation}
Then $\tau_0\sigma_0=1/\lambda_{\max}^2$, and hence the standard step-sizes satisfy condition \eqref{eq:general_stepsize} with equality. For a fixed step-size product, changing the ratio between the primal and dual step-sizes is equivalent to applying the original step-sizes to a reweighted instance in which $b$ and $c$ are rescaled while $A$ and $K_p$ are unchanged. The precise equivalence, including the invariance of the restart rule under this reweighting, is given in Appendix~\ref{app:stepsizes_weighted_instances}. Therefore, throughout the main text we state the bounds for the standard step-sizes in \eqref{eq:standard_step_size}. If a different fixed step-size ratio is used, the geometric condition measures should be interpreted as being measured on the corresponding reweighted instance.
\end{remark}

\subsection{Computational guarantees for CLP}

In this subsection we present our computational guarantees for rPDHG, which are based on the three geometry-based condition measures $D_\delta$, $r_\delta$ and $d^H_\delta$ of the level sets $\calW_\delta$. Throughout this subsection, rPDHG uses the $\beta$-restart condition in Definition~\ref{def beta restart}, which is the same adaptive restart scheme introduced in \cite{applegate2023faster} and also used in \cite{xiong2023computational}.

We will state the main computational guarantees in terms of the following two error quantities. \vspace{5pt}

\begin{definition}{(Distance to constraints and duality gap)}\label{def:basic_errors} Let $w=(x,s) \in \mathbb{R}^{2n}$ be given. Define
\begin{enumerate}
	\item (Distance to constraints): $\econs(w):=\max\{\dist(w,V),\dist(w,K)\}$, and
	\item (Duality gap): $\egap(w):=\max\{0,\gap(w)\}$.
\end{enumerate}
\end{definition} 
\vspace{5pt}
\noindent The quantity $\econs(w)$ measures the distances to $V$ and $K$, while $\egap(w)$ is the positive part of the duality gap. A candidate primal-dual pair $w = (x,s) $ is optimal if and only if $\econs(w) = 0$ and $\egap(w) = 0$. In practice, we aim to compute a solution $w$ for which $\econs(w)$ and $\egap(w)$ are both small.
Recalling the definitions of $\lambda_{\max}$, $\lambda_{\min}$, and $\kappa$ from \eqref{eq  def lamdab min max}, we now state our main computational guarantee for Algorithm \ref{alg: PDHG with restarts} (rPDHG).

\begin{theorem}\label{thm overall complexity clp}
	Under Assumption \ref{assump:striclyfeasible}, suppose that $c\in\mathcal{L}$ and Algorithm \ref{alg: PDHG with restarts} (rPDHG) is run starting from $z^{0,0} = (x^{0,0},y^{0,0} ) = (0,0)$ using the $\beta$-restart condition with $\beta := 1/e$, and the standard step-sizes $\sigma=\sigma_0$ and $\tau = \tau_0$ prescribed in \eqref{eq:standard_step_size}.  
	Given $\eps_{\mathrm{cons}},\eps_{\mathrm{gap}}>0$, let $T$ be the total number of \textsc{OnePDHG} iterations that are run in order to obtain $n$ for which $w^{n,0}=(x^{n,0},s^{n,0})$ satisfies $\econs(w^{n,0})\le \eps_{\mathrm{cons}}$ and $\egap(w^{n,0})\le \eps_{\mathrm{gap}}$.
	Then for every $\delta > 0$,
	\begin{equation}\label{eq overall complexity} 
		T \le     99\kappa \, \frac{D_\delta}{r_\delta}  \, \left[\ln\left(17\kappa \, \dist(0,\calW^\star) \left(\tfrac{1}{\eps_{\mathrm{cons}}}\vee\tfrac{2\sqrt{2}\dist(0,\calW^\star)}{ \eps_{\mathrm{gap}}}\right)\right) \right]  
		  + 26\kappa \, d^H_\delta \, \left( \tfrac{1}{\eps_{\mathrm{cons}}} \vee \tfrac{2\sqrt{2}\dist(0,\calW^\star)}{\eps_{\mathrm{gap}}} \right) 	
	\end{equation}
	Let $T_\delta$ denote the right-hand side of \eqref{eq overall complexity}, whereby \eqref{eq overall complexity} implies
	\begin{equation}\label{eq overall complexity 2}
		T \le \inf_{\delta > 0}  \, T_\delta  \, .
	\end{equation}
\end{theorem}

Theorem \ref{thm overall complexity clp} is the formal version of Main Result~\ref{mainres:clp}. It is slightly more general than Main Result~\ref{mainres:clp} in that it allows for separate values of the feasibility tolerance $\eps_{\mathrm{cons}}$ and the duality-gap tolerance $\eps_{\mathrm{gap}}$.  The inverse feasibility tolerance enters directly, while the inverse duality-gap tolerance is scaled by $\dist(0,\calW^\star)$. Thus, for example, if $\dist(0,\calW^\star)$ is small (namely, if the optimal primal and dual optimal variables have small norm), then the impact of $\eps_{\mathrm{gap}}$ on the complexity bound is less pronounced than it would be if $\dist(0,\calW^\star)$ is large.
 
Let us now examine the inverse tolerance dependence in \eqref{eq overall complexity}. Ignoring absolute constants, for any $\delta >0$ the inverse tolerance term $\frac{1}{\eps_{\mathrm{cons}}}\vee \frac{\dist(0,\calW^\star)}{\eps_{\mathrm{gap}}}$ appears in the logarithm term and the right-most term, which we will call the linear term. When the target tolerance is not very small, it is possible for the logarithm term to be the dominant term, while when the target tolerance is very small the linear term will be the dominant term.  Note that the notion of ``very small'' will depend on the relative magnitudes of $\frac{D_\delta}{r_\delta}$ and $d^H_\delta$. 

Observe in \eqref{eq overall complexity 2} that $T$ is bounded above by the \underline{smallest} $T_\delta$ over all $\delta >0$ (namely, over all level sets $\calW_\delta$). Thus the bound depends on the level set that yields the smallest upper bound for the target tolerance. As discussed after Main Result~\ref{mainres:clp}, the constant in front of the linear term decreases to $0$ as $\delta$ goes to $0$, but the constant in front of the logarithm term might go to a constant or might go to $+\infty$. When $\liminf_{\delta\searrow0}D_\delta/r_\delta$ is finite, then Theorem \ref{thm overall complexity clp} implies the following linear convergence result.

\begin{corollary}\label{cor: linear convergence complexity clp}
	In the setting of Theorem \ref{thm overall complexity clp}, the total number of iterations $T$ is bounded above by
	$99\kappa \, \left(\lim\inf_{\delta \searrow 0}\frac{D_\delta}{r_\delta} \right) \, \ln\left(17\kappa\, \dist(0,\calW^\star) \Big(\frac{1}{\eps_{\mathrm{cons}}}\vee\frac{2\sqrt{2}\dist(0,\calW^\star)}{ \eps_{\mathrm{gap}}}\Big) \right)$.
\end{corollary}

Corollary \ref{cor: linear convergence complexity clp} is controlled by the limiting local geometry of the level sets near $\calW^\star$. One way to visualize this is to look at a low-dimensional slice around a unique optimal solution of an LP. For sufficiently small $\delta$, the level set $\calW_\delta$ may resemble a slice of the cone of feasible directions pointed from $w^\star$. This is shown conceptually in Figure \ref{fig:three-instances} which illustrates several possible scenarios of level sets $\calW_\delta$ (yellow regions) for feasible regions (gray regions) truncated within small-gap hyperplanes.  In scenario 1 the level set is a narrow slice and thus has a very small $r_\delta$. In scenario 3, the angle between the feasible region and the hyperplane $\gap(w)=0$ is very small, which leads to a very large value of $D_\delta$. Both of these scenarios lead to an unfavorable limiting ratio $\lim\inf_{\delta \searrow 0}D_\delta/r_\delta$. In scenario 2 the level set has favorable geometry and the limiting ratio is not large. (Note that the balls in these pictures may not be tangent to the cone-like geometry because the figure only shows a slice of the full geometry.)

\begin{figure}[h]
    \centering
    \begingroup
    \setlength{\fboxsep}{0pt}   
    \setlength{\fboxrule}{0.5pt} 
    \fbox{%
        \includegraphics[
            width=\dimexpr\linewidth-2\fboxrule\relax
        ]{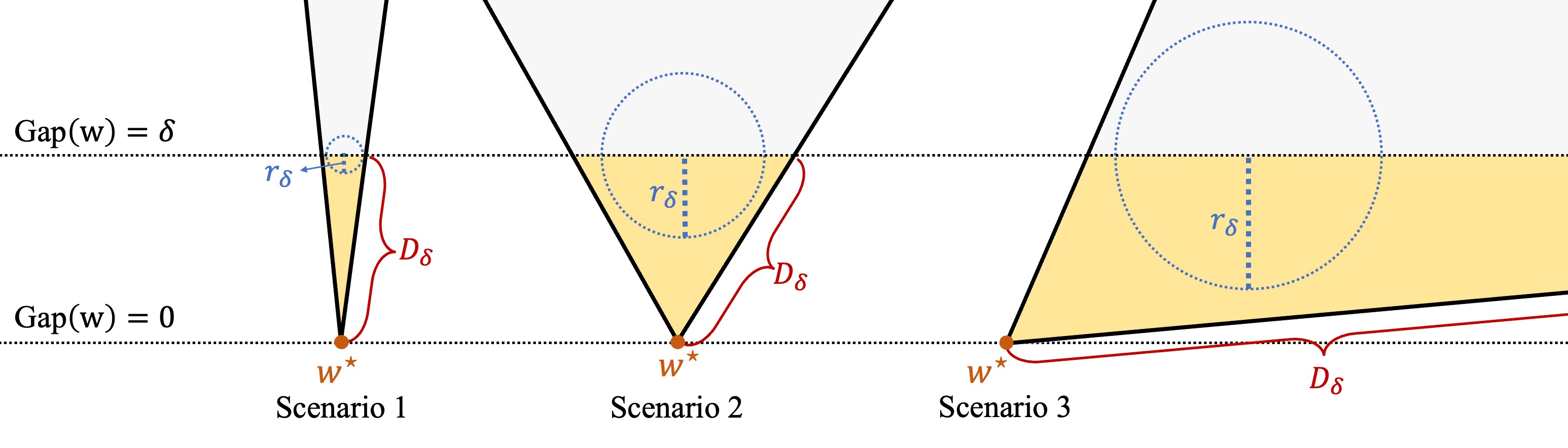}%
    }
    \endgroup
    \caption{Some local scenarios of the level set $\calW_\delta$ (yellow regions) that have favorable or unfavorable limiting ratio $\lim\inf_{\delta \searrow 0}D_\delta/r_\delta$ near a unique optimal solution $w^\star$.}
    \label{fig:three-instances}
\end{figure}

Corollary \ref{cor: linear convergence complexity clp} first appeared in the unpublished manuscript \cite{xiong2024role}. As part of follow-on research, Xiong in \cite{xiong2024accessible} has reported experimental evidence that this bound is consistent with computational practice.

Finally, Theorem \ref{thm overall complexity clp} is stated for the standard step-sizes $\sigma_0$ and $\tau_0$ prescribed in \eqref{eq:standard_step_size}. If one wishes to use a different ratio between $\sigma$ and $\tau$, the same guarantee can be obtained by applying this result to the corresponding reweighted instance; see Remark \ref{rmk:step-size-vs-reweighting}.

\subsection{A general complexity bound property for rPDHG}\label{subsec:abstract_restart_bound}

With the ultimate goal of proving Theorem \ref{thm overall complexity clp}, in this short subsection we first present a general complexity bound property for rPDHG in Theorem \ref{thm: complexity of PDHG with adaptive restart}. Essentially, Theorem \ref{thm: complexity of PDHG with adaptive restart} presents a bound on the iterations of rPDHG if for all iterates their $M$-distance to the saddlepoint set can be bounded by an affine function of the normalized duality gap used in the restart rule, which is formally stated in \eqref{eq restart L C condition}.  After proving Theorem \ref{thm: complexity of PDHG with adaptive restart} below, the proof of Theorem \ref{thm overall complexity clp} will then proceed by verifying that \eqref{eq restart L C condition} holds for rPDHG.  

\begin{theorem}\label{thm: complexity of PDHG with adaptive restart}Under Assumption \ref{assump:striclyfeasible}, suppose that Algorithm \ref{alg: PDHG with restarts} (rPDHG) is run with step-sizes $\sigma$ and $\tau$ satisfying \eqref{eq:general_stepsize}, starting from $z^{0,0} = (x^{0,0},y^{0,0} )$ with $x^{0,0}\in K_p$, using the $\beta$-restart condition with $\beta := 1/e$. Suppose further that there exist constants $\condG \ge 1$ and $ \condC \ge 0$ such that
	\begin{equation}\label{eq restart L C condition}
		\dist_M(z^{n,0}, \calZ^\star) \le \rho(\|z^{n,0} - z^{n-1,0} \|_M;z^{n,0}) \cdot \condG + \condC \
	\end{equation}
	holds for all $n \ge 1$. Let $\eps > 0$ be given and let $T$ be the total number of \textsc{OnePDHG} iterations that are run in order to obtain the first outer iteration $N$ that satisfies $\rho(\|{z}^{N,0} - z^{N-1,0}\|_M; {z}^{N,0})\le \eps$.  Then

	\begin{equation}\label{eq all complexity 6}
		T \ \le \ 12\condG  \cdot \ln\left(\frac{12 \dist_M(z^{0,0},\calZ^\star)}{\eps}\right)  +  \frac{18\condC}{\eps} \ .
	\end{equation}
\end{theorem} 
\proof{Proof.}
	We first derive an upper bound $k_n$ on the number of iterations $k$ between two consecutive restarts $z^{n, 0}$ and $z^{n+1, 0}$.  For $ n =0$, it follows trivially from Definition \ref{def beta restart} that $k_0 := 1$. For $n \ge 1$ and $k \ge 1$ it holds from Lemma \ref{lm: original sublinear PDHG} that
	\begin{equation}\label{eq thm: complexity of PDHG with adaptive restart 1}
		\rho(\|\bar{z}^{n,k}-z^{n,0}\|_M;\bar{z}^{n,k}) \le \frac{4\dist_M(z^{n,0} ,\calZ^\star)}{k}\ .
	\end{equation}
	If $\rho(\|z^{n,0}-z^{n-1,0}\|_M; z^{n,0}) = 0$ it follows from Lemma \ref{lm: convergence of PHDG without restart} that $z^{n,0} \in \calZ^\star$ which then implies that $z^{n,k} = z^{n,0}$ for all $k\ge 1$, and so in particular $k_n = 1$.  If $\rho(\|z^{n,0}-z^{n-1,0}\|_M; z^{n,0}) \neq 0$, then dividing both sides of \eqref{eq thm: complexity of PDHG with adaptive restart 1} by $\rho(\|z^{n,0}-z^{n-1,0}\|_M; z^{n,0})$ yields:
	\begin{equation}\label{eq thm: complexity of PDHG with adaptive restart 2}
		\frac{	\rho(\|\bar{z}^{n,k}-z^{n,0}\|_M;\bar{z}^{n,k})}{\rho(\|z^{n,0}-z^{n-1,0}\|_M; z^{n,0})} \le \frac{4}{k}  \cdot 	\frac{\dist_M(z^{n,0} ,\calZ^\star)}{\rho(\|z^{n,0}-z^{n-1,0}\|_M; z^{n,0})}  \ .
	\end{equation}
	Furthermore, \eqref{eq restart L C condition} implies:
	\begin{equation}\label{miggie}
		\frac{\dist_M(z^{n,0} ,\calZ^\star)}{\rho(\|z^{n,0}-z^{n-1,0}\|_M; z^{n,0})} \le  \condG + \frac{\condC}{\rho(\|z^{n,0} - z^{n-1,0} \|_M;z^{n,0})} \ .
	\end{equation}
	Let us define: $ \bar k_n :=    \frac{4}{\beta} \cdot \left(
		\condG + \frac{\condC}{\rho(\|z^{n,0} - z^{n-1,0} \|_M;z^{n,0})}
		\right) $.
	It then follows from \eqref{eq thm: complexity of PDHG with adaptive restart 2} and \eqref{miggie} that
	the restart condition \eqref{catsdogs} is satisfied for all $k \ge \bar k_n$, whereby
	\begin{equation}\label{eq restart iteration condition}k_n := \frac{4}{\beta} \cdot \left(
		\condG + \frac{\condC}{\rho(\|z^{n,0} - z^{n-1,0} \|_M;z^{n,0})}
		\right)  + 1  \end{equation} is an upper bound on the number of iterations between the consecutive restarts $z^{n, 0}$ and $z^{n+1, 0}$.

	Next we examine the first outer iteration $N$ that satisfies $\rho(\|{z}^{N,0} - z^{N-1,0}\|_M; {z}^{N,0})\le \eps$.  It follows from the $\beta$-restart condition \eqref{catsdogs} for all $n \ge 1$ that
	\begin{equation}\label{eq all complexity 1}
		\rho(\|{z}^{n,0} - z^{n-1,0}\|_M; {z}^{n,0}) \le \beta^{n-1} \cdot \rho(\|{z}^{1,0} - z^{0,0}\|_M; {z}^{1,0}) \ ,
	\end{equation}
	which in combination with Lemma \ref{lm: original sublinear PDHG} (and using $k=1$) yields:
	\begin{equation}\label{eq all complexity 2}
		\rho(\|{z}^{n,0} - z^{n-1,0}\|_M; {z}^{n,0}) \le 4 \beta^{n-1} \cdot \dist_M(z^{0,0}, \calZ^\star) \ .
	\end{equation}
	Let us define $\bar N := 1+ \frac{1}{\ln( 1/\beta) } \cdot \ln\left( \frac{4 \dist_M(z^{0,0}, \calZ^\star)}{ \eps} \right) $.  It follows from \eqref{eq all complexity 2} that $\rho(\|{z}^{n,0} - z^{n-1,0}\|_M; {z}^{n,0}) \le \eps$ for all $n \ge \bar N$, from which it follows that
	\begin{equation}\label{eq all complexity 3}
		N \ \le \  \frac{1}{\ln( 1/\beta) } \cdot \ln\left( \frac{4  \dist_M(z^{0,0}, \calZ^\star)}{ \eps} \right) + 2 \ . 	\end{equation}The total number of \textsc{OnePDHG} iterations $T$ satisfies
	\begin{equation}\label{eq all complexity 4}
		\begin{aligned}
			T \le \sum_{n=0}^{N-1} k_n  = \  & 1  + \sum_{n = 1}^{N-1}\left(
			\frac{4}{\beta} \cdot \left(
				\condG + \frac{\condC}{\rho(\|z^{n,0} - z^{n-1,0} \|_M;z^{n,0})}
				\right) + 1
			\right)                                                                                                                      \\
			\le \                            & N  + \frac{4\condG (N-1)}{\beta} +  \frac{4\condC}{\beta}  \cdot \sum_{n = 1}^{N-1}\left(
			\frac{1}{\rho(\|z^{n,0} - z^{n-1,0} \|_M;z^{n,0})}
			\right) \ .
		\end{aligned}
	\end{equation}
	We have from the definition of $N$ that $\rho(\|{z}^{N-1,0} - z^{N-2,0}\|_M; {z}^{N-1,0}) > \eps$.  Also, the $\beta$-restart condition \eqref{catsdogs} implies for all $n \le N-2$ that:
	$$
		\frac{	1}{\rho(\|z^{n,0} - z^{n-1,0}\|_M; z^{n,0})} \le 	 \frac{\beta}{\rho(\|z^{n+1,0} - z^{n,0}\|_M; z^{n+1,0})}  \le \cdots \le
		\frac{\beta^{N - 1 -n}}{\rho(\|{z}^{N-1,0} - z^{N-2,0}\|_M; {z}^{N-1,0})} \ ,
	$$
	Therefore
	\begin{equation}\label{eq all inverse gap}
		\sum_{n = 1}^{N-1}\left(
		\frac{1}{\rho(\|z^{n,0} - z^{n-1,0} \|_M;z^{n,0})}
		\right) < \frac{1}{\eps} \cdot \left(
		1 + \beta^{1} + \beta^{2} + \cdots
		\right) = \frac{1}{\eps(1-\beta)} \ .
	\end{equation}
	Substituting \eqref{eq all inverse gap}  into \eqref{eq all complexity 4} yields
	\begin{equation}\label{eq all complexity 5}
		T \le  N  + \frac{4\condG (N-1)}{\beta} +  \frac{4\condC}{\beta}  \cdot\frac{1}{\eps(1-\beta)} =  \left(1 + \frac{4 \condG}{\beta}\right) N - \frac{4 \condG}{\beta} + \frac{4 \condC}{\beta (1- \beta)}\cdot \frac{1}{\eps} \ ,
	\end{equation}
	and then using \eqref{eq all complexity 3} we arrive at
	$$
		T \le 	\left(1 + \frac{4 \condG}{\beta} \right)\cdot \frac{1}{\ln(1/\beta)} \cdot \ln\left(\frac{4 \dist_M(z^{0,0},\calZ^\star)}{\eps}\right) + \left(2 + \frac{8 \condG}{\beta} \right) - \frac{4\condG}{\beta}
		+ \frac{4 \condC}{\beta (1-\beta)}\cdot \frac{1}{\eps} \ .
	$$
	Noting that $\condG \ge 1$ and $\beta = 1/e \approx     0.3679$, the above bound can be relaxed slightly to yield 
	\begin{equation}\label{eq all complexity 7}
		T \ \le \ 12\condG  \cdot \ln\left(\frac{4 \dist_M(z^{0,0},\calZ^\star)}{\eps}\right)  + 13 \condG
		+  \frac{18\condC}{\eps} \ .
	\end{equation}
	Finally, notice that the middle term above satisfies $13 \condG \le 12\condG \cdot \ln\big(e^{13/12}\big)$ and substituting it back into \eqref{eq all complexity 7} yields  $T \le \ 12\condG  \cdot \ln\left(\frac{4 \cdot e^{13/12} \cdot \dist_M(z^{0,0},\calZ^\star)}{\eps}\right) + \frac{18\condC}{\eps}$, which then simplifies to \eqref{eq all complexity 6}.
\Halmos\endproof

Note that if there exists $\condG < \infty$ and $\condC =0$ for which \eqref{eq restart L C condition} holds for all $n \ge 1$, then Theorem \ref{thm: complexity of PDHG with adaptive restart} shows linear convergence, as it states that the total number of \textsc{OnePDHG} iterations required to obtain a normalized duality gap smaller than $\eps$ is bounded above by $O(\condG \cdot \ln(\dist_M(z^{0,0}, \calZ^\star)/\eps))$.
 
\subsection{Properties of the level-set geometry condition measures $D_\delta$, $r_\delta$, and $d^H_\delta$}\label{subsec:lemmas_condition_measures}

With the ultimate goal of proving Theorem \ref{thm overall complexity clp} by appropriately applying inequality  \eqref{eq restart L C condition}, in this subsection we first prove some properties of the level set condition measures $D_\delta$, $r_\delta$, and $d^H_\delta$ which were originally defined in Section~\ref{sec:intro_sublevel-set-condition-number}.

Typically, the smaller $\delta$ is, the smaller all three condition measures are. As $\delta$ goes to $0$, $r_\delta$ and $d^H_\delta$ must converge to $0$.  We first present the following straightforward observation regarding $D_\delta$, $r_\delta$ and $d^H_\delta$.

\begin{lemma}\label{lm D ge r}
	Under Assumption \ref{assump:striclyfeasible} it holds for all $\delta > 0$ that $D_\delta \ge d^H_\delta \ge r_\delta > 0$ .
\end{lemma}
\proof{Proof.} We have
	$$d^H_\delta = \max_{w \in \calW_\delta} \min_{w' \in \calW^\star} \|w-w'\| \le \max_{w \in \calW_\delta} \max_{w' \in \calW^\star} \|w-w'\| \le \max_{w \in \calW_\delta} \max_{w' \in \calW_\delta} \|w-w'\| = D_\delta \ , $$
	where the second inequality follows since $\calW^\star \subseteq \calW_\delta$.
	Given $(x^\star, s^\star) \in \calW^\star$, then since $x^\star \in K_p$ and $s^\star \in K_d = K_p^\star$ and $(x^\star)^\top s^\star =0$, it follows that either $x^\star \in \partial K_p$ or $s^\star \in \partial K_d$ (or both), whereby $(x^\star, s^\star) \in \partial K$.  Therefore $\calW^\star \subset \partial K$ and we have:
	$$d^H_\delta = \max_{w \in \calW_\delta} \dist(w, \calW^\star) \ge \dist(w_\delta, \calW^\star) \ge \dist(w_\delta, \partial K) = r_\delta  \ , $$
	where the second inequality follows from $\calW^\star \subset \partial K$.
	Last of all, under Assumption \ref{assump:striclyfeasible} and using the convexity of the level sets it follows that there exists $(x,s) \in \calW_\delta$ with $(x,s) \in \textsf{int} K_p \times \textsf{int} K_d$, and hence $(x,s) \in \textsf{int} K$ and $r_\delta > 0$.\Halmos\endproof

Intuitively, the larger $\delta$ is then the larger $r_\delta$ would be. The lemma below presents lower and upper bounds on $\sup_{\gamma>0} \frac{r_\gamma}{\gamma}$, and shows that this quantity is nearly a reciprocal of the maximum norm of a point in the optimal solution set $\calW^\star$. 

\begin{lemma}\label{lm r_delta over delta}
	Under Assumption \ref{assump:striclyfeasible}, it holds that
	\begin{equation}\label{eq of lm r calW small}
		\frac{\width_K}{\max_{w\in\calW^\star}\|w\|}\le \sup_{\gamma >0}\frac{r_\gamma}{\gamma} \le \frac{1}{\max_{w\in\calW^\star}\|w\|} \ ,
	\end{equation}
	in which $\width_K$ is the width of the cone $K:=K_p\times K_d$ defined below in Definition \ref{def width of cone}. \end{lemma}

\begin{definition}[Width of a cone]\label{def width of cone}
	The width of a cone $K$ is defined as:
\begin{equation}\label{eq:width_K}
	\width_{K}:=
	\max \left\{\left.\frac{r}{\|x\|} \right\rvert\, x\ne0,\ r\ge0,\ B(x, r) \subseteq K\right\} \ .
\end{equation}
\end{definition} 
\noindent
The width of a cone is an intrinsic property of the cone, though it depends on the choice of norm. Under the Euclidean norm, some standard families of cones satisfy $\width_{\mathbb{R}^n_+} = 1/\sqrt{n}$, $\width_{\mathbb{S}_{+}^{d \times d}} = 1/\sqrt{d}$, and $\width_{\mathbb{K}_{\textsf{soc}}^{d+1}} = 1/\sqrt{2}$. See, for example, \cite{freund1999condition}.
The ratio $r_\gamma/\gamma$ will appear later in the proof of our main results. The proof of Lemma \ref{lm r_delta over delta} is presented in Appendix \ref{app:proof of lm r calW small}. 

We next develop and state an ``error bound'' type of result for $\calW_\delta$ involving the quotient $D_\delta / r_\delta$ that will be used in later proofs, but that perhaps might be of independent interest. Let us use the notation $\calF_{++}$ to denote the set of strictly feasible solutions in $\calF$, namely $\calF_{++} := V \cap \textsf{int} K$. Let $w\in V \setminus \calF$ and $w_{int}\in \calF_{++}$ be given, whereby the line segment from $w$ to $w_{int}$ will contain a unique point that lies in $\partial K$, and let us denote this point by $\calF(w;w_{int}) $.  More formally we have
\begin{equation}\label{eq  def v}
	\calF(w;w_{int})  := \arg\min_{\tilde{w}}\left\{\|w - \tilde{w}\| : \tilde{w}:= \lambda \cdot w_{int} + (1- \lambda)\cdot w \ \mathrm{for \ some} \  \lambda, \ \mathrm{and} \ \tilde{w}\in \calF
	\right\} \ .
\end{equation}
The following lemma states for the level set $\calW_\delta$ that if the ratio $D_\delta / r_\delta$ is small, then a point in $V \cap \{w: \gap(w)\le \delta\}$ that is close to $K$ must also be close to $\calW_\delta$. In this way we see that $D_\delta / r_\delta$ provides an error bound for $\calW_\delta$.

\begin{lemma}\label{lm error bound R r}
	For any $\delta >0$ and $w \in V$ with $\gap(w) \le \delta$,  it holds that either $\dist(w,K) = 0$ and $w \in \calW_\delta$, or
	\begin{equation}\label{eq local error bound}
		\frac{\dist(w, \calW_\delta)}{\dist(w, K)} \le \frac{\| w - \calF(w;w_\delta) \|}{\dist(w, K)} \le \frac{\| w_\delta - \calF(w;w_\delta) \|}{r_\delta}  \le \frac{D_\delta}{r_\delta} \ .
	\end{equation}
\end{lemma}
\proof{Proof.} 
	If $\dist(w,K) =0$, then $w \in \calW_\delta$ because $w \in V$ and $\gap(w) \le \delta$ by the hypotheses of the lemma.  If $w \notin K$, then $w \notin \calF$, and let $v := \calF(w;w_\delta)$.  Notice that $w_\delta \in \calW_\delta $, $w \in V$, and $\gap(w) \le \delta$ together imply that $v := \calF(w;w_\delta) \in \calW_\delta$. Then the first inequality in \eqref{eq local error bound} holds because $v \in \calW_\delta$ and so $ \dist(w, \calW_\delta) \le \| w - v \| = \| w - \calF(w;w_\delta) \|$.  For the third inequality of \eqref{eq local error bound} notice that $w_\delta \in \calW_\delta $ and $v \in \calW_\delta $ imply that $ \| w_\delta - v \| = \| w_\delta - \calF(w;w_\delta) \|  \le \Diam(\calW_\delta) = D_\delta$, which yields the third inequality of \eqref{eq local error bound}.

	We now prove the second inequality of \eqref{eq local error bound}. From the definition in \eqref{eq  def v}, because $w_\delta \in \calF_{++}$, there exists $\lambda \in (0,1)$ for which
	\begin{equation}\label{eq actual set v}
		v = {\lambda} \cdot w_{\delta} + (1- {\lambda}) \cdot w \ .
	\end{equation}
	Furthermore, since $v \in \partial K$, then there exists a supporting hyperplane $H$ of $K$ that contains $v$.  It then follows that there exists $p \in \mathbb{R}^{2n}$ for which $H := \{\hat{w}\in \mathbb{R}^{2n}: p^\top \hat{w} = 0\}$, $p^\top v = 0$, $p \in K^*$, and $p^\top w<0$ and $p^\top \hat w >0$ for all $\hat w \in \textsf{int}K$ and so in particular $p^\top w_\delta > 0$.  From \eqref{eq actual set v} we have
	$$
		{\lambda} \cdot p^\top w_\delta + (1 - {\lambda}) \cdot p^\top  w = p^\top v = 0 \ , $$ which can be rearranged to yield the following equalities:
	\begin{equation}\label{eq error bound R r}
		\frac{\lambda}{1-\lambda} = \frac{|p^\top w |}{|p^\top w_\delta |}  = \frac{\dist(w, H)}{\dist(w_\delta,H)} \ .
	\end{equation}
	From \eqref{eq actual set v} the left side of \eqref{eq error bound R r} can be further expressed as \begin{equation}\label{eq error bound R r 1}
		\frac{\lambda}{1-\lambda} = \frac{\| w - v\|}{\|w_\delta - v\|} \ .
	\end{equation}
	Also, since $K$ and $w$ are on different sides of the hyperplane $H$, this implies that	\begin{equation}\label{eq error bound R r 2}
		\dist(w, H) \le \dist(w, K) \ .
	\end{equation}
	Additionally, because $B(w_\delta,r_\delta) \subseteq K$ (from Definition \ref{def radius}), we have
	\begin{equation}\label{eq error bound R r 3}
		\dist(w_\delta, H) \ge r_\delta \ .
	\end{equation}
	Substituting \eqref{eq error bound R r 1}, \eqref{eq error bound R r 2}, \eqref{eq error bound R r 3} back into \eqref{eq error bound R r} yields
	\begin{equation}\label{eq error bound R r 4}
		\frac{\| w - v\|}{\|w_\delta - v\|} \le \frac{ \dist(w, K)  }{r_\delta} \ ,
	\end{equation}
	which proves the second inequality in  \eqref{eq local error bound}.
\Halmos\endproof

We will use the following elementary convexity observation in the proofs below.
\begin{proposition}\label{lm:convexity_ineq}Let $f:[0,\infty)\mapsto \mathbb{R}$ and $g:[0,\infty)\mapsto \mathbb{R}$ be nonnegative convex functions for which $f(0) = g(0) $, and suppose that $f$ is linear. If there exists $u > 0$ such that $g(u)\ge f(u) $, then for any $v \ge u$ it holds that  $g(v) \ge f(v)$.
\end{proposition}
\proof{Proof.} Define $F(\cdot) := g(\cdot) - f(\cdot)$, and note that $F:[0,\infty)\mapsto \mathbb{R}$ is convex since $f(\cdot)$ is a linear function, also $F(0) = 0$, and $F(u)\ge 0$. From the convexity of $F$ we have for $v > u$ that $(\frac{u}{v})\cdot F(v)  = (\frac{u}{v})\cdot F(v) + (\frac{v - u}{v})\cdot F(0) \ge  F\left(	(\frac{u}{v})\cdot v + (\frac{v - u}{v})\cdot 0 	\right) = F(u)\ge 0$.  Hence $ g(v) -f(v) \ge 0$ for $v >u$. And when $v=u$ the result holds trivially.
\Halmos\endproof

The next lemma states that the distance to the set of optimal solutions can be bounded by terms involving the distance to constraints and the Hausdorff distance $d^H_\delta$.  (The lemma presumes that $c \in \operatorname{Null}(A)$, as discussed previously in Section \ref{sec:get_general_clp_dual}.)

\begin{lemma}\label{lm efeas bound dist} Under the presumption that $c \in \operatorname{Null}(A)$, for any  $w = (x,s) \in \mathbb{R}^{2n}$,  and any $\delta>0$, it holds that
	\begin{equation}\label{eq gap small bound}
		\dist(w,\calW^\star) \le \frac{3 D_\delta}{r_\delta} \cdot  \max\{\dist(w,V),\dist(w,K)\} + \frac{d^H_\delta}{\delta} \cdot \max\{\gap(w),\delta\} \ .
	\end{equation}
\end{lemma}

\proof{Proof.}
	We first consider the case where $\delta \ge  \gap(w)$. Define $\hat{w} := P_{V}(w) $, whereby $\|w - \hat{w}\| = \dist(w,V)$. Then because $c \in \mathcal{L}$ and $q \in \mathcal{L}^{\bot}$, then $\gap(w) = \gap(\hat{w})$, therefore $\gap(\hat{w}) \le \delta$ as well. We have:
	\begin{equation}\label{eq error bound distance clp 1}
		\dist(w, \calW^\star) \le \dist(\hat{w}, \calW^\star) + \| \hat{w}- w\|  = \dist(\hat{w}, \calW^\star) +\dist(w, V) \ ,
	\end{equation} and from the definition of $d^H_\delta$ we also have:
	\begin{equation}\label{eq error bound distance clp 2}
		\dist(\hat{w},\calW^\star) \le \dist(\hat{w}, \calW_\delta) + d^H_\delta \ .
	\end{equation}
	Since $\hat{w}\in V$ and $\gap(\hat{w})\le \delta$, from Lemma \ref{lm error bound R r} it follows that $ \dist(\hat{w}, \calW_\delta) $ in \eqref{eq error bound distance clp 2} can be bounded as follows:
	\begin{equation}\label{eq error bound distance clp 3}
		\begin{aligned}
			\dist(\hat{w}, \calW_\delta) & \le \frac{D_\delta}{r_\delta} \cdot \dist(\hat{w},K)
			\le \frac{D_\delta}{r_\delta} \cdot\left(\|w - \hat{w}\| + \dist(w,K)    \right)
			\\
			                             & \le  \frac{2 D_\delta}{r_\delta} \cdot \max\{\dist(w,V),\dist(w,K)\}  \ .
		\end{aligned}
	\end{equation}
	Then combining \eqref{eq error bound distance clp 3}, \eqref{eq error bound distance clp 2}, and \eqref{eq error bound distance clp 1} yields
	$$
		\dist(w,\calW^\star) \le \frac{2 D_\delta}{r_\delta} \cdot  \max\{\dist(w,V),\dist(w,K)\} +\dist(w , V) + d^H_\delta \ .
	$$
	Last of all note that $D_\delta \ge r_\delta$ from Lemma \ref{lm D ge r}, from which the above inequality then implies \eqref{eq gap small bound}.

	Let us now consider the case where  $\delta \le \gap(w)$.  Here we will make use of Proposition \ref{lm:convexity_ineq} to complete the proof.  Let $w^\star = P_{\calW^\star}(w) = \arg\min_{\bar{w}\in\calW^\star}\|\bar{w} - w\|$ and define $w_t := w^\star + t\cdot (w - w^\star)$ for $t\in[0,\infty)$. Then define the following functions of $t$ :
	$$
		f(t):= \dist(w_t,\calW^\star) \ ,  \text{ and  }g(t):= \frac{3 D_\delta}{r_\delta} \cdot  \max\{\dist(w_t,V),\dist(w_t,K)\} + \frac{d^H_\delta}{\delta} \cdot \max\{\gap(w_t),0\} \ .
	$$
	Then $f(t)$ is a nonnegative  linear  function on $[0,\infty)$, and $f(0) = 0$. And $g(t)$ is convex and nonnegative on $[0,\infty)$, and $g(0) = 0$. In addition, because $\gap(\cdot)$ is a linear function and $\gap(w_t) = t \cdot \gap(w)$, then setting $u := \delta / \gap(w)$ we obtain $ \gap(w_u) = u\cdot \gap(w) = \delta$. We can then invoke \eqref{eq gap small bound} using $w_u$ in the place of $w$, which yields $g(u) \ge f(u)$. Now it follows from Proposition \ref{lm:convexity_ineq} with $v := 1\ge  u$ that  $g(1) \ge f(1)$, which is precisely \eqref{eq gap small bound} in the case $\delta \le \gap(w)$, and completes the proof.
\Halmos\endproof

\subsection{Completing the proof of Theorem \ref{thm overall complexity clp} }\label{subsec:normalized_gap_restart_condition}

In this section we use the results in Sections \ref{subsec:abstract_restart_bound} and \ref{subsec:lemmas_condition_measures} to prove several more lemmas and a proposition, which are then finally used to prove Theorem \ref{thm overall complexity clp} via the complexity bound inequality \eqref{eq restart L C condition}.  We will make extensive use of the following constant in translating between distances in different spaces with the $M$-norm and the Euclidean norm:
\begin{equation}\label{icemelt}
	c_0 := \max\left\{
	\frac{1}{\sqrt{\sigma} \lambda_{\min}  }, \ \frac{1}{\sqrt{\tau}}
	\right\} \ .
\end{equation}
We first prove the following lemma on the relation between two particular norms and spaces.
\begin{lemma}\label{lm change of norm}
	Suppose that $\tau,\sigma$ satisfy \eqref{eq:general_stepsize}. Let $\calX\subseteq\mathbb{R}^n$ be a nonempty closed convex set, and let $\calS\subseteq c+\operatorname{Im}(A^\top)$ be a nonempty closed convex set. Define $\calY:=\{y\in\mathbb{R}^m:c-A^\top y\in\calS\}$. Given any point $z :=(x,y) \in \mathbb{R}^{n+m}$, let $s:= c - A^\top y$ and $w := (x,s) $.  Then it holds that
	\begin{equation}\label{eqlm change of norm}
		\dist_M(z, \calX \times \calY) \le \sqrt{2}c_0 \cdot \dist(w,\calX \times\calS) \ .
	\end{equation}
\end{lemma}
\proof{Proof.}
	Define:
	$$
		\left\|(x,y)\right\|_N := \sqrt{\frac{1}{\tau} \|x\|^2 + \frac{1}{\sigma} \|y\|^2} \ \text{ where } \
		N:= 	\begin{pmatrix}
			\frac{1}{\tau}I_n &                     \\
			                  & \frac{1}{\sigma}I_m
		\end{pmatrix} \ .
	$$
	Then $2N - M \in \mathbb{S}^{m+n}_+$ because $1/(\tau \sigma) \ge \lambda_{\max}^2$ due to \eqref{eq:general_stepsize}, and therefore
	$\|z\|_M \le \sqrt{2}\|z\|_N $ for any $z$. This means
	\begin{equation}\label{ineq key norm 1}
		\dist_M(z, \calX \times \calY) \le \sqrt{2}  \cdot \dist_N(z,\calX \times\calY) =
		\sqrt{2}\cdot \sqrt{\frac{1}{\tau}\cdot\dist(x,\calX)^2 + \frac{1}{\sigma}\cdot\dist(y,\calY)^2}
		\ .
	\end{equation}

	Next we claim that
	\begin{equation}\label{ineq key norm 2}
		\dist(y,\calY) \le \dist(s,\calS)\cdot \frac{1}{\lambda_{\min}} \ .
	\end{equation}
	Towards establishing \eqref{ineq key norm 2}, first observe that:
	$$
		\dist(s,\calS) = \dist(c - A^\top y, \calS) = \dist(c - A^\top y, c - A^\top (\calY)) = \dist(A^\top y, A^\top (\calY)) = \dist_{AA^\top} (y,\calY) \ .
	$$
	Let $AA^\top = PD^2P^\top$ denote the thin eigendecomposition of $AA^\top$, so that $P^\top P=I$ and $D$ is the diagonal matrix of positive singular values of $A$, whereby $D_{ii} \ge \min_j D_{jj} = \lambda_{\min}$ for each $i$.  Now let $\hat{y}$ solve the shortest distance problem from $y$ to $\calY$ in the norm $\| \cdot\|_{AA^\top}$, hence $\hat{y} \in \calY$ and  $\dist_{AA^\top} (y,\calY) = \| y - \hat{y}\|_{AA^\top}$, and let us write $y-\hat{y} = u + v $ where $u \in \operatorname{Im}(A)$ and $v \in \operatorname{Null}(A^\top)$. Then setting $\tilde y = \hat{y} + v$ and noting that $\tilde y \in \calY$, we have:
	\begin{equation}\label{skylight} \dist_{AA^\top}(y,\calY) \le \|y - \tilde y\|_{AA^\top} = \|u\|_{AA^\top} \ . \end{equation}
	Next notice that since $u \in \operatorname{Im}(A) = \operatorname{Im}(AA^\top)$, there exists $\pi$ for which $u=AA^\top \pi$, and define $\lambda = D^2P^\top \pi$. It then follows that $u = P \lambda$, $\lambda = P^\top u$, and $\|u\| = \|\lambda\|$.   We therefore have:
	\begin{equation}\label{firepit}\begin{aligned}
			\dist_{AA^\top} (y,\calY)^2 & = (u+v)^\top AA^\top (u+v) \\ & = u^\top AA^\top u = \lambda^\top P^\top P D^2P^\top P\lambda = \lambda^\top D^2 \lambda \ge \lambda_{\min}^2 \|\lambda\|^2 \ ,
		\end{aligned}
	\end{equation}
	and hence
	$$
		\dist(s,\calS) = \dist_{AA^\top} (y,\calY) \ge \lambda_{\min} \|\lambda\| = \lambda_{\min} \|u\| \ge \lambda_{\min}\dist(y,\calY) \ ,
	$$
	where the second inequality follows since $\tilde y\in\calY$ and $y-\tilde y=u$.  This proves \eqref{ineq key norm 2}.

	Finally, combining \eqref{ineq key norm 1} and \eqref{ineq key norm 2} we obtain
	\begin{equation}\label{ineq key norm 3}
		\dist_M(z, \calX \times \calY) \le
		\sqrt{2}\cdot \sqrt{\frac{1}{\tau}\cdot\dist(x,\calX)^2 + \frac{1}{\sigma \lambda_{\min}^2}\cdot\dist(s,\calS)^2} \le \sqrt{2}c_0\cdot \dist(w,\calX\times \calS) \ .
	\end{equation}
\Halmos\endproof

Under the presumption that $c \in \operatorname{Null}(A)$, we have the following property of the initial iterate $z^{0,0}  =(x^{0,0},y^{0,0}):=(0,0)$.
\begin{proposition}\label{lm:distance_to_optimal_initial} Suppose that $c \in \operatorname{Null}(A)$ and the initial iterate is $(z^{0,0})  =(x^{0,0},y^{0,0}):=(0,0)$, and define $w^{0,0} = (x^{0,0},s^{0,0}) := (0,c-A^\top y^{0,0})$. Then
	\begin{equation}\label{eq bound w to wstar 2}
		\dist(w^{0,0},\calW^\star) \le \dist(0,\calW^\star)   \ .
	\end{equation}
\end{proposition}
\proof{Proof.}
	Recall that $\calX^\star$ and $\calS^\star$ are the optimal primal and dual slack-variable solution sets, respectively.  Then $(x^{0,0},y^{0,0}) = (0,0)$ implies that $w^{0,0} = (x^{0,0},s^{0,0}) = (0,c)$ and hence
	\begin{equation}\label{eq bound w to wstar 1}
		\dist(w^{0,0},\calW^\star) = \sqrt{
    \dist(x^{0,0},\calX^\star)^2
    +
    \dist(s^{0,0},\calS^\star)^2} = \sqrt{
    \dist(0,\calX^\star)^2
    +
    \dist(c,\calS^\star)^2} \ .
	\end{equation}
	Let $\hat{s} \in \arg\min_{s\in\calS^\star} \| {s}\|$.  Then we have
	$$ \| \hat s \|^2 =  \| (\hat s -c) + c \|^2 = \| (\hat s -c)\|^2 + \| c \|^2 \ge \| (\hat s -c)\|^2 \ ,
	$$where the second equality follows since  $\hat{s} - c \in \mathcal{L}^{\bot}$ and $c \in \mathcal{L}$. Therefore
	$\dist(c,\calS^\star) \le \|c-\hat s\| \le \| \hat s \| = \dist(0,\calS^\star)$. Substituting this inequality into \eqref{eq bound w to wstar 1} yields
	\begin{equation}\label{eq bound w to wstar 3}
		\dist(w^{0,0},\calW^\star) \le \sqrt{\dist(0,\calX^\star)^2  + \dist(0,\calS^\star)^2 } = \dist(0,\calW^\star)   \ .
	\end{equation}\Halmos\endproof

\begin{lemma}\label{lm use gap to bound error in w}
	Suppose the step-sizes $\tau,\sigma$ satisfy \eqref{eq:general_stepsize} and the starting point is $z^{0,0} =(x^{0,0},y^{0,0}) := (0,0)$ and so $w^{0,0}= (x^{0,0},s^{0,0}) = (0,c)$. Then for any outer iteration value $n \ge 1$ it holds that:
	\begin{align}
		 & \econs(w^{n,0}) \le c_0 \cdot \rho(\| z^{n,0} - z^{n-1,0}\|_M; z^{n,0} ) \, ; \label{ineq translate feasibility error}           \\
		 & \egap(w^{n,0}) \le
		2\sqrt{2}c_0\cdot  \dist(0,\calW^\star)
		\cdot  \rho(\| z^{n,0} - z^{n-1,0}\|_M ; z^{n,0} ) \ . \label{ineq translate gap error}
	\end{align} where $c_0$ is defined in \eqref{icemelt}.
\end{lemma}
\proof{Proof.}
	The inequality \eqref{ineq translate feasibility error} follows directly from items ({\it \ref{item_gap1}.})  and ({\it \ref{item_gap2}.}) of Lemma \ref{lm: convergence of PHDG without restart}. Towards the proof of \eqref{ineq translate gap error}, note that $z^a := z^{0,0}$, $z^b:=z^{n,0}$, and $z^c:= z^{n-1,0}$ satisfy the nonexpansive conditions of Lemma \ref{lm: R in the opt gap convnergence} whereby it follows from Lemma \ref{lm: R in the opt gap convnergence} that
	$$
		\max\{ \|z^{n,0} - z^{n-1,0}\|_M, \|z^{n,0}\|_M\}  \le    2 \cdot \dist_M(z^{0,0},\calZ^\star)  + \|z^{0,0} \|_M = 2 \cdot \dist_M(z^{0,0},\calZ^\star) \ ,
	$$
	where the equality follows since $z^{0,0} = 0$. And applying Lemma \ref{lm change of norm} with $(\calX,\calS,\calY):=(\calX^\star,\calS^\star,\calY^\star)$, together with Proposition \ref{lm:distance_to_optimal_initial}, we obtain	$$
		\dist_M(z^{0,0},\calZ^\star)  \le \sqrt{2}c_0 \cdot \dist(w^{0,0}, \calW^\star)  \le \sqrt{2}c_0 \cdot \dist(0, \calW^\star)  \ .
	$$
	Inequality \eqref{ineq translate gap error} then follows directly from the above two inequalities and ({\it \ref{item_gap3}.}) of Lemma \ref{lm: convergence of PHDG without restart} using $r:= \|z^{n,0} - z^{n-1,0}\|_M$.
	This completes the proof.
\Halmos\endproof

We can now combine these estimates. Lemma \ref{lm change of norm} relates the $M$-distance in Theorem \ref{thm: complexity of PDHG with adaptive restart} to distances in the cone variables, Lemma \ref{lm use gap to bound error in w} bounds feasibility and gap errors by the normalized duality gap, and Lemma \ref{lm efeas bound dist} converts these errors into a distance to $\calW^\star$. The next lemma combines these results into a relationship of the form \eqref{eq restart L C condition}.

\begin{lemma}\label{thm L C}
	Suppose that $c \in \operatorname{Null}(A)$. Under Assumption \ref{assump:striclyfeasible}, suppose that Algorithm \ref{alg: PDHG with restarts} (rPDHG) is run starting from $z^{0,0} = (x^{0,0},y^{0,0}) = (0,0)$, and the step-sizes $\sigma$ and $\tau$ satisfy the step-size inequality \eqref{eq:general_stepsize}. Then for every $n\ge 1$ and any $\delta > 0$, it holds that
	\begin{equation}\label{eq thm L C}
		\dist_M(z^{n,0},\calZ^\star) \le 8.25 \cdot c_0^2  \cdot \frac{D_\delta}{r_\delta} \cdot \rho(\| z^{n,0} - z^{n-1,0}\|_M; z^{n,0}) + \sqrt{2}c_0 d^H_\delta \ .
	\end{equation}
\end{lemma}

\proof{Proof.}
	Applying Lemma \ref{lm change of norm} with $(\calX,\calS,\calY):=(\calX^\star,\calS^\star,\calY^\star)$ and Lemma \ref{lm efeas bound dist} directly yields:
	\begin{equation}\label{eq L C 2}
		\begin{aligned}
			\dist_M(z^{n,0},\calZ^\star) & \ \le  \sqrt{2}c_0\cdot \dist(w^{n,0},\calW^\star)
			\\
			                             & \
			\le  \frac{3\sqrt{2} c_0 D_\delta}{r_\delta}\cdot  \max\{\dist(w^{n,0},V),\dist(w^{n,0},K)\}+  \frac{\sqrt{2}c_0 d^H_\delta}{\delta} \cdot \max\{\gap(w^{n,0}),\delta\} \ .
		\end{aligned}
	\end{equation}

	\noindent We consider two cases, depending on whether $\gap(w^{n,0})\ge \delta$ or $\gap(w^{n,0})<\delta$. We first consider the case where $\gap(w^{n,0})\ge \delta$.  From Lemma \ref{lm use gap to bound error in w} and \eqref{eq L C 2} it follows that
	\begin{equation}\label{eq L C 3}
		\dist_M(z^{n,0},\calZ^\star)
		\le \left( \frac{3\sqrt{2} c_0^2 D_\delta}{r_\delta} +  \frac{4c_0^2 d^H_\delta}{\delta}
		\cdot   \dist(0,\calW^\star)
		\right)
		\cdot  \rho(\| z^{n,0} - z^{n-1,0}\|_M ; z^{n,0} ) \ .
	\end{equation}
	Furthermore, applying Lemma \ref{lm r_delta over delta} yields $
		\dist(0,\calW^\star) \le \max_{w\in\calW^\star}\|w\| \le \delta/r_\delta$ for any $\delta > 0$, and  substituting this into \eqref{eq L C 3} yields
	\begin{equation}\label{eq L C 4}
		\dist_M(z^{n,0},\calZ^\star)
		\le \left( \frac{3\sqrt{2} c_0^2 D_\delta}{r_\delta} +  \frac{4c_0^2 d^H_\delta}{r_\delta}
		\right)
		\cdot  \rho(\| z^{n,0} - z^{n-1,0}\|_M ; z^{n,0} ) \ .
	\end{equation}
	We can also have $d^H_\delta \le D_\delta$ and $3\sqrt{2} + 4 \le 8.25$, which when combined with \eqref{eq L C 4} shows that
	\begin{equation}\label{eq L C 5}
		\begin{aligned}
			\dist_M(z^{n,0},\calZ^\star) \le 8.25 \cdot c_0^2  \cdot \frac{D_\delta}{r_\delta} \cdot \rho(\| z^{n,0} - z^{n-1,0}\|_M; z^{n,0}) \ ,
		\end{aligned}
	\end{equation}
	which proves the result in this case.

	Let us now consider the case where $\gap(w^{n,0}) < \delta$. From Lemma \ref{lm use gap to bound error in w} and \eqref{eq L C 2} it follows that
	\begin{equation}\label{eq L C 6}
		\begin{aligned}
			\dist_M(z^{n,0},\calZ^\star) \le 3\sqrt{2} c_0^2 \cdot \frac{ D_\delta}{r_\delta} \cdot  \rho(\| z^{n,0} - z^{n-1,0}\|_M; z^{n,0}) + \sqrt{2}c_0 d^H_\delta \ ,
		\end{aligned}
	\end{equation}
	which proves the result in this case.	Depending on the case, we obtain either \eqref{eq L C 5} or \eqref{eq L C 6}, either of which implies \eqref{eq thm L C}.
\Halmos\endproof

At long last we are now in position to complete the proof of Theorem \ref{thm overall complexity clp}.
\proof{Proof of Theorem \ref{thm overall complexity clp}.}
	From Lemma \ref{lm use gap to bound error in w} it follows that $\econs(w^{n,0})\le\eps_{\mathrm{cons}}$ and $\egap(w^{n,0})\le\eps_{\mathrm{gap}}$ if
	\begin{equation}\label{eq overall complexity clp 1}
		\rho(\| z^{n,0} - z^{n-1,0}\|_M ; z^{n,0} ) \le \frac{1}{c_0}\left(
		\eps_{\mathrm{cons}}\wedge
		\frac{\sqrt{2}\eps_{\mathrm{gap}}}{4 \cdot \dist(0,\calW^\star)}
		\right) \ .
	\end{equation}
	It follows from the choice of step-sizes in the theorem and the definition of $c_0$ in \eqref{icemelt} that $c_0 = \sqrt{\kappa}$.

	In the proof of Lemma \ref{thm L C} we see that \eqref{eq thm L C} holds for any $\delta > 0$, so Theorem \ref{thm: complexity of PDHG with adaptive restart} can be applied since the condition \eqref{eq restart L C condition} is satisfied using \eqref{eq thm L C} with $\condG = \frac{8.25 c_0^2 D_\delta}{r_\delta}$ and $\condC = \sqrt{2} c_0 d^H_\delta$. Therefore it follows from Theorem \ref{thm: complexity of PDHG with adaptive restart} that $T$ satisfies
	\begin{equation}\label{eq overall complexity clp 2}\begin{aligned}
			T & \le 12 \left(8.25 \cdot c_0^2\cdot \frac{D_\delta}{r_\delta} \right) \cdot  \ln\left(\frac{12 c_0\cdot \dist_M(z^{0,0},\calZ^\star)}{\eps_{\mathrm{cons}}\wedge\frac{\sqrt{2}\eps_{\mathrm{gap}}}{4 \cdot \dist(0,\calW^\star)}}\right)  +  \frac{18 \sqrt{2}c_0^2d^H_\delta}{\eps_{\mathrm{cons}}\wedge\frac{\sqrt{2}\eps_{\mathrm{gap}}}{4 \cdot \dist(0,\calW^\star)}}                                                                                                                                                                                          \\ \\
			  & = 12  \left(8.25 \cdot \kappa \cdot \frac{D_\delta}{r_\delta} \right) \cdot  \ln\left(\frac{12 \sqrt{\kappa} \cdot \dist_M(z^{0,0},\calZ^\star)}{\eps_{\mathrm{cons}}\wedge\frac{\sqrt{2}\eps_{\mathrm{gap}}}{4 \cdot \dist(0,\calW^\star)}} \right)  
			+  \frac{18 \sqrt{2}\cdot \kappa \cdot d^H_\delta}{\eps_{\mathrm{cons}}\wedge\frac{\sqrt{2}\eps_{\mathrm{gap}}}{4 \cdot \dist(0,\calW^\star)}}                                                                                                                                                                            \\ \\
			  & \le  12 \left(8.25  \kappa \cdot \frac{D_\delta}{r_\delta} \right) \cdot  \ln\left(\frac{12\sqrt{2} \kappa  \cdot \dist(0,\calW^\star)}{\eps_{\mathrm{cons}}\wedge\frac{\sqrt{2}\eps_{\mathrm{gap}}}{4 \cdot \dist(0,\calW^\star)}} \right)  
			+  \frac{18\sqrt{2}  \kappa \cdot d^H_\delta}{\eps_{\mathrm{cons}}\wedge\frac{\sqrt{2}\eps_{\mathrm{gap}}}{4 \cdot \dist(0,\calW^\star)}} \ ,
		\end{aligned}\end{equation}
	where the equality uses $c_0=\sqrt{\kappa}$, and the second inequality uses $\dist_M(z^{0,0},\calZ^\star) \le  \sqrt{2}c_0 \cdot \dist(w^{0,0},\calW^\star)\le  \sqrt{2}c_0 \cdot \dist(0,\calW^\star) = \sqrt{2\kappa}\cdot \dist(0,\calW^\star)$ from Lemma \ref{lm change of norm}, applied with $(\calX,\calS,\calY):=(\calX^\star,\calS^\star,\calY^\star)$, and Proposition \ref{lm:distance_to_optimal_initial}.  Last of all, notice that $12 \times 8.25\le 99$, $12\sqrt{2}\le 17$ and $18\sqrt{2} \le 26$. Substituting these bounds into \eqref{eq overall complexity clp 2} yields \eqref{eq overall complexity} and completes the proof.
\Halmos\endproof

\section{Linear convergence of rPDHG for linear optimization}\label{sec:complexity_lp}

In the case of linear optimization problems (instances of \eqref{pro: general primal clp} with $K_p = \mathbb{R}^n_+$) we present in this section a global linear convergence bound for rPDHG that structurally improves on the bound in \cite{xiong2023computational}. Recall that $\mathcal{L}=\operatorname{Null}(A)$ and $\mathcal{L}^{\bot}=\operatorname{Im}(A^\top)$ are the linear subspaces associated with $V_p$ and $V_d$, respectively. Our analysis uses the ``best suboptimal extreme point gap'' $\bar \delta$ whose formal definition we now state.\medskip
\begin{definition}[Best suboptimal extreme point gap]\label{def:best_suboptimal_gap}
	Let $\ep_\calF$ denote the set of extreme points of $\calF$. The best suboptimal extreme point gap $\bar{\delta}$ is defined as follows:
	\begin{equation}\label{eqdef:best_suboptimal_gap}
		\bar{\delta}:= \left\{
		\begin{array}{ll}
			\min\{\gap(w): w\in \ep_\calF\setminus \calW^\star\} & \quad \text{ if $\ep_\calF\setminus \calW^\star \neq \emptyset$ } \\
			+\infty                                              & \quad \text{ if $\ep_\calF\setminus \calW^\star = \emptyset$ .}   \\
		\end{array}
		\right.
	\end{equation}
\end{definition}
\noindent We now present our global linear convergence result for rPDHG on LP instances.
\begin{theorem}\label{thm overall complexity lp} Suppose that
	\eqref{pro: general primal clp} is a linear optimization instance ($K_p = \mathbb{R}^n_+$) and $c\in\mathcal{L}$, and let $\bar\delta$ be as defined in \eqref{eqdef:best_suboptimal_gap}. Under Assumption \ref{assump:striclyfeasible}, suppose that Algorithm \ref{alg: PDHG with restarts} (rPDHG) is run starting from $z^{0,0} = (x^{0,0},y^{0,0} ) = (0,0)$ using the $\beta$-restart condition with $\beta := 1/e$, and the standard step-sizes $\sigma=\sigma_0$ and $\tau = \tau_0$ as in \eqref{eq:standard_step_size}.   
	Given $\eps_{\mathrm{cons}},\eps_{\mathrm{gap}}>0$, let $T$ be the total number of \textsc{OnePDHG} iterations that are run in order to obtain $n$ for which $w^{n,0}=(x^{n,0},s^{n,0})$ satisfies $\econs(w^{n,0})\le \eps_{\mathrm{cons}}$ and $\egap(w^{n,0})\le \eps_{\mathrm{gap}}$.
	Then
	\begin{equation}\label{eq overall complexity LP}
		T \le 133 \kappa \cdot \left( \inf_{0 < \delta \le \bar{\delta}}\frac{D_{\delta}}{r_{\delta}}\right) \cdot   \left[\ln\big(17 \kappa \cdot \dist(0,\calW^\star) \big) + 
		\ln\left(\frac{1}{\eps_{\mathrm{cons}}} \vee 
		\frac{2\sqrt{2} \dist(0,\calW^\star)}{\eps_{\mathrm{gap}}}\right) \right] \,  . 
	\end{equation}
\end{theorem}

Theorem \ref{thm overall complexity lp} is the formal statement of Main Result~\ref{mainres:lp} presented in Section \ref{sec:intro_preview}. The computational bound \eqref{eq overall complexity LP} depends on the target tolerances only through the logarithm term (in contrast to the bound in Theorem \ref{thm overall complexity clp}). Thus if there exists a $\delta$-level set for some $\delta \in (0,\bar{\delta}]$ such that the corresponding ratio $\frac{D_{\delta}}{r_{\delta}}$ is not too large, then rPDHG will have fast linear convergence for that LP instance. In addition to $\kappa$, the coefficient outside the logarithm term in \eqref{eq overall complexity LP} is controlled by the best level-set ratio up to the first suboptimal extreme point gap. We thus interpret Theorem \ref{thm overall complexity clp} together with Theorem \ref{thm overall complexity lp}: the former states a very flexible result based on the best level set that trades off the geometric condition measures, whereas the latter presents a global linear convergence guarantee in the LP case.


Note that the bound \eqref{eq overall complexity LP} is decreasing in $\bar\delta$, and so a smaller best suboptimal extreme point gap $\bar\delta$ can lead to a worse bound. Geometrically, a small $\bar\delta$ means that a suboptimal extreme point is nearly optimal, so the instance is close to having additional optimal extreme points (a small data perturbation would result in having additional optima). In this sense Theorem \ref{thm overall complexity lp} suggests that being close to having more optimal solutions can hurt the convergence rate, even though multiple optimal solutions themselves do not necessarily do so. This is similar in spirit to the observation of \cite{lu2023geometry} for PDHG without restarts, where degeneracy itself does not hurt the convergence rate, but being close to degeneracy does.  

Of course Theorem \ref{thm overall complexity lp} is not the first linear convergence result for PDHG on LP; \cite{applegate2023faster} uses the Hoffman constant of the KKT system to characterize linear convergence, but the global Hoffman constant is often overly conservative and hard to analyze \cite{lu2023geometry,xiong2023computational}. \cite{lu2023geometry} studies PDHG without restarts and uncovers a refined two-phase behavior. \cite{xiong2023computational,xiong2023relation} provide linear convergence guarantees based on limiting error ratios and LP sharpness. Theorem \ref{thm overall complexity lp} shows that the linear convergence rate can be bounded using only $\kappa$ and the level-set condition measures $D_\delta$ and $r_\delta$.

Theorem \ref{thm overall complexity lp} is stated for the standard step-sizes $\sigma_0$ and $\tau_0$ as in \eqref{eq:standard_step_size}. If using other ratios between $\sigma$ and $\tau$, applying Theorem \ref{thm overall complexity lp} on the corresponding reweighted instance suffices; see Remark \ref{rmk:step-size-vs-reweighting}.

\subsection{An illustrative family of LP instances}\label{subsec:illustrative_lp_instance}

In this section we study a simple family of LP instances to illustrate several aspects of the advantages/disadvantages of the bounds in Theorem \ref{thm overall complexity lp} and Theorem \ref{thm overall complexity clp}.  
We consider the following LP instances parameterized by $\nu \ge 0$:
\Equationvalidatefalse 
\begin{equation}\tag{$P_\nu$}\label{pro:example_lp_validate}
	\begin{array}{c}
		\min_{x = (x_1,x_2,x_3) \in \mathbb{R}^3_+} \frac{2 + \nu }{10}\cdot x_1 + x_2 + (1 + \nu)x_3 \ \ \text{s.t. } -10x_1 + x_2 + x_3  = 1  \ .
	\end{array}
\end{equation}
\Equationvalidatetrue%
For this family of instances we compute explicit values of the bounds in Theorems \ref{thm overall complexity lp} and \ref{thm overall complexity clp} by setting $\eps_{\,\mathrm{cons}}=\eps_{\,\mathrm{gap}} = \bar\eps$ for a given uniform tolerance value $\bar\eps$. 

We first examine the complexity bound functions $T_\delta$ defined in \eqref{eq overall complexity} and derived in Theorem \ref{thm overall complexity clp} for the instance \eqref{pro:example_lp_validate} with $\nu=10^{-4}$. Figure \ref{fig:Tdelta and actual T} presents plots of $T_\delta$ for this instance for $\delta=10^{-1}$ and $\delta=10^{-8}$ as a function of the uniform tolerance value $\bar\eps$ (left subfigure), as well as the lower envelope $\inf_{\delta>0}T_\delta$ and the actual iteration counts of rPDHG (right subfigure). The figure illustrates that different level sets can yield better bounds at different target tolerances. For moderate tolerances, a larger and better-shaped level set results in a more informative bound; for smaller tolerances, smaller level sets become relevant and the bound transitions to the linear-convergence regime. Although the theoretical envelope is off by a large constant factor from the actual iteration counts of rPDHG for this instance, the envelope captures the qualitative transition of the actual iteration counts better than a purely asymptotic characterization.

\begin{figure}[htbp]
	\centering
	\includegraphics[width=0.65\textwidth]{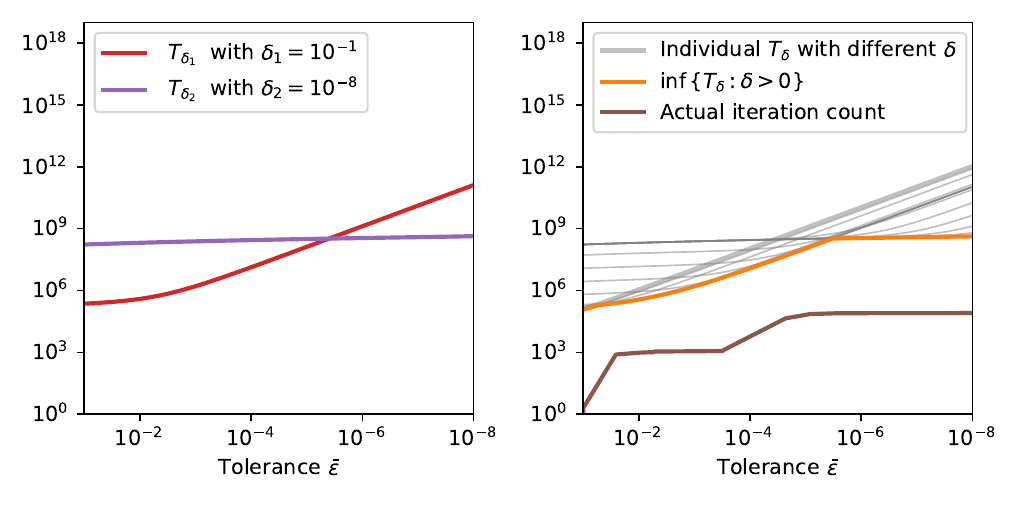}\vspace{-10pt}
	\caption{(Left) Plots of $T_\delta$ for $\delta=10^{-1}$ and $\delta=10^{-8}$, as a function of the required tolerance $\bar\eps = \eps_{\mathrm{cons}} = \eps_{\mathrm{gap}} $ for the linear optimization instance \eqref{pro:example_lp_validate} with $\nu=10^{-4}$. (Right) Plots of the actual iteration counts of rPDHG, and the lower envelope $\inf_{\delta >0} T_\delta$.}\label{fig:Tdelta and actual T}
\end{figure}

For $\nu=0$ the LP instance \eqref{pro:example_lp_validate} has multiple primal optimal solutions, namely $x^\star =(0, \alpha, 1-\alpha)$ for all $\alpha \in [0,1]$. Figure \ref{fig:two_thms_easy_hard_LP_paper} compares the bounds from Theorems \ref{thm overall complexity clp} and \ref{thm overall complexity lp} with the actual number of iterations required to satisfy $\econs(w)\le \bar{\eps}$ and $\egap(w)\le \bar{\eps}$, for both $\nu=0$ and $\nu=10^{-4}$. For $\nu=0$, the bound in Theorem \ref{thm overall complexity clp} does not assert linear convergence because the optimal solution set is not a singleton, whereas the bound in Theorem \ref{thm overall complexity lp} presents a linear convergence guarantee. For $\nu=10^{-4}$, Theorem \ref{thm overall complexity lp} presents a linear rate, but the rate does not capture the early-stage behavior as well as the more flexible bound in Theorem \ref{thm overall complexity clp}.

\begin{figure}[htbp]
	\centering
	\includegraphics[width=0.65\linewidth]{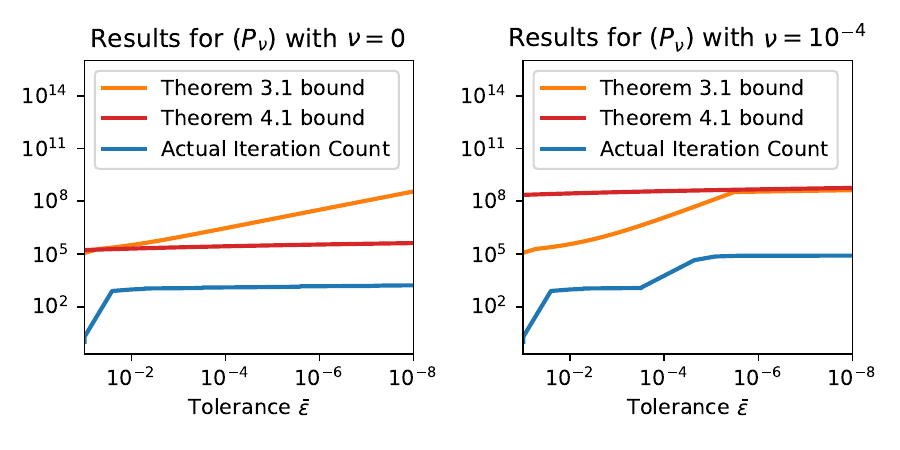}\vspace{-10pt}
	\caption{Plots of the guaranteed bounds on the number of iterations of rPDHG and the actual number of iterations of rPDHG for the LP instance \eqref{pro:example_lp_validate} for $\nu=0$ (left) and $\nu=10^{-4}$ (right). }\label{fig:two_thms_easy_hard_LP_paper}
\end{figure}

In Table \ref{tabletable} we show the computed values of the quantities
that help explain this behavior. The value of $\bar\delta$ is much
larger for $\nu=0$ than for $\nu=10^{-4}$, while the two instances
have the same value of $\kappa$ and nearly identical values of
$\max_{w\in\calW^\star}\|w\|$. However, the value of
$\inf_{0<\delta\leq\bar\delta}D_\delta/r_\delta$ is $115.7538$ for
$\nu=0$ and $158401.2392$ for $\nu=10^{-4}$. This significant difference is also reflected in the bounds and
actual iteration counts in Figure
\ref{fig:two_thms_easy_hard_LP_paper}.
 
\begin{table}[htbp]
	\centering
	\begin{tabular}{l|cccccc}
		& $\kappa$
		& $\bar{\delta}$
		& $D_{\bar{\delta}}$
		& $r_{\bar{\delta}}$
		& {\small$\displaystyle
			\inf_{0<\delta\leq\bar{\delta}}
			\frac{D_\delta}{r_\delta}$}
		& {\small$\displaystyle
			\max_{w\in\calW^\star}\|w\|$}
		\\ \hline
		$\nu=0$
		& 1
		& 1.02
		& 10.5419
		& 0.091071
		& 115.7538
		& 10.2489
		\\
		$\nu=0.0001$
		& 1
		& 0.0001
		& 1.4143
		& 0.0000089285
		& 158401.2392
		& 10.2489
	\end{tabular}\vspace{5pt}
	\caption{Values of some quantities of interest for the family of
	LP instances \eqref{pro:example_lp_validate} for $\nu=0$ and
	$\nu=0.0001$.}
	\label{tabletable}
\end{table}


\subsection{Proof of Theorem \ref{thm overall complexity lp}}\label{subsec:proof_theorem_lp}

We now prove Theorem \ref{thm overall complexity lp}. Linear optimization problems enjoy a ``sharpness'' property that is not guaranteed for the more general conic optimization problem \eqref{pro: general primal clp}.  Similar in spirit to \cite{xiong2023computational} we define the (primal-and-dual) \textit{PD sharpness} $\mu$ as follows:
\begin{equation}\label{eqdef:LPsharpness}
	\mu:= \inf_{w \in \calF \setminus \calW^\star}\frac{\dist(w, V \cap \{w: \gap(w)= 0\})}{\dist(w,\calW^\star)} \ .
\end{equation}

Recalling the definition of the best suboptimal extreme point gap $\bar\delta$ from Definition \ref{def:best_suboptimal_gap}, a key property that characterizes $\mu$ using $\bar\delta$ is the following lemma, which was shown in \cite{xiong2023computational}, albeit using different notation.
\begin{lemma}\label{lm:LPsharpness} {\bf (essentially Theorem 5.2 of \cite{xiong2023computational})}
	Under Assumption~\ref{assump:striclyfeasible}, for any
$\delta\in(0,\bar{\delta}]$ it holds that  
	\begin{equation}\label{eq:LPsharpness_keyslice}
		\mu = \inf_{w \in \calF \cap \{w: \gap(w) = \delta\}}\frac{\dist(w, V \cap \{w: \gap(w)= 0\})}{\dist(w,\calW^\star)} \ .
	\end{equation}
\end{lemma}
\noindent The following lower bound on $\mu$ in terms of $\delta$ and $d^H_\delta$ is a direct implication of Lemma \ref{lm:LPsharpness}.
\begin{lemma}\label{lm:lowerbound_mu}
	Under the hypothesis of Lemma \ref{lm:LPsharpness}, for any $\delta\in(0,\bar{\delta}]$ it holds that
	\begin{equation}\label{eq:lowerbound_mu}
		\mu \ge  \frac{\delta}{d^H_\delta}\cdot \frac{1}{\sqrt{\|P_{\mathcal{L}}(c)\|^2 + \|q\|^2}} \ .
	\end{equation}
\end{lemma}
\proof{Proof.}
	For $ w \in V$ satisfying $\gap(w) = \delta$, the numerator of \eqref{eq:LPsharpness_keyslice} has the closed form 		$$
		\dist(w, V \cap \{w: \gap(w)= 0\}) =  \frac{\delta}{\|P_{\mathcal{L}\times\mathcal{L}^{\bot}}([c, q])\|} = \frac{\delta}{\sqrt{\|P_{\mathcal{L}}(c)\|^2 + \|q\|^2}} \ .
	$$
	Also, from the definition of $d^H_\delta$ in Definition \ref{def distance to optima} we have:
	$$
		d^H_\delta = \max_{w \in \calF \cap\{w:\gap(w)\le \delta\}} \dist(w,\calW^\star) \ge \max_{w \in \calF \cap\{w:\gap(w) = \delta\}} \dist(w,\calW^\star) \ ,
	$$
	whereby from
	\eqref{eq:LPsharpness_keyslice}  we have for any $\delta \in (0, \bar{\delta}]$ that
	$$
		\mu =  \frac{\delta}{\sqrt{\|P_{\mathcal{L}}(c)\|^2 + \|q\|^2}} \cdot \inf_{w \in \calF \cap \{w: \gap(w) = \delta\}}\frac{1}{\dist(w,\calW^\star)}  \ge
		\frac{\delta}{\sqrt{\|P_{\mathcal{L}}(c)\|^2 + \|q\|^2}} \cdot \frac{1}{d^H_\delta} \ ,
	$$
	which is exactly \eqref{eq:lowerbound_mu}.
\Halmos\endproof

\begin{lemma}\label{lm c123 sharp LP}
	Suppose that $c\in \mathcal{L}$. For any $\delta \in ( 0,\bar{\delta}]$ and any $w:=(x,s)$ it holds that
	\begin{equation}\label{eq lm c123 sharp LP}
		\dist(w,\calW^\star) \le 	\frac{ d^H_\delta   }{\delta} \cdot\gap(w) +  \frac{5D_\delta}{r_\delta} \cdot \max\{ \dist(w,K) , \dist(w, V)\} \ .
	\end{equation}
\end{lemma} 
\proof{Proof.}
	We first consider the case when $\gap(w)\le \delta$. Let $\bar{w} := P_V(w)$. Then because $c \in \mathcal{L}$ and $q \in \mathcal{L}^{\bot}$, we have $\gap(w) = \gap(\bar{w})$ and hence $\gap(\bar{w}) = \gap(w)\le \delta$.
	Let $(w_\delta, r_\delta)$ be the conic center and conic radius of $\calW_\delta$ as defined in Definition \ref{def radius}. 
	Suppose first that $\bar{w}\notin K$. Because $\bar{w} \in V$, Lemma \ref{lm error bound R r} yields
	\begin{equation}\label{eq lm c123 sharp LP 1}
		\frac{\| \bar{w} - \calF(\bar{w};w_\delta) \|}{\dist(\bar{w}, K)} \le \frac{\| w_\delta - \calF(\bar{w};w_\delta) \|}{r_\delta} \ .
	\end{equation}

	Also, because $\bar{w}$, $\calF(\bar{w};w_\delta)$, and $w_\delta$ are collinear, it follows that
	\begin{equation}\label{eq lm c123 sharp LP 2}
		\begin{aligned}
			\big| \gap(\bar{w}) - \gap(\calF(\bar{w};w_\delta)) \big| & = \big|
			\gap(\calF(\bar{w};w_\delta)) - \gap(w_\delta)
			\big|
			\cdot \frac{\| \bar{w} - \calF(\bar{w};w_\delta)\|}{\|\calF(\bar{w};w_\delta) - w_\delta\|} \\
			                                                          & \le  \big|
			\gap(\calF(\bar{w};w_\delta)) - \gap(w_\delta)
			\big|
			\cdot \frac{\dist(\bar{w},K)}{r_\delta} \le \delta \cdot  \frac{\dist(\bar{w},K)}{r_\delta} \ ,
		\end{aligned}
	\end{equation}
	where the first inequality uses \eqref{eq lm c123 sharp LP 1} and the second inequality follows since $w_\delta$ and $\calF(\bar{w};w_\delta)$ are in $\calW_\delta$. 
	If $\bar{w}\in K$, then $\bar w\in\calW_\delta$ and we define $\calF(\bar w;w_\delta):=\bar w$. The inequality \eqref{eq lm c123 sharp LP 2} then also holds immediately.

	Next, because $\calF(\bar{w};w_\delta) \in \calW_\delta$, then from the definition of the PD sharpness $\mu$ in \eqref{eqdef:LPsharpness} and the lower bound on $\mu$ from Lemma \ref{lm:lowerbound_mu}, we have:
	\begin{equation}\label{eq lm c123 sharp LP 3}
		\begin{aligned}
			\dist(\calF(\bar{w};w_\delta),\calW^\star) & \ \le   \frac{\dist(\calF(\bar{w};w_\delta), V \cap \{w:\gap(w)  = 0\}         ) }{ \mu} \\
			                                           & \ =
			\frac{\gap(\calF(\bar{w};w_\delta))}{\sqrt{\|c\|^2 + \|q\|^2}\cdot \mu} \le \frac{ d^H_\delta   }{\delta} \cdot \gap(\calF(\bar{w};w_\delta))\ .
		\end{aligned}
	\end{equation}

	\noindent We also have the following bound on $\dist(\bar{w},\calW^\star)$:
	\begin{equation}\label{eq lm c123 sharp LP 4}
		\begin{aligned}
			\dist(\bar{w},\calW^\star) & \le \dist(\calF(\bar{w};w_\delta),\calW^\star)  + \| \calF(\bar{w};w_\delta) - \bar{w}\| \\  & \le  \frac{ d^H_\delta   }{\delta} \cdot \gap(\calF(\bar{w};w_\delta))  +\frac{D_\delta}{r_\delta} \cdot \dist(\bar{w},K) \ ,
		\end{aligned}
	\end{equation}
	where the second inequality uses \eqref{eq lm c123 sharp LP 3} as well as the inequality $\|   \calF(\bar{w};w_\delta) - \bar{w}\| \le \frac{D_\delta}{r_\delta} \cdot \dist(\bar{w},K)$, which is immediate when $\dist(\bar w,K)=0$ and otherwise follows from Lemma \ref{lm error bound R r}.

	From \eqref{eq lm c123 sharp LP 2} we have
	\begin{equation}\label{eq lm c123 sharp LP 6}
		\begin{aligned}
			\gap(\calF(\bar{w};w_\delta)) & \ \le  \gap(\bar{w}) + \delta\cdot \frac{\dist(\bar{w},K)}{r_\delta} \ ,
		\end{aligned}
	\end{equation}
	whereby \eqref{eq lm c123 sharp LP 4} implies
	\begin{equation}\label{eq lm c123 sharp LP 7}
		\begin{aligned}
			\dist(\bar{w},\calW^\star) & \le  \frac{ d^H_\delta   }{\delta} \cdot \left( \gap(\bar{w}) + \delta\cdot \frac{\dist(\bar{w},K)}{r_\delta}    \right)  +\frac{D_\delta}{r_\delta} \cdot \dist(\bar{w},K) \\
			                           & \le \frac{ d^H_\delta   }{\delta} \cdot\gap(\bar{w}) +  \frac{2D_\delta}{r_\delta} \cdot \dist(\bar{w},K) \ ,
		\end{aligned}
	\end{equation}
	in which the second inequality uses $d^H_\delta \le D_\delta$ from Lemma \ref{lm D ge r}.

	Finally, we use the upper bound on $\dist(\bar{w},\calW^\star)$ to obtain an upper bound on $\dist(w,\calW^\star)$:
	\begin{equation}\label{eq lm c123 sharp LP 8}\begin{aligned}
			\dist(w,\calW^\star) & \le 	\dist(\bar{w},\calW^\star) + \| \bar{w} - w\|  = 	\dist(\bar{w},\calW^\star) +\dist(w, V) \\ & \le \frac{ d^H_\delta   }{\delta} \cdot\gap(w) + \frac{2D_\delta}{r_\delta} \cdot \dist(\bar w, K) + \dist(w, V)
			\\ & \le \frac{ d^H_\delta   }{\delta} \cdot\gap(w) + \frac{2D_\delta}{r_\delta} \cdot \left(\dist(w, K) + \dist(w,V) \right) + \dist(w, V)
			\\ & \le 	\frac{ d^H_\delta   }{\delta} \cdot\gap(w) +  \frac{5D_\delta}{r_\delta} \cdot \max\{ \dist(w,K) , \dist(w, V)\} \ ,
		\end{aligned}\end{equation}
	where the second inequality uses \eqref{eq lm c123 sharp LP 7} and $\gap(\bar{w}) = \gap(w)$, the third inequality uses $\dist(\bar{w},K) \le \dist(w,K) + \|w - \bar{w}\| = \dist(w,K)  + \dist(w,V)$, and the fourth inequality uses $D_\delta \ge r_\delta$ from Lemma \ref{lm D ge r}. This proves \eqref{eq lm c123 sharp LP} in the case when $\gap(w) \le \delta$.

	Let us now consider the case where  $\delta \le \gap(w)$.  Here we will make use of Proposition \ref{lm:convexity_ineq} to complete the proof.  Let $w^\star := P_{\calW^\star}(w) = \arg\min_{\bar{w}\in\calW^\star}\|\bar{w} - w\|$ and define $w_t := w^\star + t\cdot (w - w^\star)$ for $t\in[0,\infty)$. Then define the following functions of $t$ :
	$$
		f(t):= \dist(w_t,\calW^\star) \ , \text{ and }g(t):=\frac{ d^H_\delta   }{\delta} \cdot\gap(w_t) +  \frac{5D_\delta}{r_\delta} \cdot \max\{ \dist(w_t,K) , \dist(w_t, V)\}  \ .
	$$
	Then $f(t)$ is a nonnegative  linear  function on $[0,\infty)$, and $f(0) = 0$. And $g(t)$ is convex and nonnegative on $[0,\infty)$, and $g(0) = 0$. In addition, because $\gap(\cdot)$ is a linear function and $\gap(w_t) = t \cdot \gap(w)$, then setting $u := \delta / \gap(w)$ we obtain $ \gap(w_u) = u\cdot \gap(w) = \delta$. We can then invoke \eqref{eq lm c123 sharp LP 8} using $w_u$ in the place of $w$, which yields $g(u) \ge f(u)$. Now it follows from Proposition \ref{lm:convexity_ineq} with $v := 1\ge  u$ that  $g(1) \ge f(1)$, which is precisely \eqref{eq lm c123 sharp LP 8} in the case $\delta \le \gap(w)$, and completes the proof.
\Halmos\endproof

\begin{lemma}\label{thm L sharp LP} Suppose that $c\in\mathcal{L}$. Under Assumption \ref{assump:striclyfeasible}, suppose that Algorithm \ref{alg: PDHG with restarts} (rPDHG) is run starting from $z^{0,0} = (x^{0,0},y^{0,0} ) = (0,0)$, and the step-sizes $\sigma$ and $\tau$ satisfy the step-size inequality \eqref{eq:general_stepsize}. Then for every $n\ge 1$ and any $\delta  \in (0,\bar{\delta}]$ it holds that
	\begin{equation}\label{eq thm L sharp LP}
		\dist_M(z^{n,0},\calZ^\star) \le    (5\sqrt{2} + 4) \cdot c_0^2 \cdot \frac{ D_\delta}{r_\delta}  \cdot \rho(\| z^{n,0} - z^{n-1,0}\|_M; z^{n,0}) \ .
	\end{equation}
\end{lemma}

\proof{Proof of Lemma \ref{thm L sharp LP}.}
	Directly using Lemma \ref{lm change of norm} and Lemma \ref{lm c123 sharp LP} we have
	\begin{equation}\label{eq L sharp LP 1}
		\begin{aligned}
			\dist_M(z^{n,0},\calZ^\star) & \le  \sqrt{2}c_0 \cdot	\dist(w^{n,0},\calW^\star)                                                                                                                           \\
			                             & \le  \	\frac{ \sqrt{2}c_0 d^H_\delta   }{\delta} \cdot\gap(w^{n,0}) +  \frac{5\sqrt{2}c_0  D_\delta}{r_\delta} \cdot \max\{ \dist(w^{n,0},K) , \dist(w^{n,0}, V)\} \ ,
		\end{aligned}
	\end{equation}
	and applying Lemma \ref{lm use gap to bound error in w} yields
	\begin{equation}\label{eq L sharp LP 2}			\begin{aligned}
			\dist_M(z^{n,0},\calZ^\star) \le \  & \left(  \frac{5\sqrt{2}c_0^2 D_\delta}{r_\delta}  + \frac{4 c_0^2 d^H_\delta   }{\delta} \cdot \dist(0,\calW^\star)    \right) \cdot \rho(\| z^{n,0} - z^{n-1,0}\|_M; z^{n,0}) \ .
		\end{aligned}
	\end{equation}
	From Lemma \ref{lm r_delta over delta} we have
	$r_\delta \cdot \dist(0,\calW^\star) \le r_\delta \cdot \max_{w\in\calW^\star}\|w\| \le \delta$, whereby it follows from \eqref{eq L sharp LP 2} that
	\begin{equation}\label{eq L sharp LP 3}
		\dist_M(z^{n,0},\calZ^\star)
		\le \left( \frac{5\sqrt{2} c_0^2 D_\delta}{r_\delta} +  \frac{4c_0^2 d^H_\delta}{r_\delta}
		\right)
		\cdot  \rho(\| z^{n,0} - z^{n-1,0}\|_M ; z^{n,0} ) \ .
	\end{equation}
	Last of all, notice that $d^H_\delta \le D_\delta$ (from Lemma \ref{lm D ge r}), and so \eqref{eq thm L sharp LP} follows from \eqref{eq L sharp LP 3}.
\Halmos\endproof

We now prove Theorem \ref{thm overall complexity lp}. 

\proof{Proof of Theorem \ref{thm overall complexity lp}.}
	From Lemma \ref{lm use gap to bound error in w} it follows that $\econs(w^{n,0})\le\eps_{\mathrm{cons}}$ and $\egap(w^{n,0})\le\eps_{\mathrm{gap}}$ if
	\begin{equation}\label{eq overall complexity lp 1}
		\rho(\| z^{n,0} - z^{n-1,0}\|_M ; z^{n,0} ) \le \frac{1}{c_0}\left(
		\eps_{\mathrm{cons}}\wedge
		\frac{\sqrt{2}\eps_{\mathrm{gap}}}{4 \cdot \dist(0,\calW^\star)}
		\right) \ .
	\end{equation}
	It follows from the choice of step-sizes in the theorem and the definition of $c_0$ in \eqref{icemelt} that $c_0 = \sqrt{\kappa}$.

	In the proof of Lemma \ref{thm L sharp LP} we see that \eqref{eq thm L sharp LP} holds for any $\delta \in (0,\bar\delta]$, so Theorem \ref{thm: complexity of PDHG with adaptive restart} can be applied since the condition \eqref{eq restart L C condition} is satisfied using \eqref{eq thm L sharp LP} with $\condG = (5\sqrt{2} + 4) \cdot c_0^2 \cdot \frac{ D_\delta}{r_\delta}$ and $\condC = 0$. Therefore it follows from Theorem \ref{thm: complexity of PDHG with adaptive restart} that $T$ satisfies
	\begin{equation}\label{eq overall complexity lp 2}
		T \le 12 \cdot \left((5\sqrt{2}+4)\cdot c_0^2\cdot \frac{D_\delta}{r_\delta} \right) \cdot  \ln\left(\frac{12 c_0\cdot  \dist_M(z^{0,0},\calZ^\star)}{\min\left\{\eps_{\mathrm{cons}},\frac{\sqrt{2}\eps_{\mathrm{gap}}}{4 \cdot \dist(0,\calW^\star)}\right\}}\right) \ .
	\end{equation}		Here we have $c_0^2 = \kappa$, and $\dist_M(z^{0,0},\calZ^\star) \le  \sqrt{2}c_0 \cdot \dist(w^{0,0},\calW^\star)\le  \sqrt{2}c_0 \cdot \dist(0,\calW^\star) = \sqrt{2\kappa} \cdot \dist(0,\calW^\star)$ from Lemma \ref{lm change of norm} and Proposition \ref{lm:distance_to_optimal_initial}.
	Additionally, $12 \cdot (5\sqrt{2}+4)  \le 133$ and $12\sqrt{2}\le 17$.     Therefore \eqref{eq overall complexity lp 2} yields \eqref{eq overall complexity LP}.
\Halmos\endproof

\section{Summary and research directions}\label{sec:discussion}

In this paper we have studied how the primal-dual level-set geometry informs the convergence of rPDHG for conic linear optimization. We have introduced three geometric condition measures of the primal-dual level sets $\calW_\delta$: the diameter $D_\delta$, the conic radius $r_\delta$, and the Hausdorff distance $d_\delta^H$ to the optimal solution set. We have shown that rPDHG can converge at a fast rate (at least in theory) if there exists a level set $\calW_\delta$ that is close to the optimal solution set and if the ratio $D_\delta/r_\delta$ is small. For LP instances, we have obtained a global linear convergence guarantee governed by the best-conditioned level set $\calW_\delta$ among those for which $\delta$ does not exceed the best non-optimal extreme-point gap. 

After the unpublished manuscript \cite{xiong2024role} was posted, there have appeared several follow-up works that have developed related theoretical results and computational enhancements \cite{xiong2024accessible,xiong2025high,xiong2025new,lin2025pdcs}.  
In particular, \cite{xiong2024accessible}  reports experimental evidence that the iteration bound in Theorem \ref{thm overall complexity lp} is consistent with computational practice for LP instances with unique optima. 

We end this paper with the following short list of open questions for further investigation:

\textbf{1. Extensions to other primal-dual algorithms.} We expect that the use of the level-set geometric condition measures for informing convergence is not limited to just rPDHG. We note that the only properties of rPDHG that were explicitly exploited in our analysis were those in Lemmas \ref{lm: nonexpansive property}, \ref{lm: R in the opt gap convnergence} and \ref{lm: original sublinear PDHG}. And in fact properties similar to these also hold for other primal-dual first-order algorithms such as ADMM and EGM \cite{applegate2023faster,ryu2022large}.  For this reason we expect that much of our analysis can be extended to these other primal-dual first-order methods as well.

\textbf{2. Condition measures and analysis for more general nonlinear programs.} In this paper we studied conic linear optimization problems (CLPs), which are a (rather important) subclass of constrained convex optimization. It would be interesting to explore how the level-set geometry influences the convergence rate of first-order methods for more general convex and nonconvex constrained optimization.

\textbf{3. Improving the level set geometry to speed up convergence of rPDHG.} Theorem \ref{thm overall complexity clp} begs the question of whether there exists a linear transformation of the variables (or equivalently a change in the inner product norm) that improves the geometry of the level sets $\calW_\delta$ and hence makes the instance better conditioned for rPDHG.  Some preliminary work on this was undertaken in the unpublished manuscript \cite{xiong2024role} which showed how knowledge of points on the central path can be used to create such a linear transformation, along with heuristics for implementing this strategy efficiently.  We anticipate that further efforts to develop effective pre-conditioning heuristics could lead to substantial improvements in practice (and in theory).

\textbf{4. Condition measures and analysis for the infeasibility detection problem.} It was shown in \cite{applegate2024infeasibility} that rPDHG can be used to detect infeasibility for LP instances. It would be interesting to study how the geometry of the problem informs the speed of infeasibility detection and how to possibly improve the geometry to enhance practical algorithm performance.

\begin{APPENDICES}
	\SingleSpacedXI

\section{Proof of Lemma \ref{lm: convergence of PHDG without restart}}\label{app:proof_lm_convergence_pdhg_without_restart}

\proof{Proof of Lemma \ref{lm: convergence of PHDG without restart}.}
	The following proof is a variation of the proof of  Lemma 2.1 in \cite{xiong2023computational}. Define $\bar{s}=c - A^\top \bar{y} $ and $\bar{w} = (\bar{x},\bar{s})$. From the definition of $\rho(r;\cdot)$ we have:
	\begin{equation}\label{eq  lm: convergence of PHDG without restart 1}
		L(\bar{x},y) - L(x,\bar{y}) \le r \rho(r;\bar{z}) \ \ \text{for any $z=(x,y) \in B(r;\bar{z})$ . }
	\end{equation}

	We first prove item (\textit{\ref{item_gap1}}.), which is the distance to the affine subspace $V$.
	Let $u:=b-A\bar{x}$. If $u=0$, then $\bar x\in V_p$ and the bound in item (\textit{\ref{item_gap1}}.) holds immediately. Otherwise, define $y:= \bar{y} + \sqrt{\sigma}r \cdot u/ \|u\|$. Set $z := (\bar x, y)$, whereby $ z \in B(r;\bar{z})
	$ and hence from \eqref{eq  lm: convergence of PHDG without restart 1} we have
	$$ r \rho(r;\bar{z}) \ge L(\bar{x},y) - L(\bar x,\bar{y}) = (b-A \bar x)^\top (y - \bar y)  = \sqrt{\sigma}r \|u\| \ , $$ which means $\|u\|=\|A\bar{x} - b\| \le \frac{\rho(r;\bar{z})}{\sqrt{\sigma}}$. Let $\hat{x} \in \arg\min_{x\in V_p}\|x - \bar{x}\|$ and hence $\dist(\bar{x},V_p) = \|\hat{x} - \bar{x}\|$.  Note from the standard optimality conditions that $\hat{x} - \bar{x} \in \operatorname{Im}(A^\top)$. Since
	$$
		\|A\bar{x} - b\| = \|A(\bar{x}-\hat{x})\| \ge \lambda_{\min} \|\bar{x} - \hat{x}\| \ ,
	$$
	then it follows that $\dist(\bar{x},V_p) = \|\hat{x} - \bar{x}\| \le  \frac{\rho(r;\bar{z})}{\sqrt{\sigma} \lambda_{\min}}$. This proves item (\textit{\ref{item_gap1}}.).

	Let us now prove item (\textit{\ref{item_gap2}}.). It holds trivially from the supposition that $\bar x \in K_p$ that $\dist(\bar{x},K_p)  = 0$, and hence we only need to prove  $ \dist(\bar{s},K_d) \le \frac{1}{\sqrt{\tau}} \cdot \rho(r;\bar{z})$. Let us denote $\hat s := P_{K_d}(\bar{s})$ and $d:= \bar s - \hat s$.  Then it follows from the optimality conditions of the projection problem $\hat s =\arg\min_{s \in K_d} \|s - \bar s\|$ that $\hat s \in K_d$, $-d \in K_d^* = K_p$, and $ d ^\top \hat s = 0$.  If $d = 0$ then $\bar s \in K_d$ and the bound in item (\textit{\ref{item_gap2}}.) holds trivially.  If $d \ne0$ then define $x:= \bar{x} - \sqrt{\tau} r  \cdot d / \|d\|$ and set $z := (x, \bar y)$, whereby $ z \in B(r;\bar{z})
	$ and hence from \eqref{eq  lm: convergence of PHDG without restart 1} we have
	\begin{equation}\label{ineq item 3}
		\begin{aligned}
			r \rho(r;\bar{z}) & \ge L(\bar x, \bar y) - L(x,\bar{y})  =  (c - A^\top \bar{y})^\top (\bar{x} - x) = \bar{s}^\top  d \cdot \sqrt{\tau} r / \|d\|  = ( \hat s  + d)^\top d \cdot \sqrt{\tau} r / \|d\| \\
			                  & = d^\top  d\cdot \sqrt{\tau} r / \|d\|  = \sqrt{\tau} r  \|d\| =  \sqrt{\tau} r \cdot \dist(\bar{s},K_d) \ ,
		\end{aligned}
	\end{equation} where the equality $(\hat s+d)^\top d=d^\top d$ follows from $\hat s^\top d=0$.  It therefore follows that $\dist(\bar{s},K_d)  \le \frac{1}{\sqrt{\tau}} \cdot \rho(r;\bar{z}) $, which proves item (\textit{\ref{item_gap2}}.).

	Lastly, we examine the duality gap $\gap(\bar{x},\bar{s}) = c^\top \bar{x} - b^\top \bar{y}$, and we consider two cases, namely $\bar z = 0$ and $\bar{z} \ne 0$.  If $\bar{z} = 0$, then $\gap(\bar{x},\bar{s}) = c^\top \bar{x} - b^\top \bar{y}=0$, which satisfies the duality gap bound trivially.  If $\bar{z} \neq 0$, then define $z := \bar{z} - \min\{\frac{r}{\|\bar{z}\|_M},1\}\bar{z}$, which satisfies $\|z - \bar{z}\|_M \le r$. Substituting this value of $z$ in \eqref{eq  lm: convergence of PHDG without restart 1} yields:
	\begin{equation}\label{eq  lm: convergence of PHDG without restart 7}
		r \rho(r;\bar{z}) \ge L(\bar{x},y) - L(x,\bar{y}) = \min\left\{\frac{r}{\|\bar{z}\|_M},1\right\}(c^\top \bar{x} - b^\top \bar{y}) = \min\left\{\frac{r}{\|\bar{z}\|_M},1\right\}(c^\top \bar{x} +q^\top \bar s -q_0) \ ,
	\end{equation}
	which after rearranging yields
	\begin{equation}\label{eq  lm: convergence of PHDG without restart 8}
		\gap(\bar{w})  = c^\top \bar{x} +q^\top \bar s -q_0 \le   \max\{ r, \|\bar{z}\|_M\} \rho(r;\bar{z}) \ .
	\end{equation}
	This proves the desired bound in item (\textit{\ref{item_gap3}}.).\Halmos\endproof

\section{Step size ratio and equivalent reweighting on problem instances}\label{app:stepsizes_weighted_instances}

In implementations of PDHG, the ratio between the primal and dual step-sizes is
often treated as a tunable parameter. This ratio changes the relative scale of the
primal and dual updates, and it can substantially affect practical performance
\cite{applegate2023faster,applegate2021practical}. The purpose of this section
is to explain why, for a fixed step-size product, changing the primal-dual step-size
ratio is equivalent to applying the original step-sizes to a weighted formulation of
the same problem. Such equivalence has already been given for PDHG in the analysis of \cite{applegate2023faster}.  We state it in our notation and extend it to rPDHG.

Fix any step-size pair $(\tau,\sigma)$ satisfying \eqref{eq:general_stepsize}. For any $\theta>0$, define
\begin{equation}\label{eq:scaled-step-sizes}
        \tau_\theta := \frac{1}{\theta}\tau,
        \qquad
        \sigma_\theta := \theta\sigma .
\end{equation}
Then $\tau_\theta\sigma_\theta=\tau\sigma$, so the step-sizes satisfy \eqref{eq:general_stepsize}. Consider the following weighted instance:
\Equationvalidatefalse
\begin{equation}\label{pro:scaled-problem}
        \min_{\tilde x\in \mathbb R^n}
        \tilde c^{\top}\tilde x
        \quad
        \text{s.t.}
        \quad
        A\tilde x=\tilde b,\quad
        \tilde x\in K_p ,
        \tag{$\widetilde{\mathrm P}_\theta$}
\end{equation}
\Equationvalidatetrue%
where $\tilde c := \frac{1}{\sqrt{\theta}}c$ and	
        $\tilde b := \sqrt{\theta}b$.
Thus the weighted instance differs from the original instance only through the scaling of the
objective vector and the right-hand side. The constraint matrix $A$ and the cone $K_p$ are
unchanged. Moreover, because $K_p$ is a cone, $x$ is feasible for \eqref{pro: general primal clp} if and only if
$\tilde x=\sqrt{\theta}x$ is feasible for \eqref{pro:scaled-problem} and the
corresponding objective values are the same:
$\tilde c^{\top}\tilde x = c^{\top}x$.
The Lagrangian of \eqref{pro:scaled-problem} is
\begin{equation}\label{eq:scaled-Lagrangian}
        \widetilde L(\tilde x,\tilde y)
        :=
        \tilde c^{\top}\tilde x
        +
        \tilde b^{\top}\tilde y
        -
        \tilde x^{\top}A^{\top}\tilde y
        =
        L\left(\frac{\tilde x}{\sqrt{\theta}},
        \sqrt{\theta}\tilde y\right).
\end{equation}
Define the linear map $S_\theta: \mathbb R^{m+n} \to \mathbb R^{m+n}$, denoted in matrix form as:
\begin{equation}\label{eq:scaled-map}
        S_\theta :=
        \begin{pmatrix}
        \sqrt{\theta}I_n & 0 \\
        0 & \frac{1}{\sqrt{\theta}}I_m
        \end{pmatrix}.
\end{equation}

\begin{proposition}[Step sizes and weighted instances]
\label{prop:stepsizes_weighted_instances}
Let $\{z^k\}_{k\ge 0}$ be the PDHG iterates for the original instance \eqref{pro: general primal clp} with
step-sizes $(\tau_\theta,\sigma_\theta)$, and let
$\{\tilde z^k\}_{k\ge 0}$ be the PDHG iterates for the weighted instance
\eqref{pro:scaled-problem} with step-sizes $(\tau,\sigma)$. If
$\tilde z^0=S_\theta z^0$, then $\tilde z^k=S_\theta z^k$ for all $k\ge 0$. 

Furthermore, suppose that Algorithm \ref{alg: PDHG with restarts} is run on both instances using the
$\beta$-restart condition. If the initial points satisfy $\tilde z^{0,0}=S_\theta z^{0,0}$, then the
entire restarted trajectories coincide under the same transformation. In particular, $ \tilde z^{n,k}=S_\theta z^{n,k}$, $\overline{\tilde z}^{\,n,k}=S_\theta \bar z^{n,k}$, and $\tilde z^{n,0}=S_\theta z^{n,0}$ for all $n\ge 0$ and $k\ge 0$.
\end{proposition}

\proof{Proof.}
The equivalence of the non-restarted PDHG trajectories follows from the standard rescaling
argument for PDHG with different primal and dual step-sizes; see
\cite[Footnote~5]{applegate2023faster}. In the present notation, applying that result gives
$\tilde z^k=S_\theta z^k$ for all $k\ge 0$. 
Since averaging commutes with the linear map $S_\theta$, it also follows that $\overline{\tilde z}^{\,k}=S_\theta \bar z^k$ for the average iterates $\overline{\tilde z}$ and $\bar z$.

It remains to check that the restart condition is invariant under the same transformation. 
Recall that $M :=
        \left(\begin{smallmatrix}
        \tau^{-1}I_n & A^{\top}\\
        A & \sigma^{-1}I_m
        \end{smallmatrix}\right)$ in \eqref{robsummer} and let $M_\theta :=
        \left(\begin{smallmatrix}
        \tau_\theta^{-1}I_n & A^{\top}\\
        A & \sigma_\theta^{-1}I_m
        \end{smallmatrix}\right)$ be the corresponding matrix for the step-size \eqref{eq:scaled-step-sizes}. Using \eqref{eq:scaled-map}, a direct calculation gives
        $M_\theta = S_\theta^{\top} M S_\theta$.
Hence, for any $z,z'$, it holds that
\begin{equation}\label{eq:scaled-norm}
        \|z-z'\|_{M_\theta}
        =
        \|S_\theta z-S_\theta z'\|_{M}.
\end{equation}
Furthermore, for $\tilde z=S_\theta z$ and
$\widehat{\tilde z}=S_\theta\widehat z$, it follows from
\eqref{eq:scaled-Lagrangian} that
\begin{equation}\label{eq:scaled-lagragian}
        \widetilde L(\tilde x,\widehat{\tilde y})
        -
        \widetilde L(\widehat{\tilde x},\tilde y)
        =
        L(x,\widehat y)-L(\widehat x,y).
\end{equation}
Also, since $K_p$ is a cone, $S_\theta(K_p\times \mathbb R^m)=K_p\times \mathbb R^m$.
Therefore the balls used in the definition of the normalized duality gap are mapped
bijectively by $S_\theta$, and the normalized duality gaps agree:
\begin{equation}\label{eq:scaled-gap}
        \tilde \rho(r;S_\theta z)
        =
        \rho_\theta(r;z)
\end{equation} 
for all $r>0$,
where $\rho_\theta$ denotes the normalized duality gap for the original instance with matrix
$M_\theta$, and $\tilde \rho$ denotes the normalized duality gap for the weighted instance
with matrix $M$.

Therefore the
$\beta$-restart condition is satisfied in the original instance if and only if it is satisfied in
the weighted instance. The restart times coincide, and the equivalence of the restart points
then follows by induction over the outer iterations.
\Halmos\endproof

\section{Computing the normalized duality gap}\label{appendix:compute_rho}

We first show in Section \ref{subsec:compute_rho} a methodology for computing the normalized duality gap $\rho(r;z)$.  In Section \ref{subsec:approx_rho} we modify this methodology by changing the norm from the $M$-norm $\|\cdot\|_M$ to a different norm $\|\cdot\|_N$ which we call the $N$-norm, and whose associated normalized duality gap is denoted by $\rho^N(r;z)$. Throughout this appendix we assume that  $\sigma, \tau$ satisfy \eqref{eq:general_stepsize}, which then implies that $M \succeq 0$ in \eqref{robsummer}. The strategy presented herein for computing and approximating $\rho(r;z)$ is a generalization of the method developed by \cite{applegate2023faster} for the particular case of LP.

\subsection{Computing the normalized duality gap}\label{subsec:compute_rho}
Throughout this subsection we suppose that $\tau$ and $\sigma$ satisfy \eqref{eq:general_stepsize}, so that $M\succeq0$. Computing $\rho(r;z)$ for  $z\in\bar{K}:= K_p \times \mathbb{R}^m$ and $r > 0$ is basically equivalent to solving the following convex optimization problem:
\begin{equation}\label{pro:compute_rho}
	\max_{ \hat{z} = (\hat{x},\hat{y}): \ \hat{x}\in K_p, \|\hat{z} -z\|_M \le r }\left[
		L(x,\hat{y}) - L(\hat{x},y)
		\right] = \left(
	\begin{array}{ll}
			\max_{\hat{z}} \       & h^\top (\hat{z} - z)                           \\
			\operatorname{s.t.} \  & \hat{z}\in \bar{K}, \ \|z-\hat{z}\|_M^2\le r^2
		\end{array}
	\right)
\end{equation}
in which $z = (x,y)$ and $h = \begin{pmatrix}
	h_1 \\
	h_2
\end{pmatrix}:=
\begin{pmatrix}
	A^\top y - c \\
	b - A x
\end{pmatrix}\in\mathbb{R}^{n+m}$. Suppose that $\hat{z}^\star$ is an optimal solution of \eqref{pro:compute_rho}; then $\rho(r;z)$ is obtained by
$$
\rho(r;z) = \frac{h^\top \big(\hat{z}^\star - z\big)}{r}\ .
$$

We now show how to construct an optimal solution of \eqref{pro:compute_rho}. Consider the following parameterized optimization problem over the parameter $t \ge 0$, with optimal solution $z(t)$ :
\begin{equation}\label{eq:z_t}
	z(t):= \arg\max_{\tilde{z}\in \bar{K}}  \ 
	t\cdot h^\top (\tilde{z} - z) - \|\tilde{z} - z\|_M^2 \ .
\end{equation} 
This problem essentially replaces the norm constraint $\|\hat{z} - z\|_M^2 \le r^2$ in \eqref{pro:compute_rho} with a penalty term $-\tfrac{\|\hat{z} - z\|_M^2}{t}$ in the objective function. The following lemma shows that solving \eqref{pro:compute_rho} is essentially a root-finding problem of the univariate function $f(t):=\|z - z(t)\|_M - r$ defined on $t \in [0,\infty)$.

\begin{lemma}\label{lm:compute_rho}
	Suppose that $\tau$ and $\sigma$ satisfy \eqref{eq:general_stepsize}. Suppose $t^\star$ satisfies $t^\star>0$ and $f(t^\star) = 0$. Then $z(t^\star)$ is an optimal solution of \eqref{pro:compute_rho}. 
\end{lemma}
\noindent
Based on this lemma, we consider using the bisection method to find the root of $f(t)$ in the region $t \in (0,\infty)$. Notice that $f(0) < 0$.  We can compute $f(t)$ for an increasing sequence of values of $t$, for example $t_k := 2^k$ for $k=1, 2, \ldots$, until we obtain $k$ for which $f(t_k) \ge 0$. Let $K$ denote the first value of $k$ for which $f(t_k) \ge 0$. Then $f(t_{K-1})$ has a different sign than $f(t_{K})$, and so $[t_{K-1},t_K]$ contains a root of $f(t)$ which can be computed using the bisection method. In the special case that none of the $t_k$ satisfy $f(t_k)\ge 0$, the following lemma shows that $\frac{h^\top (z(t_k) - z)}{r}$ itself converges to $\rho(r;z)$ at a conveniently bounded rate.

\begin{lemma}\label{lm:compute_rho_2}  Suppose that $\tau$ and $\sigma$ satisfy \eqref{eq:general_stepsize}. If $t>0$ and $f(t) < 0$, then $\left|\frac{h^\top (z(t) - z)}{r} - \rho(r;z)\right|\le \frac{r}{t}$ . 
\end{lemma}

The main computational bottleneck of the above scheme is solving \eqref{eq:z_t}. If $M$ is full-rank (i.e., \eqref{eq:general_stepsize} holds strictly), the objective function of \eqref{eq:z_t} is strongly concave and smooth. Furthermore, let us presume that the task of computing the projection onto $\bar{K}$ under the Euclidean norm is reasonable (as it is for the nonnegative orthant and the cross-product of second-order cones for example).  Then projected gradient ascent and its accelerated versions can be applied to \eqref{eq:z_t} with linear convergence rates, see \cite{lan2020first}. Moreover, since the restart condition does not need to be checked frequently in practice, the cost of computing the normalized duality gap can be further reduced if not computed very often. 

\subsubsection{Proofs of Lemmas \ref{lm:compute_rho} and \ref{lm:compute_rho_2}}
We now proceed with the proofs of Lemmas \ref{lm:compute_rho} and \ref{lm:compute_rho_2}. We first recall the optimality conditions for problems \eqref{pro:compute_rho} and \eqref{eq:z_t}.
Since strong duality holds for \eqref{eq:z_t}, for each $t \ge 0$ the optimal solution $z(t)$ of \eqref{eq:z_t} must satisfy the following conditions:
	\begin{equation}\label{eq:z_t_opt}
	z(t) \in \bar{K}, \ s:= 2M z(t) - 2Mz - t\cdot h \in  \bar{K}^*, \text{ and  } (z(t))^\top s = 0 \ .
	\end{equation} 
Regarding problem \eqref{pro:compute_rho}, $\hat{z}^\star$ is an optimal solution of \eqref{pro:compute_rho} if there exists a scalar multiplier $\lambda^\star$ that together with $\hat{z}^\star$ satisfy the KKT optimality conditions:
	\begin{equation}\label{eq:rho_KKT}
		\text{Inclusions: } \hat{z}^\star \in \bar{K}, \ s^\star :=    2 \lambda^\star M\hat{z}^\star - 2 \lambda^\star Mz -h \in \bar{K}^*, \ \lambda^\star \ge 0,  \  \|z-\hat{z}^\star\|_M^2\le r^2,   \text{ and }
	\end{equation}
	\begin{equation}\label{eq:rho_KKT_1}
		\text{Complementarity: } \  (\hat{z}^\star)^\top s^\star = 0 , \ \text{ and }  \lambda^\star \cdot \left( r^2 -  \|z-\hat{z}^\star\|_M^2\right) = 0 \ . \ \ \ \ \ \ \ \ \ \ \ \ \ \ \ \ \ \ \ \ \ \ \ \ \ \ \ \ 
	\end{equation}

\proof{Proof of Lemma \ref{lm:compute_rho}.}
	If $t^\star>0$ and $\|z - z(t^\star)\|_M = r$, then it follows from \eqref{eq:z_t_opt} and $f(t^\star) = 0 $ that $z^\star  := z(t^\star)$ and $\lambda^\star := \frac{1}{t^\star}$ satisfy the optimality conditions \eqref{eq:rho_KKT} and \eqref{eq:rho_KKT_1}, and therefore $z(t^\star)$ is an optimal solution of \eqref{pro:compute_rho}.	
\Halmos\endproof

\proof{Proof of Lemma \ref{lm:compute_rho_2}.}	
	We first suppose that $M$ is positive definite, in which case using a standard Lagrangian construction one can derive the following dual problem of \eqref{pro:compute_rho}:
\begin{equation}\label{bertsekas}
	\min_{ \lambda \ge0, \ s \in \bar K^* } \frac{1}{4\lambda}\|s+h\|_{M^{-1}}^2 + z^\top s + \lambda r^2\end{equation}
Now define $\lambda:=\frac{1}{t}$ and $s:= 2\lambda M z(t) - 2\lambda Mz -  h$, whereby from \eqref{eq:z_t_opt} it follows that $(\lambda,s)$ is feasible for \eqref{bertsekas}.  Also, since $f(t)<0$ we have $z(t)$ is feasible for \eqref{pro:compute_rho} and the duality gap of this pair of primal and dual solutions works out to be exactly $\lambda(r^2 - \|z-z(t) \|_M^2) = \tfrac{r^2 - \|z-z(t) \|_M^2}{t} $ which can be verified by simple arithmetic manipulation. Let the optimal objective value of \eqref{pro:compute_rho}  be $g^\star$; then $|h^\top (z(t) - z)  - g^\star | \le \tfrac{r^2 - \|z-z(t) \|_M^2}{t}$, from which it follows that $\left|\frac{h^\top (z(t) - z)}{r} - \rho(r;z)\right| = \left|\frac{h^\top (z(t) - z)}{r} - \frac{g^\star}{r}\right| \le \tfrac{r^2 - \|z-z(t) \|_M^2}{rt} \le \tfrac{r^2}{rt} = \tfrac{r}{t} $. This proves the result for the case when $M$ is positive definite.  
	
If $M$ is not positive definite, then under the assumption that $\sigma, \tau$ satisfy \eqref{eq:general_stepsize} we have $M \succeq 0$ \eqref{robsummer}. In this case the dual problem of \eqref{pro:compute_rho} no longer has the very convenient expression \eqref{bertsekas}, but all of the properties of the proof in the previous paragraph follow nevertheless. \Halmos\endproof

Note that the proofs of Lemmas \ref{lm:compute_rho} and \ref{lm:compute_rho_2} are also valid if we replace $M$ by another positive semidefinite matrix $\tilde{M}$ to define a $\tilde M$-norm, and let us denote the normalized duality gap using the $\tilde M$-norm as $\rho^{\tilde{M}}(r;z)$. In Section \ref{subsec:approx_rho} we will show that with a proper choice of $\tilde{M}$, $\rho^{\tilde{M}}(r;z)$ will provide a good approximation of $\rho^M(r;z)$ but with significantly lower computational cost of solving \eqref{eq:z_t} .

\subsection{Approximating  the normalized duality gap}\label{subsec:approx_rho}

In Section \ref{subsec:compute_rho} we showed that computing $\rho(r;z)$ can be accomplished by parametrically solving the optimization problem \eqref{eq:z_t}, which is equivalent to a certain projection onto the cone $\bar{K}=K_p\times \mathbb{R}^m$ in the $M$-norm. Although PDHG is premised on the notion that Euclidean projections onto $\bar{K}$ are simple to compute (see Lines \ref{line:update_x} and \ref{line:update_y} of \textsc{OnePDHG} in Algorithm \ref{alg: one PDHG}), projections onto $\bar K$ under the $M$-norm might be significantly more difficult.  In this subsection we describe how to efficiently approximate $\rho(r;z)$ by working with a different matrix norm, namely the $N$-norm which was introduced in the proof of Lemma \ref{lm change of norm}, and for which the equivalent optimization problem \eqref{eq:z_t} works out to be a Euclidean projection onto $K_p$.

The $N$-norm is the matrix norm $\|z\|_N$ in which $N := \Big(\begin{smallmatrix}
		\frac{1}{\tau}I_n &                     \\
		                  & \frac{1}{\sigma}I_m
	\end{smallmatrix}\Big)$, which was introduced in the proof of Lemma \ref{lm change of norm}. Let $\rho^N(r;z)$ denote the corresponding normalized duality gap function in $N$-norm.  We now show that in the $N$-norm, solving $z(t)$ of \eqref{eq:z_t} is simply a Euclidean projection onto $K_p$. Because $\bar{K} = K_p \times \mathbb{R}^m$ and $\|z\|_N^2 = \frac{1}{\tau}\|x\|^2 + \frac{1}{\sigma} \|y\|^2$, \eqref{eq:z_t} can be separated into two independent problems:
\begin{equation}\label{eq:z_t_2}
\begin{aligned}
	z(t) = (x(t),y(t)) & = \Big(
	\arg\max_{\tilde{x} \in K_p} \ t\cdot h_1^\top \tilde{x} - \tfrac{1}{\tau}\|\tilde{x} - x\|^2,  
	\arg\max_{\tilde{y} \in \mathbb{R}^m} \ t\cdot h_2^\top \tilde{y} - \tfrac{1}{\sigma}\|\tilde{y} - y\|^2\Big) \\
	& = \left(P_{K_p}\left(x + \tfrac{t\tau}{2} \cdot h_1\right), y + \tfrac{t\sigma}{2}\cdot h_2 \right) \ .
\end{aligned}
\end{equation} 
Hence the primary computational cost of computing $z(t)$ in the $N$-norm is just the Euclidean projection onto $K_p$, which is no more of a computational burden than Line \ref{line:update_x} of \textsc{OnePDHG} in Algorithm \ref{alg: one PDHG}, and might be considerably easier than the $M$-norm projection onto $\bar{K}$.

Furthermore, the following proposition shows that $\rho^N(r;z)$ is equivalent to $\rho(r;z)$ up to a constant factor so long as the step-sizes are chosen a bit conservatively.
\begin{proposition}\label{pr:rhoNrho}
	If $\tau,\sigma$ satisfy \eqref{eq:general_stepsize} strictly, then for any $z\in \bar{K}$ and $r>0$ it holds that:
	\begin{equation}\label{eq:lm:rhoNrho}
		\frac{1}{\sqrt{2}}\cdot \rho^N\left( r;z\right) \le
		\rho(r;z) \le \frac{1}{\sqrt{1-\sqrt{\tau \sigma}\lambda_{\max }}}\cdot \rho^N\left( r;z\right)  \ .
	\end{equation}
\end{proposition}
\noindent
For example, if $\sqrt{\sigma\tau} = \tfrac{1}{2\lambda_{\max}}$, then \eqref{eq:lm:rhoNrho} becomes $\frac{1}{\sqrt{2}}\cdot \rho^N\left( r;z\right) \le  \rho(r;z) \le \sqrt{2}\cdot \rho^N\left( r;z\right)$. In practice we have used $\rho^N\left( r;z\right)$ instead of $\rho\left( r;z\right)$ to evaluate the restart condition, and it can be proven that a computational guarantee of rPDHG still holds. This technique has also been used in \cite{applegate2023faster,applegate2021practical}. 

The proof of Proposition \ref{pr:rhoNrho} uses the following two lemmas.
\begin{lemma}{\bf(Proposition 2.6 of \cite{xiong2023computational})}\label{lm:MNnorms}
	If $\tau,\sigma$ satisfy \eqref{eq:general_stepsize}, then for any $z\in\mathbb{R}^{n+m}$ it holds that $\sqrt{1-\sqrt{\tau \sigma} \lambda_{\max }} \cdot\|z\|_N \leq\|z\|_M \leq \sqrt{2}\|z\|_N$.
\end{lemma}
\begin{lemma}{\bf(Proposition 5 of \cite{applegate2023faster})}\label{lm:monotonicity_of_rho}
	For any $z\in\bar K$, $\rho^N(r;z)$ is monotonically nonincreasing in $r\in(0,\infty)$.
\end{lemma}
\proof{Proof of Proposition \ref{pr:rhoNrho}.}
	From Lemma \ref{lm:MNnorms} we have $\big\{\hat{z}:\|\hat{z} - z\|_N\le  \frac{r}{\sqrt{2}}
	\big\}\subseteq \left\{\hat{z}:\|\hat{z} - z\|_M\le r
	\right\} \subseteq \big\{\hat{z}:\|\hat{z} - z\|_N\le  \frac{r}{\sqrt{1-\sqrt{\tau \sigma}\lambda_{\max }}}
	\big\}$, which leads to
	\begin{equation}\label{eq:lm:rhoNrho 2}
		\textstyle
		\frac{r}{\sqrt{2}}\cdot \rho^N\left( \frac{r}{\sqrt{2}};z\right) \le
		r\cdot \rho(r;z) \le \frac{r}{\sqrt{1-\sqrt{\tau \sigma}\lambda_{\max }}}\cdot \rho^N\Big( \frac{r}{\sqrt{1-\sqrt{\tau \sigma}\lambda_{\max }}};z\Big)  \
	\end{equation}
	for any $r$ and $z$, because of the inclusion relationship between the feasible sets of the corresponding optimization problems. Finally, applying Lemma \ref{lm:monotonicity_of_rho} to \eqref{eq:lm:rhoNrho 2} yields \eqref{eq:lm:rhoNrho}.
\Halmos\endproof

\section{Proof of Lemma \ref{lm r_delta over delta}}\label{app:proof of lm r calW small}

Before presenting the proof we first discuss a related result.
For the general primal and dual problems \eqref{pro: general primal clp} and \eqref{pro: general dual clp on s}, \cite{freund2003primal} proves a geometric relationship between the primal and dual level sets of \eqref{pro: general primal clp} and \eqref{pro: general dual clp on s} which we now describe. Under Assumption \ref{assump:striclyfeasible} the problems \eqref{pro: general primal clp} and \eqref{pro: general dual clp on s} have a common optimal objective value $f^\star$.  Then for any $\eps,\delta\in\mathbb{R}_+$, define
\begin{equation}\label{eq:def_R_r}
	\bar{R}_\eps:=
	\left(\begin{array}{cc}
			\max\limits_{x}     & \|x\|                         \\
			\operatorname{s.t.} & x \in V_p \cap K_p            \\
			                    & c^\top x \le  f^\star + \eps
		\end{array}
	\right)
	\ \text{ and } \ \bar{r}_\delta:=  \left(\begin{array}{cc}
			\max\limits_{s}     & \max\limits_{r:B(s,r) \subseteq K_d}r \\
			\operatorname{s.t.} & s \in V_d \cap K_d                    \\
			                    &  - q^\top s + q_0 \ge f^\star - \delta
		\end{array}
	\right) \ .
\end{equation}
Recall the definition of the width $\width_K$ of cone $K$ in Definition \ref{def width of cone}, and then the product of $\bar{R}_\eps$ and $\bar{r}_\delta$ has both lower and upper bounds.
\begin{lemma}{\bf (Geometric relationship between primal and dual level sets, Theorem 3.2 of \cite{freund2003primal})}\label{lm:primal-dual-sublevelsets}
	Under Assumption \ref{assump:striclyfeasible}, for any $\eps,\delta\in\mathbb{R}_+$ it holds that
	\begin{equation}
		\width_{K_d} \cdot \min\{\delta, \eps\} \le  \bar{r}_\delta\bar{R}_\eps  \le \delta + \eps \ .
	\end{equation}
\end{lemma}
\noindent Lemma \ref{lm:primal-dual-sublevelsets} leads directly to the following corollary regarding the conic radius $r_\delta$ of the level set $\calW_\delta$:
\begin{corollary}\label{cor: r R in symmetric form}
	For any $\eps\ge0$ and $\delta>0$, it holds that
	\begin{equation}\label{eq:r R in symmetric form}
		\width_K \cdot \min\{\delta,\eps\} \le r_\delta \cdot \left(\max_{w\in\calW_\eps}\|w\|\right) \le \delta + \eps \ .
	\end{equation}
\end{corollary}
\proof{Proof.}
	The idea of the proof is to relate $r_\delta$ and $\max_{w\in\calW_\eps}\|w\|$ to the quantities $\bar{r}_\delta$ and $\bar{R}_\eps$ defined in \eqref{eq:def_R_r} for a certain pair of primal and dual problems. Since replacing the objective representative $c$ by another vector in $c+\mathcal{L}^{\bot}$ does not change the level sets $\calW_\delta$ or $\calW_\eps$, without loss of generality we presume that $c \in \mathcal{L}$. Recall here that $\mathcal{L}$ and $\mathcal{L}^{\bot}$ are the linear subspaces associated with $V_p$ and $V_d$. We can combine \eqref{pro: general primal clp} and \eqref{pro: general dual clp on s} into a single conic linear optimization problem of the following form:
	\begin{equation}\label{pro:primal-dual LP problem}
		\begin{aligned}
			\min_{w\in\mathbb{R}^{2n}} \  & z_0^\top w \quad
			\text{s.t.} \                 & w \in V = (\mathcal{L}\times\mathcal{L}^{\bot}) + w_0, \ w \in K = K_p \times K_d \ , 
		\end{aligned}
	\end{equation}
where $z_0 := (c,q) \in \mathcal{L}\times\mathcal{L}^{\bot}$ and $w_0 := (q,c)$, yielding $V = (\mathcal{L}\times\mathcal{L}^{\bot}) + w_0 = (q+\mathcal{L})\times(c+\mathcal{L}^{\bot}) = V_p \times V_d$.  Using a similar approach for deriving the dual problem with that in Section \ref{sec:get_general_clp_dual}, it is known that its symmetric dual problem (see also Section 3.1 of \cite{renegar2001mathematical} for details) is the following problem:
	\begin{equation}\label{pro:dual of the primal-dual LP problem}
		\begin{aligned}
			\min_{z\in\mathbb{R}^{2n}} \  & w_0^\top z \quad
			\text{s.t.} \                 & z \in \tilde{V} := (\mathcal{L}^{\bot}\times\mathcal{L}) + z_0 \ , \ z \in K^* = K_d \times K_p \ . 
		\end{aligned}
	\end{equation}
	Note that $\mathcal{L}$ and $\mathcal{L}^{\bot}$ are orthogonal complementary subspaces of $\mathbb{R}^n$, and so $\mathcal{L}^{\bot}\times\mathcal{L}$ is the orthogonal complement of $\mathcal{L}\times\mathcal{L}^{\bot}$ in $\mathbb{R}^{2n}$. Moreover, because $K_p$ and $K_d$ are dual cones of each other, then $K^* = K_p^* \times K_d^* = K_d \times K_p$. Therefore, now we can see the feasible sets of \eqref{pro:primal-dual LP problem} and \eqref{pro:dual of the primal-dual LP problem} are related by simply exchanging the order of the $n$-dimensional variable components, namely from $w=(x,s)$ to $z=(s,x)$.  And also the objective vectors $z_0$ and $w_0$ are similarly related by exchanging the order of $c$ and $q$.

It thus follows that the quantity $\bar{r}_\delta$ associated with \eqref{pro:primal-dual LP problem} is identical to the quantity $r_\delta$ of $\calW_\delta$, and quantity $\bar{R}_\eps$ associated with \eqref{pro:dual of the primal-dual LP problem} is identical to $\max_{w\in\calW_\eps}\|w\|$. Furthermore, because $K = K_p\times K_d$ and $K^* = K_d \times K_p$, then $\tau_K = \tau_{K^*}$.
	Therefore, directly applying Lemma \ref{lm:primal-dual-sublevelsets} on \eqref{pro:primal-dual LP problem} and  \eqref{pro:dual of the primal-dual LP problem}  yields \eqref{eq:r R in symmetric form}.
\Halmos\endproof

Furthermore, the following monotonicity result is presented in Remark 2.1 of \cite{freund2003primal}.
\begin{lemma}{\bf (Monotonicity of $\bar{r}_\delta/\delta$)}\label{lm:monotonicity of delta r}
	For any $\delta' > \delta >0$, it holds that $\frac{\bar{r}_{\delta} }{\delta}\ge \frac{\bar{r}_{\delta'}}{\delta'}$.
\end{lemma}

Using Corollary \ref{cor: r R in symmetric form} and Lemma \ref{lm:monotonicity of delta r}, we now prove Lemma \ref{lm r_delta over delta}.
\proof{Proof of Lemma \ref{lm r_delta over delta}.}
	We first prove the second inequality of \eqref{eq of lm r calW small}.
	According to the second inequality of \eqref{eq:r R in symmetric form} with $\eps = 0$, it holds that $r_\delta \cdot \max_{w \in \calW^\star}\|w\| \le \delta$ for any $\delta > 0$, which directly implies the second inequality of \eqref{eq of lm r calW small}.

As for the first inequality of \eqref{eq of lm r calW small}, the first inequality of \eqref{eq:r R in symmetric form} yields $r_\delta \cdot \max_{w \in \calW_\eps}\|w\| \ge \width_K \cdot \min\{\delta,\eps\}$ for any $\delta ,\eps \in \mathbb{R}_{++}$. Taking $\eps = \delta$ yields $\frac{r_\delta}{\delta}\ge \frac{\width_K}{\max_{w \in \calW_\delta}\|w\|} $ for any $\delta > 0$. And using Lemma \ref{lm:monotonicity of delta r} we have $\sup_{\delta>0}\frac{r_\delta}{\delta} = \lim_{\delta\searrow 0}\frac{r_\delta}{\delta} \ge \lim_{\delta \searrow 0} \frac{\width_K}{\max_{w \in \calW_\delta}\|w\|} = \frac{\width_K}{\max_{w\in \calW^\star}\|w\|}$, which proves the first inequality of \eqref{eq of lm r calW small}.
\Halmos\endproof

\end{APPENDICES}

\ACKNOWLEDGMENT{We thank Haihao Lu, Kim-Chuan Toh, and Haoyue Wang for useful discussions and feedback on this work. Part of the work by Z. Xiong was performed at the Massachusetts Institute of Technology supported by AFOSR Grant No. FA9550-22-1-0356. Another part of the work by Z. Xiong was performed at Georgia Institute of Technology supported by ONR Award \#N00014-25-1-2088. The work by R. Freund was supported by AFOSR Grant No. FA9550-22-1-0356. The authors declare that they have no relevant financial or non-financial interests to disclose.}



\bibliographystyle{informs2014} 
\bibliography{reference} 


\end{document}